\documentclass{elsarticle}
\usepackage{srcltx}
\usepackage{eurosym}
\usepackage{mathtools}
\usepackage{amsmath}
\usepackage{amsfonts}
\usepackage{amssymb}
\usepackage{amsthm}
\usepackage{graphicx}
\usepackage{mathrsfs}
\usepackage{xcolor}
\usepackage{color}
\usepackage{exscale}
\usepackage{latexsym}
\usepackage{array,longtable}
\usepackage{textcomp}
\usepackage[plainpages=true,pdfpagelabels,hypertexnames=true,colorlinks=true,pdfstartview=FitV,linkcolor=blue,citecolor=blue,urlcolor=blue,unicode]{hyperref}
\PassOptionsToPackage{unicode}{hyperref}
\PassOptionsToPackage{naturalnames}{hyperref}
\usepackage{enumerate}
\usepackage[shortlabels]{enumitem}
\usepackage{bookmark}
\usepackage{wasysym}
\usepackage{esint}
\def\Yint#1{\mathchoice
    {\YYint\displaystyle\textstyle{#1}}%
    {\YYint\textstyle\scriptstyle{#1}}%
    {\YYint\scriptstyle\scriptscriptstyle{#1}}%
    {\YYint\scriptscriptstyle\scriptscriptstyle{#1}}%
      \!\iint}
\def\YYint#1#2#3{{\setbox0=\hbox{$#1{#2#3}{\iint}$}
    \vcenter{\hbox{$#2#3$}}\kern-.50\wd0}}
\def\longdash{-\mkern-9.5mu-} 
\def\tiltlongdash{\rotatebox[origin=c]{18}{$\longdash$}}
\def\fiint{\Yint\tiltlongdash}

\def\Xint#1{\mathchoice
    {\XXint\displaystyle\textstyle{#1}}%
    {\XXint\textstyle\scriptstyle{#1}}%
    {\XXint\scriptstyle\scriptscriptstyle{#1}}%
    {\XXint\scriptscriptstyle\scriptscriptstyle{#1}}%
      \!\int}
\def\XXint#1#2#3{{\setbox0=\hbox{$#1{#2#3}{\int}$}
    \vcenter{\hbox{$#2#3$}}\kern-.50\wd0}}
\def\hlongdash{-\mkern-13.5mu-}
\def\tilthlongdash{\rotatebox[origin=c]{18}{$\hlongdash$}}
\def\hint{\Xint\tilthlongdash}
\makeatletter
\def\namedlabel#1#2{\begingroup
  \def\@currentlabel{#2}%
  \label{#1}\endgroup
}
\makeatother
\makeatletter
\def\ps@pprintTitle{%
\let\@oddhead\@empty
\let\@evenhead\@empty
\def\@oddfoot{}%
\let\@evenfoot\@oddfoot}
\makeatother
\makeatletter
\newcommand{\rmh}[1]{\mathpalette{\raisem@th{#1}}}
\newcommand{\raisem@th}[3]{\hspace*{-1pt}\raisebox{#1}{$#2#3$}}
\makeatother

\newcommand{\lsb}[2]{#1_{\rmh{-3pt}{#2}}}
\newcommand{\lsbo}[2]{#1_{\rmh{-1pt}{#2}}}
\newcommand{\lsbt}[2]{#1_{\rmh{-2pt}{#2}}}

\newcommand{\descitem}[1]{\item[{(#1)}] \label{#1}}
\newcommand{\descref}[1]{\hyperref[#1]{\textnormal{\textcolor{black}{(}\textcolor{blue}{\bf #1}\textcolor{black}{)}}}}

\newcommand{\lomt}{\lsb{\chi}{\Om_T}}
\newcommand{\qomt}{16Q \cap \Om_T}
\newcommand{\omtz}{\Om \times \{t=0\}}

\newcommand{\steplabel}[2]{\item[{#1}] \label{#2}}
\newcommand{\stepref}[2]{\hyperref[#1]{\textnormal{\textcolor{blue}{\bf #2}}}}

\usepackage[ddmmyyyy]{datetime}

\usepackage[margin=2.5cm]{geometry}

\usepackage[capitalize,nameinlink]{cleveref}
\crefname{section}{Section}{Sections}
\crefname{subsection}{Subsection}{Subsections}
\crefname{condition}{Condition}{Conditions}
\crefname{hypothesis}{Hypothesis}{Conditions}
\crefname{assumption}{Assumption}{Assumptions}
\crefname{lemma}{Lemma}{Lemmas}

\crefformat{equation}{\textup{#2(#1)#3}}
\crefrangeformat{equation}{\textup{#3(#1)#4--#5(#2)#6}}
\crefmultiformat{equation}{\textup{#2(#1)#3}}{ and \textup{#2(#1)#3}}
{, \textup{#2(#1)#3}}{, and \textup{#2(#1)#3}}
\crefrangemultiformat{equation}{\textup{#3(#1)#4--#5(#2)#6}}%
{ and \textup{#3(#1)#4--#5(#2)#6}}{, \textup{#3(#1)#4--#5(#2)#6}}%
{, and \textup{#3(#1)#4--#5(#2)#6}}

\Crefformat{equation}{#2Equation~\textup{(#1)}#3}
\Crefrangeformat{equation}{Equations~\textup{#3(#1)#4--#5(#2)#6}}
\Crefmultiformat{equation}{Equations~\textup{#2(#1)#3}}{ and \textup{#2(#1)#3}}
{, \textup{#2(#1)#3}}{, and \textup{#2(#1)#3}}
\Crefrangemultiformat{equation}{Equations~\textup{#3(#1)#4--#5(#2)#6}}%
{ and \textup{#3(#1)#4--#5(#2)#6}}{, \textup{#3(#1)#4--#5(#2)#6}}%
{, and \textup{#3(#1)#4--#5(#2)#6}}

\crefdefaultlabelformat{#2\textup{#1}#3}

\makeatletter
\newcommand{\vo}{\vec{o}\@ifnextchar{^}{\,}{}}
\makeatother

\numberwithin{equation}{section}
\everymath{\displaystyle}

\newtheorem{theorem}{Theorem}[section]
\newtheorem{lemma}[theorem]{Lemma}
\newtheorem{corollary}[theorem]{Corollary}

	\newtheorem{assumption}[theorem]{Assumption}
\newtheorem{definition}[theorem]{Definition}
\newtheorem{remark}[theorem]{Remark}        
\numberwithin{equation}{section}

\def\al{\alpha}

\def\be{\beta}
\def\ga{\gamma}
\def\de{\delta}
\def\ve{\varepsilon}
\def\vt{\vartheta}

\def\th{\theta}
\def\la{\lambda}
\def\ep{\epsilon}
\def\ka{\kappa}

\def\pa{\partial}

\def\Th{\Theta}

\def\Om{\Omega}
\def\La{\Lambda}
\def\aa{\mathcal{A}}

\newcommand{\tO}{\tilde{\Om}}
\newcommand{\tQ}{\tilde{Q}}

\newcommand{\tm}{I}

\newcommand{\mm}{\mathcal{M}}
\newcommand{\mmn}[1]{\mathcal{P}^{-1,#1}}

\newcommand{\RR}{\mathbb{R}}

\newcommand{\NN}{\mathbb{N}}

\newcommand{\redref}[2]{\texorpdfstring{\protect\hyperlink{#1}{\textcolor{black}{(}\textcolor{red}{#2}\textcolor{black}{)}}}{}}
\newcommand{\redlabel}[2]{\hypertarget{#1}{\textcolor{black}{(}\textcolor{red}{#2}\textcolor{black}{)}}}

\newcommand{\iprod}[2]{\left\langle #1 ,  #2\right\rangle}
\newcommand{\abs}[1]{\left| #1 \right|}
\newcommand{\lbr}[1][(]{\left#1}
\newcommand{\rbr}[1][)]{\right#1}

\newcommand{\cpt}[1][p]{\capacity_{1,#1}}

\newcommand{\txt}[1]{\qquad \text{#1} \quad}

\newcommand{\zv}{\zeta_{\ve}}

\DeclareMathOperator{\dv}{div}
\DeclareMathOperator{\capacity}{cap}
\DeclareMathOperator{\diam}{diam}

\DeclareMathOperator{\spt}{spt}

\DeclareMathOperator{\loc}{loc}

\newcommand{\tp}{\tilde{p}}

\newcommand{\tk}{\tilde{k}}
\newcommand{\tz}{\tilde{z}}

\newcommand{\mfz}{\mathfrak{z}}

\newcommand{\mch}{\mathcal{H}}

\newcommand{\mcc}{\mathcal{C}}

\newcommand{\avgs}[2]{\lsbo{\lbr {#1} \rbr}{#2}}

\newcommand{\ddt}[1]{\frac{d#1}{dt}}
\newcommand{\dds}[1]{\frac{d#1}{ds}}

\newcommand{\vlh}{\lsbt{v}{\la,h}}
\newcommand{\vl}{\lsbt{v}{\la}}

\newcommand{\vh}{\lsbt{v}{h}}

\newcommand{\elam}{\lsbo{E}{\lambda}}

\newcommand{\wwspace}{L^{p-\be}(0,T; W^{1,p-\be}(\Om))}

\newcommand{\lamot}{\La_0,\La_1}

\newcommand{\tTh}{{\Upsilon}}

\newcommand{\pbo}{\frac{p-\be}{1-\be}}
\newcommand{\pbp}{\frac{p-\be}{p-1}}

\begin{document}

\begin{frontmatter}

\title{Partial existence result for Homogeneous Quasilinear parabolic problems beyond the duality pairing}

\author[myaddress]{Karthik Adimurthi\corref{mycorrespondingauthor}\tnoteref{thanksfirstauthor}}
\cortext[mycorrespondingauthor]{Corresponding author}
\ead{karthikaditi@gmail.com and kadimurthi@snu.ac.kr}
\tnotetext[thanksfirstauthor]{Supported by the National Research Foundation of Korea grant NRF-2015R1A2A1A15053024.}

\author[myaddress,myaddresstwo]{Sun-Sig Byun\tnoteref{thankssecondauthor}}
\ead{byun@snu.ac.kr}
\tnotetext[thankssecondauthor]{Supported by the National Research Foundation of Korea grant  NRF-2015R1A4A1041675. }

\author[myaddress]{Wontae Kim\tnoteref{thanksthirdauthor}}
\ead{m20258@snu.ac.kr}
\tnotetext[thanksthirdauthor]{}

\address[myaddress]{Department of Mathematical Sciences, Seoul National University, Seoul 08826, Korea.}
\address[myaddresstwo]{Research Institute of Mathematics, Seoul National University, Seoul 08826, Korea.}

\begin{abstract}
In this paper, we study the existence of distributional solutions solving \cref{main-3}
on a bounded domain $\Om$ satisfying a uniform capacity density condition where the nonlinear structure $\aa(x,t,\nabla u)$ is modelled after the standard parabolic $p$-Laplace operator. In this regard, we need to prove a priori estimates for the gradient of the solution below the natural exponent and a higher integrability result for very weak solutions at the initial boundary.  The elliptic counterpart to these two estimates are fairly well developed over the past few decades, but no analogous theory exists in the quasilinear parabolic setting. 

Two  important features of the estimates proved here are that they are non-perturbative in nature and we are able to take non-zero boundary data. \emph{As a consequence, our estimates are new even for the heat equation on bounded domains.} This partial existence result is a nontrivial extension of the existence theory of very weak solutions from the elliptic setting to the quasilinear parabolic setting. Even though we only prove partial existence result, nevertheless we establish the necessary framework that when proved would lead to  obtaining the full result for the homogeneous problem.

%
\end{abstract}


\end{frontmatter}

\tableofcontents

\section{Introduction}
\label{section1}

In this paper, we are mainly interested in obtaining a priori estimates for \cref{main-1} and  \cref{main-2} which will then be used to obtain the existence of \emph{very weak solutions} to equations of the form  \cref{main-3}. Here the nonlinearity  $\aa(x,t,\zeta)$  is modelled after the well known $p$-Laplace operator and $\Om \subset \RR^n$ denotes a bounded domain with potentially nonsmooth boundary.  The elliptic analogue of the estimates and existence theory studied in this paper are quite well understood. The first result for the elliptic homogeneous problem was proved in \cite{iwaniec1994weak} with the sharp version of the a priori estimate and existence obtained recently in \cite{AP1}. The parabolic counterpart to these questions have remained open for a long time and in this paper, we obtain some partial answers in this direction.

Weak solutions to \eqref{main-1} are in the space $u \in L^2(0,T; L^2(\Om)) \cap L^p(0,T; W_0^{1,p}(\Om))$ which allows one to use $u$ as a test function. But from the definition of weak solution (see \cref{def_weak_solution}), we see that the expression makes sense if we only assume $u \in L^2(0,T; L^2(\Om)) \cap L^{s}(0,T; W_0^{1,s}(\Om))$ for some $s > \max\{p-1,1\}$. But under this milder notion of solution called \emph{very weak solution}, we lose the ability to use $u$ as a test function.  This difficulty was overcome in \cite{KL} where the method of Lipschitz truncation was developed to construct a suitable test function, which was then used to obtain interior higher integrability result below the natural exponent. This technique was subsequently extended in \cite{AdiByun2} to obtain analogous estimates upto the lateral  boundary with zero boundary data. The extension of the higher integrability result for very weak solutions at the initial boundary seems to be nontrivial, mainly because of a lack of certain suitable cancellations and the lack of time derivative for the solutions.

In order to obtain the existence of \emph{very weak solution} to \cref{main-3}, there are three main ingredients:  first we need to obtain  suitable a priori estimates that control the gradient of the solution in terms of the boundary data and  secondly, we need to obtain  higher integrability result for very weak solutions at the initial boundary and finally, we can combine the previous two estimates with standard compactness arguments to prove the existence result. In the subsequent three subsections, we shall discuss the main questions and the new tools that are going to be used in this paper.

In order to prove the main results of this paper, we need to develop new ideas which include bounds for a suitably modified maximal function on negative Sobolev spaces, careful construction of the Lipschitz truncation function at the initial boundary which preserves the necessary boundary values and the ability to handle non-zero boundary data which are complicated by the lack of a time derivative.

\subsection{Discussion about the a priori estimate}
In this subsection, we shall discuss the  a priori estimates for \emph{very weak solutions} of 
\begin{equation}
\label{main-1}
\left\{
\begin{array}{ll}
u_t - \dv  \aa(x,t,\nabla u) = 0 & \text{on} \ \Om \times (0,T), \\
u = w & \text{on} \ \pa_p(\Om \times (0,T)).\\
\end{array}
\right.
\end{equation}
Since the notion of a solution does not a priori have any regularity in the time variable, the term $u_t$ (and hence $w_t)$ is to be understood in the distributional sense.  This complicates the construction of Lipschitz truncation due to \cite{KL} (see the boundary extension in \cite{AdiByun2}) which must be able to handle the boundary data as well as the distributional time derivative of $u$ and $w$.  Since we are interested in controlling $|\nabla u|$ in terms of $|\nabla w|$ in suitable norms (whose exponents are below $p$), we are invariably forced to estimate $w_t$ in suitable negative Sobolev spaces (see \cref{main_theorem_1}).

We employ two main ideas to achieve this goal; the first is that the extension is very carefully obtained to preserve boundary data (see also \cite{adimurthi2018gradient} for more on this) and secondly, we define a new Maximal function (based on \cite{bernicot2009maximal}) on appropriate negative Sobolev spaces. Combining these two ideas, we can develop the method  Lipschitz truncation to handle non-zero boundary data and obtain a priori estimates below the natural exponent. \emph{These a priori estimates are new even for the heat equation on bounded domains}.

In the process, we encounter a natural difficulty that arises due to the definition of the \emph{very weak solution}.  Since we make use of Steklov averages to define \emph{very weak solutions} of \eqref{main-2}, we are also forced to understand the relation between $\ddt{[w]_h}$ and $\ddt{w}$. It is well known that $\ddt{[w]_h}$ is a function whereas $\ddt{w}$ is a distribution and in general, the following does not hold:
\[
\ddt{[w]_h} \nrightarrow \ddt{w} \txt{as} h \searrow 0.
\]
To overcome this difficulty, we  make an additional assumption regarding $\ddt{w}$ (see \cref{imp_rmk}) which enables us to obtain \cref{lemma_crucial_2_app}. While this assumption seems restrictive, in applications, the boundary data $w$ generally solves an analogous parabolic equation and thus, this difficulty goes away.

\subsection{Discussion about the higher integrability estimate}
In this subsection, we shall discuss the problem of obtaining  higher integrability of \emph{very weak solutions} to
\begin{equation}
\label{main-2}
\left\{
\begin{array}{ll}
u_t - \dv  \aa(x,t,\nabla u) = 0 & \text{on} \ \Om \times (0,T), \\
u=w & \text{on} \ \Om \times \{t=0\}.
\end{array}
\right.
\end{equation}
In the interior, the result was proved in the seminal paper of \cite{KL} which laid the Foundation for the method of parabolic Lipschitz truncation. This was subsequently extended to the lateral boundary with zero boundary data in \cite{AdiByun2}. Following the idea of the proof in \cite{KL,AdiByun2}, it becomes apparent that the method fails while handling non-zero boundary data at the lateral boundary or at the initial boundary, the main obstruction being a lack of time derivative in general for the boundary data.

The second theorem we prove is the higher integrability for \emph{very weak solutions} to \cref{main-2} at the initial boundary with non-zero initial data. There are several comments to be made regarding this result and to understand the remarks, let us look at the higher integrability for \emph{weak solutions} proved in \cite{Mik1,Par1} at the initial boundary. In those papers, they were able to take a suitable test function which exhibited a very crucial cancellation (see \cite[Equation (5.6)]{Par}. In order to prove the higher integrability for \emph{very weak solutions}, this cancellation needs to be preserved for the test function constructed through the method of Lipschitz truncation. In the interior case or at the lateral boundary with zero boundary data, such a cancellation trivially holds whereas it fails in general when handing non-zero boundary data at both the initial boundary and the lateral boundary.

In main theorem of this subsection \cref{main_theorem_2}, we obtain the higher integrability for \emph{very weak solutions} at the initial boundary with non-zero initial data noting that the same estimate holds at the lateral boundary with non-zero boundary data and is optimal.  In our construction, we are unable to recover the crucial cancellation of \cite[Equation (5.6)]{Par} and hence at the initial boundary with non-zero data, our result is not optimal. Nevertheless, at least for zero initial data, the estimate is sharp. It would be interesting to obtain a modified construction of the Lipschitz test function that can preserve the necessary cancellations, which would then provide an optimal result at the initial boundary.

Let us now highlight some of the new ideas that are developed to obtain the result. Firstly, due to the presence of Steklov average in the \cref{very_weak_solution}, the initial boundary value is not always preserved. Secondly, the problem with the lack of time derivative for the initial data is still present and to overcome this, we use the ideas developed in the proof of \cref{main_theorem_1}. It is interesting to note the unusual choice of the function (see \cref{def-v-ini}) used to perform the Lipschitz truncation upon. In particular, we handle the initial boundary problem as a problem at the lateral boundary, which leads to difficulties while applying the standard parabolic Poincar\'e's inequality, and this is where we crucially exploit the fact that the extension constructed in \cref{lipschitz-extension-ini} is zero on the \emph{bad} part of the initial boundary. This is a very subtle technicality which originated from \cite{AdiByun2}. Once we have the modified construction, along with the bounds from \cref{max_neg_bound}, we can obtain a time localized version of the Caccioppoli type inequality followed by a  reverse H\"older type inequality. Finally applying the parabolic Gehring's lemma gives the desired higher integrability at the initial boundary.  \emph{This result is new even for the heat equation on bounded domains.}

\subsection{Discussion about existence}
In this subsection, we shall discuss the existence of \emph{very weak solutions} to 
\begin{equation}
\label{main-3}
\left\{
\begin{array}{ll}
u_t - \dv  \aa(x,t,\nabla u) = 0 & \text{on} \ \Om \times (0,T), \\
u = u_0 & \text{on} \ \pa \Om \times (0,T),\\
u = 0 & \text{on} \ \Om \times \{t=0\}.
\end{array}
\right.
\end{equation}
Before we explain the main result, let us discuss the main ideas behind the elliptic counterpart of the existence theory developed in \cite[Theorem 2]{iwaniec1994weak}. In that paper, the authors needed two main ingredients, first is an a priori estimate controlling the solution in terms of the boundary data and the second is an interior higher integrability result. Once both these estimates exist, then one can perform a standard approximation argument to get a sequence of weak solutions uniformly bounded in the right function spaces and then compactness methods  can be used to deduce the converge of the approximate solution to the desired very weak solution.

In the parabolic setting, we also follow the same strategy. The desired a priori estimate is obtained in \cref{main_theorem_1}, but we now need higher integrability for \emph{very weak solutions} in the interior (proved in \cite{KL} and at the initial boundary (see \cref{main_corollary_2}). Note that we can only use zero initial data to obtain the existence mainly because our higher integrability result at the initial boundary is not sharp for non-zero initial data.

We now  need to first consider a suitable approximating sequence of solutions and show that this sequence converges to the desired \emph{very weak solution}. To construct such a approximating sequence of solutions, the standard idea is to smoothen the given data (say by mollifying), but unfortunately, since $\ddt{u_0}$ is only a distribution, mollifying does not work. To overcome this difficulty, we assume either $\ddt{u_0}$ is an $L^1$ function, or more generally, we assume the existence of an approximating sequence satisfying \cref{assumption1} and \cref{assumption2} (see \cref{rmk_3.8} for more about the necessity of assuming the existence of an approximating sequence).  Given such a sequence, we can then follow standard compactness arguments to obtain the desired \emph{very weak solutions} to \cref{main-3} which is \cref{main_existence}.

\subsection{Outline of the paper}

In \cref{section2}, we collect all the preliminary information along with structural assumptions regarding the nonlinearity and domain. In \cref{section3}, we describe the main theorems and in \cref{section4}, we shall recall and in some cases prove some well known lemmas that will be needed in the subsequent sections. In \cref{section5}, we will obtain the proof of the a priori estimate from \cref{main_theorem_1}, in \cref{section6}, we will obtain the proof of the higher integrability at the initial boundary as stated in \cref{main_theorem_2} and finally in \cref{section7}, we shall give the proof of the existence result from \cref{main_existence}.

\section{Preliminaries}
\label{section2}

\subsection{Variational \texorpdfstring{$p$}.-Capacity}
Let $1<p<\infty$, then the \emph{variational $p$-capacity} of a compact set $K \Subset \RR^n$ is defined to be 
\begin{equation*}
\label{var_cap_def}
\cpt(K, \RR^n) = \inf \left\{ \int_{\RR^n} |\nabla \phi|^p \ dx \ : \ \phi \in C_c^{\infty}(\RR^n),\ \lsb{\chi}{K}(x) \leq \phi(x) \leq 1 \right\},
\end{equation*}
where $\lsb{\chi}{K} (x) = 1 $ for $x \in K$ and $\lsb{\chi}{K}(x) = 0$ for $x \notin K$. To define the variational $p$-capacity of an open set $O \subset \RR^n$, we take the supremum over the capacities of the compact sets contained in $O$. The variational $p$-capacity of an arbitrary set $E \subset \RR^n$ is defined by taking the infimum over the capacities of the open sets containing $E$. For further details, see \cite{AH,HKM}.


Let us now introduce the capacity density conditions which we later impose on the complement of the domain.
\begin{definition}[Uniform $p$-thickness] \label{p_thick}
Let $\tO\subset\RR^n$ be a bounded domain and $b_0,r_0$ be any two given positive constants. We say that the complement $\tO^c:=\RR^n\setminus \tO$ is 
uniformly $p$-thick for some $1< p \leq n$ with constants  $ b_0,r_0>0$, if  the inequality  
$$ \cpt ( \overline{B_r(y_0)} \cap \tO^c, B_{2r}(y_0)) \geq  b_0\, \cpt (\overline{B_{r}(y_0)}, B_{2r}(y_0)),$$ 
holds for any $y_0 \in \partial \tO$ and $r\in(0, r_0]$. 
\end{definition}

It is well-known that the class of  domains with uniform $p$-thick complements
is very large. They include all domains with Lipschitz boundaries or even those that satisfy a uniform exterior corkscrew condition, where the latter means that there exist constants $c_0, r_0>0$ such that for all $0<t\leq r_0$ and all $x\in \RR^n\setminus \Om$, there is $y\in B_t(x)$
such that $B_{t/c_0}(y)\subset \RR^n\setminus \Om$.

If we replace the capacity by the Lebesgue measure in \cref{p_thick}, then we obtain a measure density condition. A set $E$ satisfying the measure density condition is uniformly $p$-thick for all $p>1$. If $p>n$, then every non-empty set is uniformly $p$-thick. The following lemma from \cite[Lemma 3.8]{Par} extends the capacity estimate in \cref{p_thick} to make precise the notion of being \emph{uniformly $p$-thick}:
\begin{lemma}[\cite{Par}]
Let $\tO$ be a bounded open set, and suppose that $\RR^n \setminus \tO$ is uniformly $p$-thick with constant $b_0,r_0$. Choose any $y \in \tO$ such that $B_{\frac34 r}(y)\setminus \tO \neq \emptyset$, then there exists a constant $b_1 = b_1(b_0,r_0,n,p)>0$ such that 
\[
  \cpt(\overline{B_{2r}(y)}\setminus \tO, B_{4r}(y) ) \geq b_1 \cpt( \overline{B_{2r}(y)}, B_{4r}(y) ). 
\]
\end{lemma}

Following the definition of $p$-thickness, a simple consequence of H\"older's and Young's inequality gives the following result (for example, see \cite[Lemma 3.13]{Par} for the proof):
\begin{lemma}
Let $1 < p \leq n$ be given and suppose  a set $E \subset \RR^n$ is uniformly $p$-thick with constants $b_0,r_0$. Then $E$ is uniformly $q$-thick for all $q \geq p$ with constants $b_1, r_1$.
\end{lemma}

A very important result regarding the uniform $p$-thickness condition is that it has the self improving property (see \cite{Le88} or  \cite{Anc,Mik} for the details):
\begin{theorem}[\cite{Le88}]
\label{self_improv_cap}
Let $1 < p \leq n$ be given and suppose a set $E \subset \RR^n$ is uniformly $p$-thick with constants $b_0,r_0$. Then there exists an exponent $q = q(n,p,b_0)$ with $1<q<p$ for which $E$ is uniformly $q$-thick with constants $b_1, r_1$. 
\end{theorem}
We next state a generalized Sobolev-Poincar\'e's inequality which was originally obtained by V. Maz'ya \cite[Sec. 10.1.2]{Maz} (see also    \cite[Sec. 3.1]{KK94} and \cite[Corollary 8.2.7]{AH}). 
\begin{theorem} \label{sobolev-poincare} Let $B$ be a ball and $\phi \in W^{1,p}(B)$ for some $p>1$. Let $\kappa\in [1,n/(n-p)]$ if $1<p<n$ and $\kappa\in[1,2]$ if 
$p=n$. Then there exists a constant $c = c(n, p) > 0$ such that 
\begin{equation*}
\label{sob_poin_est}
\left( \hint_B |\phi|^{\kappa p} \, dx\right)^{\frac{1}{\kappa p}} \leq c \left( 
\frac{1}{{\rm cap}_{1,p} (\overline{N(\phi)}, 2B ) } \int_B |\nabla \phi|^p \, dx\right)^{\frac{1}{p}},
\end{equation*}
where $N(\phi)=\{x \in B: \phi (x) = 0 \}$.
\end{theorem}
\subsection{Structural assumptions} 
In this subsection, we will mention all the assumptions we make on the operator $\aa(x,t,\zeta)$ as well as on the domain $\Om$. 
\subsubsection{Assumptions on \texorpdfstring{$\aa(x,t,\zeta)$}.}
We shall now collect the assumptions on the nonlinear structure $\aa(\cdot,\cdot,\cdot)$. Let $T>0$ be a fixed number, we then assume that the nonlinearity $\aa(x,t,\zeta) : \Om \times [0,T] \times \RR^n \mapsto \RR^n$ is an Carath\'eodory function, i.e.,  $(x,t) \mapsto \aa(x,t,\zeta)$  is measurable for every $\zeta \in \RR^n$ and 
$\zeta \mapsto \aa(x,t,\zeta)$ is continuous for almost every  $(x,t) \in \Om \times [0,T]$.

We further assume that for a.e. $(x,t) \in \Om \times [0,T]$ and for any $\zeta \in \RR^n$, there exist two  positive constants $\lamot$ such that  the following bounds are satisfied   by the nonlinear structures :
\begin{gather}
\iprod{\aa(x,t,\zeta)}{\zeta} \geq \La_0 |\zeta |^p  - h_1  \txt{and} |\aa(x,t,\zeta)| \leq \La_1  |\zeta|^{p-1} + h_2, \label{abounded}
\end{gather}
where, the functions $h_1, h_2: \Om \times [0,T] \mapsto \RR$ are assumed to be  measurable  with bounded norm 
\begin{equation}
\label{bound_b}
h_0^p:=|h_1| + |h_2|^{\frac{p}{p-1}} \txt{and} \| h_0\|_{L^{\hat{q}}(\Om \times [0,T])} < \infty \quad \text{ for some  } \hat{q} \geq p.
\end{equation}
\emph{An important aspect of the estimates obtained in this paper is that we do not make any  assumptions regarding the smoothness of $\aa(x,t,\zeta)$ with respect to $x,t,\zeta$.}

As the basic sets for our estimates, we will use parabolic cylinders where the radii in space and time are coupled. This is due to the fact that in the case that $p \neq 2$, the size of the cylinders intrinsically depends on the solution itself. This difficulty extends to the problems dealing with \emph{very weak} solutions also. 

In what follows, we will always assume the following restriction on the exponent $p$:
\begin{equation}
\label{restriction_p}
\frac{2n}{n+2} < p < \infty.
\end{equation}

\begin{remark}The restriction in \cref{restriction_p} is necessary when dealing with parabolic problems because of the compact embedding $W^{1,p} \hookrightarrow L^2$. Since solutions to parabolic problems require us to deal with $L^2$-norm of the solution which comes from the time-derivative, this restriction is natural. \end{remark}

%

\subsubsection{Assumptions on \texorpdfstring{$\pa \Om$}.}

\begin{definition}
\label{p_thick_domain}
In this paper, we shall assume that the domain $\Om$ is bounded and that it's complement $\Om^c$ is uniformly $p$-thick with constants $b_0,r_0$ in the sense of  \cref{p_thick}.

Applying \cref{self_improv_cap}, we will henceforth fix the exponent $\ve_0= \ve_0(n,p,b_0,r_0)$ to denote the self improvement property associated to $\pa\Om$.
\end{definition}

\subsection{Function Spaces}

Let $1\leq \vt < \infty$, then $W_0^{1,\vt}(\Om)$ denotes the standard Sobolev space which is the completion of $C_c^{\infty}(\Om)$ under the $\|\cdot\|_{W^{1,\vt}}$ norm. 

The parabolic space $L^{\vt}(0,T; W^{1,\vt}(\Om))$ for any $\vt \in (1,\infty)$ is the collection of measurable functions $f(x,t)$ such that for almost every $t \in (0,T)$, the function $x \mapsto f(x,t)$ belongs to $W^{1,\vt}(\Om)$ with the following norm being finite:
\[
 \| f\|_{L^{\vt}(0,T; W^{1,\vt}(\Om)} := \lbr \int_{0}^T \| u(\cdot, t) \|_{W^{1,\vt}(\Om)}^{\vt} \ dt \rbr^{\frac{1}{\vt}} < \infty.
\]

Analogously, the parabolic space $L^{\vt}(0,T; W_0^{1,\vt}(\Om))$ is the collection of measurable functions $f(x,t)$ such that for almost every $t \in (0,T)$, the function $x \mapsto f(x,t)$ belongs to $W_0^{1,\vt}(\Om)$.

\subsubsection{Negative Sobolev spaces}
We denote $W^{-1,\vt'}(\Om) := \lbr W_0^{1,\vt}(\Om)\rbr^*$ to  be the usual dual space. Then we have the following well known lemma (see \cite[Proposition 9.20]{brezis} for the proof).
\begin{lemma}
\label{lemma_lihe_wang}	
Let $\Om$ be any bounded domain, a function $ \varphi \in W^{-1,\vt'}(\Om)$ if and only if there exists functions $\{\phi_0, \phi_1,\phi_2,\ldots,\phi_n\}\in L^{\vt'}(\Om)$ such that 
\begin{equation}
\label{2.5_lihe}
\iprod{\varphi}{v} = \int_{\Om} \phi_0 v \ dx + \int_{\Om} \sum_{i=1}^n \phi_i \frac{\pa v}{\pa x_i} \ dx \qquad \forall\  v \in W_0^{1,\vt}(\Om).
\end{equation}
Moreover, there are $\{\phi_0,\phi_1,\phi_2,\ldots,\phi_n\}\in L^{\vt'}(\Om)$ such that  
\begin{equation}
\label{neg_sob_norm}
\|\varphi\|_{W^{-1,\vt'}(\Om)}^{\vt'} = \|\phi_0\|_{L^{\vt'}(\Om)}^{\vt'} + \sum_{i=1}^n \|\phi_i\|_{L^{\vt'}(\Om)}^{\vt'}.
\end{equation}
Here we can formally integrate by parts \cref{2.5_lihe} to get the representation
\[
\varphi = \phi_0 - \sum_{i=1}^n \frac{\pa \phi_i}{\pa x_i} = \phi_0-\dv (\phi_1,\phi_2,\ldots,\phi_n) = \phi_0 - \dv \vec{\phi}.
\]
Since $\Om$ is a bounded domain, we can take $\phi_0=0$.
\end{lemma}

An equivalent definition of the norm defined in \cref{neg_sob_norm} is given by
\begin{equation}\label{neg_sob_norm_quiv}
\|\phi\|_{W^{-1,\vt'}(\Om)} := \inf_{\phi = \varphi - \dv \psi} \|\varphi\|_{L^{\vt'}(\Om)} + \|\psi\|_{L^{\vt'}(\Om,\RR^n)},
\end{equation}
where the infimum is taken over all representations of the form  $\phi = \varphi - \dv \psi$ with $\varphi\in L^{\vt'}(\Om)$ and $\psi \in L^{\vt'}(\Om,\RR^n)$.

\subsection{Notion of Solution}

There is a well known difficulty in defining the notion of solution for \cref{main-1}, \cref{main-2} or \cref{main-3} due to a lack of time derivative of $u$. To overcome this, one can either use Steklov average or convolution in time. 
In this paper, we shall use the former approach (see also \cite[Page 20, Equation (2.5)]{DiB1} for further details). 

We will use two \emph{equivalent notions of solutions} depending on which equation we are handling. 

\subsubsection{Definition of Solution for \texorpdfstring{\cref{main-1}}. and \texorpdfstring{\cref{main-2}}.}

Let us first define Steklov average as follows: let $h \in (0, 2T)$ be any positive number, then we define
\begin{equation}\label{stek1}
  u_{h}(\cdot,t) := \left\{ \begin{array}{ll}
			      \hint_t^{t+h} u(\cdot, \tau) \ d\tau \quad & t\in (0,T-h), \\
			      0 & \text{else}.
			    \end{array}\right.
\end{equation}

We shall recall the following well known lemma regarding integral averages (for a proof in this setting, see for example \cite[Chapter 8.2]{bogelein2007thesis} for the details).
\begin{lemma}
\label{time_average}
Let $\psi:\RR^{n+1}\to\RR$ be an integrable function, $\la >0$ be any fixed number and suppose  $[\psi]_h(x,t) : = \hint_{t-\la h^2}^{t+\la h^2} \psi(x,\tau) \ d\tau$. Then we have the following properties:
\begin{enumerate}[(i)]
 \item $[\psi]_h \rightarrow \psi$ a.e $(x,t) \in \RR^{n+1}$ as $h \searrow 0$, 
 \item $[\psi]_h(x,\cdot)$ is continuous and  bounded  in time for a.e. $x \in \RR^n$.
 \item For any cylinder $Q_{r, \la r^2} \subset \RR^{n+1}$ with $r >0$, there holds
\[
  \fiint_{Q_{r,\ga r^2}} [\psi]_h(x,t) \ dx \ dt \apprle_n \fiint_{Q_{r,\la(r+h)^2}}\psi(x,t) \ dx \ dt. 
 \]
 \item The function $[\psi]_h(x,t)$ is differentiable with respect to $t \in \RR$, moreover $[\psi]_h(x, \cdot) \in C^{1}(\RR)$ for a.e. $x \in \RR^n$. 
 \end{enumerate}
\end{lemma}

We shall now define the notion of very weak solution:
\begin{definition}[Very weak solution] 
\label{very_weak_solution}
Let $ \be \in (0,1)$ and $h \in (0,2T)$ be given and suppose $p-\be > 1$. We  then say $u \in L^2(0,T; L^2(\Om)) \cap L^{p-\be}(0,T; u_0 + W_0^{1,p-\be}(\Om))$ is a very weak solution of 
\begin{equation*}
\left\{
\begin{array}{ll}
u_t - \dv  \aa(x,t,\nabla u) = 0 & \text{on} \ \Om \times (0,T), \\
u = u_0 & \text{on} \ \pa \Om \times (0,T),\\
u = u_1 & \text{on} \ \Om \times \{t=0\},
\end{array}
\right.
\end{equation*}
if for any $\phi \in W_0^{1,\frac{p-\be}{1-\be}}(\Om) \cap L^{\infty}(\Om)$, the following holds:
\begin{equation*}
\label{def_weak_solution}
  \int_{\Om \times \{t\}} \frac{d [u]_{h}}{dt} \phi + \iprod{[\aa(x,t,\nabla u)]_{h}}{\nabla \phi} \ dx = 0 \txt{for any}0 < t < T-h.
\end{equation*}
The initial condition is taken in the sense of $L^2(\Om)$, i.e.,
\begin{equation*}
\int_{B}\abs{u_h(x,0) - u_1(x)}^2 \ dx \xrightarrow{h \searrow 0} 0 \txt{for every} B \Subset \Om.
\end{equation*}

\end{definition}

\subsubsection{An equivalent Definition of very weak solution}
\begin{definition}[Very weak solution]
\label{very_weak_distribution}
Let $ \be \in (0,1)$ and $h \in (0,2T)$ be given and suppose $p-\be > 1$. We  then say $u \in L^2(0,T; L^2(\Om)) \cap L^{p-\be}(0,T; u_0 + W_0^{1,p-\be}(\Om))$ is a very weak solution of 
\begin{equation*}
\left\{
\begin{array}{ll}
u_t - \dv  \aa(x,t,\nabla u) = 0 & \text{on} \ \Om \times (0,T), \\
u = u_0 & \text{on} \ \pa \Om \times (0,T),\\
u = u_1 & \text{on} \ \Om \times \{t=0\},
\end{array}
\right.
\end{equation*}
if for any $\varphi \in C_0^\infty(\Om_T)$, the following holds:
\begin{equation*}
\iint_{\Om \times (0,t)}  -u \varphi_t + \iprod{\aa(x,t,\nabla u)}{\nabla \varphi}  \ dz = 0,
\end{equation*}
\end{definition}
and
\begin{equation*}
\int_{B}\abs{u_h(x,0) - u_1(x)}^2 \ dx \xrightarrow{h \searrow 0} 0 \txt{for every} B \Subset \Om.
\end{equation*}

\subsection{Some results about Maximal functions}

For any $f \in L^1(\RR^{n+1})$, let us now define the strong maximal function in $\RR^{n+1}$ as follows:
\begin{equation}
 \label{par_max}
 \mm(|f|)(x,t) := \sup_{\tQ \ni(x,t)} \fiint_{\tQ} |f(y,s)| \ dy \ ds,
\end{equation}
where the supremum is  taken over all parabolic cylinders $\tQ_{a,b}$ with $a,b \in \RR^+$ such that $(x,t)\in \tQ_{a,b}$. An application of the Hardy-Littlewood maximal theorem in $x-$ and $t-$ directions shows that the Hardy-Littlewood maximal theorem still holds for this type of maximal function (see \cite[Lemma 7.9]{Gary} for details):
\begin{lemma}
\label{max_bound}
 If $f \in L^1(\RR^{n+1})$, then for any $\al >0 $, there holds
 \[
  |\{ z \in \RR^{n+1} : \mm(|f|)(z) > \al\}| \leq \frac{5^{n+2}}{\al} \|f\|_{L^1(\RR^{n+1})},
 \]
 and if $f \in L^{\vartheta}(\RR^{n+1})$ for some $1 < \vartheta \leq \infty$, then there holds
 \[
  \| \mm(|f|) \|_{L^{\vartheta}(\RR^{n+1})} \leq C_{(n,\vartheta)} \| f \|_{L^{\vartheta}(\RR^{n+1})}.
 \]

\end{lemma}

Let us define the following new Maximal function defined in the dual Sobolev space: Let  $1<\vartheta<\infty$, then for any $f\in W^{-1,\vartheta}(\RR^n)$, we define
\begin{equation}
\label{def_max_neg}
\mm^{-1,\vt}(f)(x) := \sup_{B \ \text{ball}, B\ni x} \frac{1}{|B|^{\frac{1}{\vt}}} \|f\|_{W^{-1,\vt}(B)}.
\end{equation}

We now have the following important boundedness result for \cref{def_max_neg} obtained in \cite[Proposition 2.5]{bernicot2009maximal}. For the sake of completeness, we provide the proof.
\begin{lemma}
\label{max_neg_bound}
Let $1<\vt<\infty$ be given and let $q>\vt$ be fixed. Then for any $f \in W^{-1,q}(\RR^n)$, there holds
\[
\| \mm^{-1,\vt}(f) \|_{L^q(\RR^n)} \apprle_{(\vt,q,n)} \|f\|_{W^{-1,q}(\RR^n)}.
\]
\end{lemma}

\begin{proof}
Applying \cref{lemma_lihe_wang}, there exists $\phi_0 \in L^{q}(\RR^n)$ and $\psi_0\in L^{q}(\RR^n,\RR^n)$ such that $f=\phi_0-\dv\psi_0$. Let $B \subset \RR^n$ be any given ball, then from \cref{neg_sob_norm_quiv}, we see that
\[
\|f\|_{W^{-1,q}(B)}:= \inf_{f = \phi - \dv \psi} \|\phi\|_{L^q(B)} + \|\psi\|_{L^q(B,\RR^n)}.
\]
Using this, we get the following sequence of estimates:
\begin{equation*}
\begin{array}{rcl}
\frac{1}{|B|^{\frac{1}{\vt}}} \|f\|_{W^{-1,\vt}(B)}(x) & = & \inf_{f = \phi - \dv \psi} \frac{1}{|B|^{\frac{1}{\vt}}} \lbr \|\phi\|_{L^{\vt}(B)} + \|\psi\|_{L^{\vt}(B,\RR^n)}\rbr \\
&\le&  \frac{1}{|B|^{\frac{1}{\vt}}} \lbr \|\phi_0\|_{L^{\vt}(B)} + \|\psi_0\|_{L^{\vt}(B,\RR^n)}\rbr \\
& = &  \lbr \fint_B |\phi_0|^{\vt} \ dx \rbr^{\frac{1}{\vt}}+ \lbr \fint_B |\psi_0|^{\vt} \ dx \rbr^{\frac{1}{\vt}} \\
& \leq &  \mm(|\phi_0|^{\vt})^{\frac{1}{\vt}}(x) + \mm(|\psi_0|^{\vt})^{\frac{1}{\vt}}(x).
\end{array}
\end{equation*}
Here $x$ is any point in the ball $B$, since we have used uncentered Hardy-Littlewood maximal function as defined in \cref{par_max}. Now taking supremum over all balls $B\ni x$  followed by taking the norm in $L^q$, we get
\begin{equation*}
\begin{array}{rcl}
\left\|\sup_{B\ni x} \frac{1}{|B|^{\frac{1}{\vt}}} \|f\|_{W^{-1,\vt}(B)}(x) \right\|_{L^q(\RR^n)} &\apprle & \|\mm(|\phi_0|^{\vt})^{\frac{1}{\vt}}\|_{L^q(\RR^n)} + \|\mm(|\psi_0|^{\vt})^{\frac{1}{\vt}}\|_{L^q(\RR^n)}\\
& \apprle &  \|\phi_0\|_{L^q(\RR^n)} + \|\psi_0\|_{L^q(\RR^n)}\\
& = & C(\vt,q,n) \|f\|_{W^{-1,q}(\RR^n)}.
\end{array}
\end{equation*}
This completes the proof of the lemma.
\end{proof}

\subsection{Notations}
We shall clarify the notation that will be used throughout the paper: 
\begin{enumerate}[(i)]
    
\item\label{not1} We shall use $\nabla$ or $\dv$ to denote derivatives only with respect the space variable $x$.
\item\label{not2} We shall sometimes alternate between using $\ddt{f}$, $\pa_t f$ and $f'$ to denote the time derivative of a function $f$.
\item\label{not3} We shall use $D$ to denote the derivative with respect to both the space variable $x$ and time variable $t$ in $\RR^{n+1}$. 
\item\label{not4}  Let $z_0 = (x_0,t_0) \in \RR^{n+1}$ be a point and $\rho, s >0$ be two given parameters and let $\al \in (0,\infty)$. We shall use the following symbols to denote the following regions:
\begin{equation*}
\label{notation_space_time}
\def\arraystretch{1.5}
\begin{array}{ll}
\tm_s(t_0) := (t_0 - s, t_0+s) \subset \RR,
& \qquad Q_{\rho,s}(z_0) := B_{\rho}(x_0) \times \tm_{s}(t_0) \subset \RR^{n+1},\\ 
\al Q_{\rho,s}(z_0) := B_{\al \rho}(x_0) \times \tm_{\al^2s}(t_0)  \subset \RR^{n+1},
&\qquad \mch_s(t_0) := \RR^n \times \tm_s(t_0) \subset \RR^{n+1},\\ 
\alpha \mch_s(t_0) := \RR^n \times \tm_{\alpha^2 s}(t_0) \subset \RR^{n+1},
& \qquad\mcc_{\rho}(x_0) := \Om \cap B_{\rho}(x_0) \times \RR \subset \RR^{n+1},\\
\Om_{\rho,s}(z_0) := \Om \cap B_{\rho}(x_0) \times \tm_s(t_0) \subset \RR^{n+1}, 
& \qquad\Om_{\rho}(x_0) := \Om \cap B_{\rho}(x_0) \subset \RR^{n}.
\end{array}
\end{equation*}

\item We shall use the notation $\{t\leq 0\}$ to denote the region $\RR^n \times (-\infty, 0]$. The region $\{t\geq 0\}$ is analogously defined. 

\item\label{not5} We shall use $\int$ to denote the integral with respect to either space variable or time variable and use $\iint$ to denote the integral with respect to both space and time variables simultaneously. 

Analogously, we will use $\fint$ and $\fiint$ to denote the average integrals as defined below: for any set $A \times B \subset \RR^n \times \RR$, we define
\begin{gather*}
\avgs{f}{A}:= \fint_A f(x) \ dx = \frac{1}{|A|} \int_A f(x) \ dx,\\
\avgs{f}{A\times B}:=\fiint_{A\times B} f(x,t) \ dx \ dt = \frac{1}{|A\times B|} \iint_{A\times B} f(x,t) \ dx \ dt.
\end{gather*}

\item\label{not6} Given any positive function $\mu$, we shall denote $\avgs{f}{\mu} := \int f\frac{\mu}{\|\mu\|_{L^1}}dm$ where the domain of integration is the domain of definition of $\mu$ and $dm$ denotes the associated measure.

\item\label{not7} Given any $\la> 0$, we can convert $\RR^{n+1}$ into a metric space where the parabolic cylinders correspond to \emph{balls} under the parabolic metric given by:
\begin{equation}
\label{parabolic_metric}
d_{\la} (z_1,z_2) := \max \left\{ |x_2-x_1|, \sqrt{\la^{p-2}|t_2-t_1|} \right\}.
\end{equation}

\item\label{not8} In what follows, $r_0$ and $b_0$ will denote the constants arising from the assumption that  $\Om^c$ is uniformly $p$-thick and denote $\ve_0=\ve_0(n,p,b_0,r_0)$ to be the self improvement exponent (see \cref{p_thick_domain}).
\end{enumerate}

\section{Main Theorems}
\label{section3}

In this section, we will describe the main theorems that will be proved. Note that \cref{restriction_p} is always in force. The first theorem is an a priori estimate that controls the gradient of the solution in terms of the boundary data.
\begin{theorem}
\label{main_theorem_1}
Let $\Om$ be a bounded domain whose complement is uniformly $p$-thick with constants $(b_0,r_0)$ as in \cref{p_thick}. There exists $\be_1 = \be_1(p,n,\lamot,b_0,r_0)$ such that for any $\be \in (0,\be_1)$, suppose 
\begin{equation}
\label{thm3.1-hyp} 
w \in \wwspace,  \quad \ddt{w} \in L^1_{\loc}(\Om\times [0,T])\quad  \text{and} \quad \ddt{w} \in L^{\pbp} (0,T; W^{-1,\pbp}(\Om)), 
\end{equation}
be any given function. Then for any very weak solution $u \in L^2(0,T; L^2(\Om)) \cap L^{p-\be}(0,T;W^{1,p-\be}(\Om))$ solving \cref{main-1}, the following a priori estimate holds:
\begin{equation*}
\label{a_priori_hom}
\iint_{\Om_T} |\nabla u|^{p-\be} \ dz \apprle_{(n,p,\be,\lamot)} \iint_{\Om_T}  |\nabla w|^{p-\be} \ dz +\iint_{\Om_T}|h_0|^{p-\be}\ dz + \left\|\ddt{w} \right\|_{L^{\pbp}(0,T; W^{-1,\pbp}(\Om))}^{\pbp}.
\end{equation*}
\end{theorem}

\begin{remark}
\label{imp_rmk}
The additional assumption $\ddt{w} \in L^1_{\loc}(\Om\times [0,T])$ in \cref{thm3.1-hyp} can be replaced by the following weaker assumption: Let $t_1, t_2 \in [0,T]$, $\phi\in C_c^{\infty}(\Om)$ and $\varphi \in C^{\infty}[t_1,t_2]$, then assume the following holds
\begin{equation}
\label{est3.3}
\int_{t_1}^{t_2} \iprod{\ddt{[w]_h}}{\phi}_{(W^{-1,\pbp}(\Om),W_0^{1,\pbo}(\Om))} (t) \varphi(t)\ dt = \int_{t_1}^{t_2} \left[ \iprod{\ddt{w}}{\phi}_{(W^{-1,\pbp}(\Om),W_0^{1,\pbo}(\Om))}\right]_h (t) \varphi(t)\ dt.
\end{equation}
This is necessary to obtain the estimate in \cref{lemma_crucial_2_app}. As a consequence, all the estimates in \cref{section4} are applicable provided \cref{est3.3} holds. 

Heuristically speaking, the equality in \cref{est3.3} asks for a general form of the  fundamental theorem of calculus to  hold for distributions of the form $\ddt{w}$. In particular, we would need the following to hold for all $\phi \in C_c^{\infty}(\Om)$ and $\varphi \in C_c^{\infty}(\RR)$:
\[
{\int_{\Om} \lbr (w\varphi)(x,b) - (w\varphi)(x,a) \rbr \phi(x)\ dx} = \int_a^b \int_{\Om}\ddt{(w\varphi)}(x,t) \phi(x) \ dx \ dt.
\]
If the distribution $\ddt{w} \in L^1_{\loc}(\Om_T)$, then the above equality holds,  which implies the equality in \cref{est3.3}, see \cite{heikkila2013distributions} for the details.
\end{remark}

In the special case of the initial boundary value being zero, we get an analogous result stated below.
\begin{corollary}
\label{corollary_main_theorem_1}
Let $\Om$ be a bounded domain whose complement is uniformly $p$-thick with constants $(b_0,r_0)$ as in \cref{p_thick}. Let
\begin{equation*}
\label{cor3.2-hyp}
w \in \wwspace,  \quad \ddt{w} \in L^1_{\loc}(\Om\times [0,T])\quad  \text{and} \quad \ddt{w} \in L^{\pbp} (0,T; W^{-1,\pbp}(\Om)), 
\end{equation*}
be any given function. Then there exists a $\be_1 = \be_1(p,n,\lamot,b_0,r_0)$ such that for any $\be \in (0,\be_1)$ and any very weak solution $u \in L^2(0,T; L^2(\Om)) \cap L^{p-\be}(0,T;W^{1,p-\be}(\Om))$ solving
\begin{equation*}
\left\{
\begin{array}{ll}
u_t - \dv  \aa(x,t,\nabla u) = 0 & \text{on} \ \Om \times (0,T), \\
u = w & \text{on} \ \pa \Om \times (0,T),\\
u = 0 & \text{on} \ \Om \times \{t=0\},
\end{array}
\right.
\end{equation*}
the following a priori estimate holds:
\begin{equation*}
\label{corollary_main_theorem_1_estimate}
\iint_{\Om_T} |\nabla u|^{p-\be} \ dz \apprle_{(n,p,\be,\lamot)} \iint_{\Om_T} \lbr |\nabla w|+|h_0|\rbr^{p-\be}\ dz + \left\|\ddt{w} \right\|_{L^{\pbp}(0,T; W^{-1,\pbp}(\Om))}^{\pbp}.
\end{equation*}
Here the initial condition is taken to hold in the sense
\[
\int_B \lbr |w_h(x,0)|^2  +  |u_h(x,0)|^2 \rbr\ dx \xrightarrow{h \searrow 0} 0 \txt{for all} B \subset \Om.
\]

\end{corollary}

The second theorem that we will prove is a higher integrability result for very weak solutions at the initial boundary. 
\begin{theorem}
\label{main_theorem_2}
Let $w$ be such that
\begin{equation*}
\label{w_ini_data}
w \in L^p(0,T; W^{1,p}(\Om)) \txt{and} \ddt{w} \in L^{\frac{p}{p-1}}(0,T; W^{-1,\frac{p}{p-1}}(\Om))\cap L^1_{\loc}(\Om_T),
\end{equation*}
then there exists $\be_2 = \be_2(n,p,\lamot) \in (0,1)$ such that the following holds: for any $\be \in (0,\be_2)$ and any very weak solution $u \in L^2(0,T; L^2_{\loc}(\Om)) \cap L^{p-\be}(0,T; W_{\loc}^{1,p-\be}(\Om))$ solving \cref{main-2} in $\Om_T$,  we have the following improved integrability $u \in L^2(0,T; L^2_{\loc}(\Om)) \cap L^{p}(0,T; W_{\loc}^{1,p}(\Om))$. In particular, for any fixed $\rho \in (0,\infty)$ and $s \in (0,T/2)$, the following quantitative estimate holds: let   $Q_{2\rho,2s}\Subset \Om\times \RR$ be any parabolic cylinder, then there holds
\begin{equation*}
\begin{array}{rcl}
 \fiint_{Q_{\rho,s}} |\nabla u|^p \lsb{\chi}{[0,T]}\ dz & \apprle &  \lbr \fiint_{Q_{2\rho,2s}} \lbr |\nabla u| + |h_0| \rbr^{p-\be}\lsb{\chi}{[0,T]} \ dz \rbr^{1+\frac{\be}{d}} + \fiint_{Q_{2\rho,2s}} (1+h_0^p)\lsb{\chi}{[0,T]} \ dz \\
 && + \lbr \fiint_{Q_{2\rho,2s}} |\nabla w|^{p-\be}  \lsb{\chi}{[0,T]} \ dz \rbr^{1 + \frac{\be}{d}} + \lbr \fiint_{Q_{2\rho,2s}} |\vec{w}|^{\frac{p-\be}{p-1}} \lsb{\chi}{[0,T]}  \ dz \rbr^{1 + \frac{\be}{d}}\\
 &&  + \fiint_{Q_{2\rho,2s}} |\nabla w|^p \lsb{\chi}{[0,T]} \ dz+ \fiint_{Q_{2\rho,2s}} |\vec{w}|^{\frac{p}{p-1}} \lsb{\chi}{[0,T]}  \ dz, 
 \end{array}
\end{equation*}
where $h_0$ is from \cref{bound_b}, $\vec{w}$ is as obtained in \cref{def_vec_w},  $C = C(n,p,\lamot)$  and 
\begin{equation}
 \label{def_d-ini}
 d:= \left\{ \begin{array}{ll}
              2-\be & \text{if} \ p \geq 2, \\
              p-\be - \frac{(2-p)n}{2} & \text{if} \ p<2.
             \end{array}\right.
\end{equation}

\end{theorem}

Note that \cref{main_theorem_2} holds under the weaker assumption from \cref{imp_rmk} instead of the stronger assumption $\ddt{w} \in L^1_{\loc}(\Om_T)$. In the special case of zero initial data, we have the following important corollary:
\begin{corollary}
\label{main_corollary_2}
There exists $\be_2 = \be_2(n,p,\lamot) \in (0,1)$ such that the following holds: for any $\be \in (0,\be_2)$ and any very weak solution $u \in L^2(0,T; L^2_{\loc}(\Om)) \cap L^{p-\be}(0,T; W_{\loc}^{1,p-\be}(\Om))$ solving
\begin{equation*}
\left\{
\begin{array}{ll}
u_t - \dv  \aa(x,t,\nabla u) = 0 & \text{in} \ \Om \times (0,T), \\
u=0 & \text{on} \ \Om \times \{t=0\},
\end{array}
\right.
\end{equation*}
we have the following improved integrability $u \in L^2(0,T; L^2_{\loc}(\Om)) \cap L^{p}(0,T; W_{\loc}^{1,p}(\Om))$. In particular, for any fixed $\rho \in (0,\infty)$ and $s \in (0,T/2)$, the following quantitative estimate holds: let   $Q_{2\rho,2s}\Subset \Om\times \RR$ be any parabolic cylinder, then there holds
\begin{equation*}
\label{main_est_corollary_2}
  \fiint_{Q_{\rho,s}} |\nabla u|^{p} \lsb{\chi}{[0,T]} \ dz \leq C\left[ \lbr \fiint_{Q_{2\rho,2s}} \lbr |\nabla u| + h_0\rbr^{p-\be} \lsb{\chi}{[0,T]} \ dz\rbr^{1+\frac{\be}{d}}  + \fiint_{Q_{2\rho,2s}} (1 + h_0^p)\lsb{\chi}{[0,T]} \ dz \right],
 \end{equation*}
where $h_0$ is from \cref{bound_b}, $C = C(n,p,\lamot)$  and $d$ is from \cref{def_d-ini}.
\end{corollary}

In the above theorem, it is easy to see that the choice of the parabolic cylinder $Q_{\rho,s}$ is made such that it crosses the initial boundary at $\{t=0\}$. If the cylinder is completely contained in $\Om \times (0,T)$, then there is nothing to prove as this is the main result obtained in \cite{KL}. Hence, the new contribution is only when the cylinder crosses the initial boundary. 

\begin{remark}
\label{rmk_be_0}
In what follows, we will define $\be_0 := \min\{ \be_1, \be_2, {\be}_{int}\}$ where $\be_1$ is from \cref{main_theorem_1}, $\be_2$ is from \cref{main_corollary_2} and ${\be}_{int}$ is the interior higher integrability exponent such that \cite[Theorem 2.8]{KL} holds for all $\be \in (0,{\be}_{int}]$. All the subsequent results will hold for any $\be \in (0,\be_0)$.
\end{remark}

Finally, we are ready to state the main existence theorem that will be proved in this paper.
\begin{theorem}
\label{main_existence}
Let $\Om$ be a bounded domain whose complement is uniformly $p$-thick with constants $(b_0,r_0)$ as in \cref{p_thick}, let $\be \in (0,\be_0)$ be given with $\be_0= \be_0(n,p,\lamot,b_0,r_0)$ as in \cref{rmk_be_0} and \cref{bound_b} holds. Moreover, let the nonlinearity $\mathcal{A}$ satisfy \cref{abounded} and
\begin{equation}\label{3.7-1}
\left\langle \mathcal{A}(x,t,\eta)-\mathcal{A}(x,t,\zeta),\eta-\zeta \right\rangle \ge \Lambda_0(\lvert \eta\rvert^2+\lvert \zeta\rvert^2)^\frac{p-2}{2}\lvert\eta-\zeta\rvert^2\qquad \forall \ (x,t)\in \Om_T\text{ and   }\ \  \forall \ \eta,\zeta\in\mathbb{R}^{n}.
\end{equation}
Suppose that the boundary condition satisfies
\[
u_0 \in  \wwspace \quad \text{with} \quad \ddt{u_0} \in L^{\pbp} (0,T; W^{-1,\pbp}(\Om))\cap L^1_{\loc}(\Om_T),
\]
(or the weaker assumption from \cref{imp_rmk} holds instead of $\ddt{u_0} \in L^1_{\loc}(\Om_T)$). Furthermore, we assume that there exists an approximating sequence $u_0^k \in L^p(0,T;W^{1,p}(\Om))$ with $\ddt{u_0^k} \in L^{\frac{p}{p-1}}(0,T; W^{-1,\frac{p}{p-1}}(\Om))$ such that
\begin{equation}\label{assumption1}
u_0^k \rightarrow u_0 \ \  \text{in}\ \wwspace \txt{and}  \ddt{u_0^k} \rightarrow \ddt{u_0}\ \  \text{in}\  L^{\pbp} (0,T; W^{-1,\pbp}(\Om)),
\end{equation}
satisfying the bound
\begin{equation}
\label{assumption2}
\begin{array}{c}
\iint_{\Om_T} |\nabla u_0^k|^{p-\be} \ dz \leq C_{app}\iint_{\Om_T} |\nabla u_0|^{p-\be} \ dz, \\
\left\|\ddt{u_0^k} \right\|_{L^{\pbp}(0,T; W^{-1,\pbp}(\Om))}^{\pbp} \leq C_{app} \left\|\ddt{u_0} \right\|_{L^{\pbp}(0,T; W^{-1,\pbp}(\Om))}^{\pbp}.
\end{array}
\end{equation}
 Then there exists a very weak solution $u \in C^0(0,T;L^2(\Om)) \cap L^{p-\be} (0,T; u_0 + W_0^{1,p-\be}(\Om))$ solving \cref{main-3} in the sense of \cref{very_weak_distribution}. 
 \end{theorem}
In the above theorem, the initial condition is assumed to be satisfied in the following sense:
\begin{equation*}
\label{ext_thm_ini}
\int_B \lbr |[u_0]_h(x,0)|^2 + |u_h(x,0)|^2 \rbr\ dx \xrightarrow{h \searrow 0 } 0 \txt{for all} B \subset \Om.
\end{equation*}

\begin{remark}
\label{rmk_3.8}
In \cref{main_existence}, the assumption $\ddt{u_0} \in L^1(\Om_T)$ can be used to construct an approximating sequence $u_0^k$ satisfying \cref{assumption1} and \cref{assumption2} through the process of mollification, i.e., if $\eta_k \in C_c^{\infty}(\RR^{n+1})$ is the standard mollifier, then letting $u_0^k : = u_0 \ast \eta_k$ would be an admissible approximating sequence (see \cite[Section 4.4]{brezis} for more details). 

On the other hand, the assumptions \cref{assumption1} and \cref{assumption2} are necessary if $\ddt{u_0}$ is a distribution and not a function.
\end{remark}

%
%

\section{Some well known lemmas}
\label{section4}

\subsection{Sobolev and Sobolev-Poincar\'e lemmas}
Let us first recall a  time localised version of the parabolic Poincar\'e inequality proved in \cite[Lemma 4.2]{adimurthi2018gradient}.
\begin{lemma}
 \label{lemma_crucial_1}
 Let $\psi \in L^{\vt} (0,T; W^{1,\vt}(\Om))$ with $\vt \in [1,\infty)$  and suppose that $B_{\rho} \Subset \Om$ be compactly contained ball of radius $\rho>0$. Let $I \subset (0,T)$ be a time interval  and $\xi(x,t) \in L^1(B_{\rho} \times I)$ be any positive function such that 
 $$
 \|\xi\|_{L^{\infty}(B_{\rho}\times I)} \apprle_n \frac{\|\xi\|_{L^1(B_{\rho}\times I)}}{|B_{\rho}\times I|}, 
 $$
 and $\mu(x) \in C_c^{\infty}(B_{\rho})$ be  such that $\int_{B_{\rho}} \mu(x) \ dx = 1$ with $|\mu| \apprle  \frac{1}{{\rho}^n}$ and $|\nabla \mu| \apprle  \frac{1}{{\rho}^{n+1}}$, then there holds:
 \begin{equation*}
 \begin{array}{ll}
  \fiint_{B_{\rho} \times I} \left|\frac{\psi(z)\lsb{\chi}{J} - \avgs{\psi\lsb{\chi}{J}}{\xi}}{{\rho}}\right|^{\vt} \ dz & \apprle_{(n,\vt)} \fiint_{B_{\rho}\times I} |\nabla \psi|^{\vt}\lsb{\chi}{J} \ dz + \sup_{t_1,t_2 \in I} \left| \frac{\avgs{\psi\lsb{\chi}{J}}{\mu}(t_2) - \avgs{\psi\lsb{\chi}{J}}{\mu}(t_1)}{{\rho}} \right|^{\vt},
  \end{array}
 \end{equation*}
where $\avgs{\psi}{\xi}:= \int_{B_{\rho}\times I} \psi(z) \frac{\xi(z)}{\|\xi\|_{L^1(B_{\rho}\times I)}}\lsb{\chi}{J} \ dz\ $, $\avgs{\psi\lsb{\chi}{J}}{\mu}(t_i) := \int_{B_{\rho}} \psi(x,t_i) \mu(x) \lsb{\chi}{J} \ dx$ and $J \Subset (-\infty,\infty)$ be some fixed time-interval. 
\end{lemma}
In the above \cref{lemma_crucial_1}, we can take any bounded region $\tilde{\Om}$ instead of $B_r$ such that $\tilde{\Om}^c$ is uniformly $p$-thick as in \cref{p_thick}. In that case, the constants will subsequently depend on the $p$-thickness constants of $\tilde{\Om}^c$.

We will  need the following well known 
%
%
Gagliardo-Nirenberg's inequality (see \cite[Lemma 3.2]{bogelein2009very} for the details):
\begin{lemma}
\label{gagliardo_nirenberg_lemma}
Let $B_{\rho} \subset \RR^n$ with $\rho \in (0,1]$ and $f \in W^{1,\vt}(B_{\rho})$ and $1 \leq \sigma, \vt,r \leq \infty$ and $\delta \in (0,1)$ be given satisfying
\begin{equation}
\label{9.20-ini}
-\frac{n}{\sigma} \leq \delta\lbr 1 - \frac{n}{\vt}\rbr - (1-\delta) \frac{n}{r}.
\end{equation}
Then there exists a constant $C = C(n,\sigma, \vt)$ such that there holds
\begin{equation*}
\fint_{B_{\rho}} \abs{\frac{f}{\rho}}^{\sigma} \ dx \leq C \lbr \fint_{B_{\rho}} \abs{\frac{f}{\rho}}^{\vt} + |\nabla f|^{\vt} \ dx \rbr^{\frac{\delta \sigma}{\vt}} \lbr \fint_{B_{\rho}} \abs{\frac{f}{\rho}}^r \ dx \rbr^{\frac{(1-\delta)\sigma}{r}}.
\end{equation*}
\end{lemma}

\subsection{Whitney type decomposition lemma}
Let us first recall a well known Whitney type  decomposition Lemma proved in  \cite[Lemma 3.1]{diening2010existence} or \cite[Chapter 3]{bogelein2013regularity}:
\begin{lemma}
\label{whitney_decomposition}

Let $\mathbb{E}$ be any closed set and  $\la \in (0,\infty)$ be a fixed constant. Define $\gamma := \la^{2-p}$, then there exists a $\gamma$-parabolic Whitney covering  $\{Q_i(z_i)\}$ of $\mathbb{E}^c$ in the following sense:
\begin{description}
  \descitem{W1} $Q_j(z_j) = B_j(x_j) \times I_j(t_j)$ where $B_j(x_j) = B_{r_j}(x_j)$ and $I_j(t_j) = (t_j - \gamma r_j^2, t_j + \gamma r_j^2)$. 
  \descitem{W2}  $d_{\la}(z_j,\mathbb{E}) = 16r_j$.
  \descitem{W3} $\bigcup_j \frac12 Q_j(z_j) = \mathbb{E}^c$.
  \descitem{W4} for all $j \in \NN$, we have $8Q_j \subset \mathbb{E}^c$ and $16Q_j \cap \mathbb{E} \neq \emptyset$.
  \descitem{W5} if $Q_j \cap Q_k \neq \emptyset$, then $\frac12 r_k \leq r_j \leq 2r_k$.
  \descitem{W6} $\frac14 Q_j \cap \frac14Q_k = \emptyset$ for all $j \neq k$.
  \descitem{W7} $\sum_j \lsb{\chi}{4Q_j}(z) \leq c(n)$ for all $z \in \mathbb{E}^c$.
    \end{description}

    For a fixed $k \in \NN$, let us define $A_k := \left\{ j \in \NN: \frac34Q_k \cap \frac34Q_j \neq \emptyset\right\}$,  then we have
    \begin{description}
  \descitem{W8} For any $i \in \NN$, we have $\# A_i \leq c(n)$.
  \descitem{W9} Let $i \in \NN$ be given and let  $j \in A_i$, then $\max \{ |Q_j|, |Q_i|\} \leq C(n) |Q_j \cap Q_i|.$
  \descitem{W10}  Let $i \in \NN$ be given and let  $j \in A_i$, then $ \max \{ |Q_j|, |Q_i|\} \leq \left|\frac34Q_j \cap \frac34Q_i\right|.$
  \descitem{W11} Let $i \in \NN$ be given, then for any $j \in A_i$, we have $\frac34Q_j \subset 4Q_i$.
  \end{description}
  \end{lemma}
  
  Subordinate to the above Whitney covering, we have an associated partition of unity which we recall in the following lemma.
  \begin{lemma}
  \label{partition_unity}
  Associated to the covering given in \cref{whitney_decomposition}, there exists functions $\{ \Psi_j\}_{j \in \NN} \in C_c^{\infty}\lbr\frac34Q_j\rbr$ such that the following holds:
  \begin{description}
  \descitem{W12} $\lsb{\chi}{\frac12Q_j} \leq \Psi_j \leq \lsb{\chi}{\frac34Q_j}$.
  \descitem{W13} $\|\Psi_j\|_{\infty} + r_j \| \nabla \Psi_j\|_{\infty} + r_j^2 \| \nabla^2 \Psi_j\|_{\infty} + \la^{2-p} r_j^2 \| \pa_t \Psi_j\|_{\infty} \leq C(n)$.
  \descitem{W14} Let $i \in \NN$ be given, then $\sum_{j \in A_i} \Psi_j(z) = 1$  for all $z \in \frac34Q_i$.
\end{description}
\end{lemma}

\subsection{A few other well known lemmas}

Let us first recall the well known iteration lemma  (see \cite[Lemma 3.3]{bogelein2009very} for the details):
\begin{lemma}
\label{iteration_lemma}
Let $ \de \in (0,1)$, $B \geq 0$, $A \geq 0$, $\al >0$ and $0 < r < \rho < \infty$ and let $f \geq 0$ be a bounded, measurable function satisfying 
\[
f(t_1) \leq \de f(t_2) + A (t_2-t_1)^{-\al} + B \txt{for all} r \leq t_1 < t_2 \leq \rho,
\]
then there exists a constant $C = C(\al,\de)$ such that the following holds:
\[
f(r) \leq C \lbr A (\rho - r)^{-\al} + B \rbr.
\]

\end{lemma}

We will use the following result  which can be found in \cite[Theorem 3.1]{PG} (see also \cite{Prato} where it was originally proved) for proving the Lipschitz regularity for the constructed test function. This very important simplification of the original technique from \cite{KL} first appeared in \cite[Chapter 3]{bogelein2013regularity}.
\begin{lemma}
\label{metric_lipschitz}
 Let $\ga >0$ and  $\mathcal{D} \subset \RR^{n+1}$ be given. For any $z \in \mathcal{D}$ and $r>0$, let $Q_{r,\ga r^2}(z)$ be the parabolic cylinder centred  at $z$ with radius $r$.  Suppose  there exists a constant $C>0$ independent of $z$ and $r$ such that the following bound holds:
 \[
  \frac{1}{|Q_{r,\ga r^2}(z)\cap \mathcal{D}|} \iint_{Q_{r,\ga r^2}(z)\cap \mathcal{D}} \left|\frac{f(x,t) - \avgs{f}{Q_{r,\ga r^2}(z)\cap \mathcal{D}}}{r}\right| \ dx \ dt \leq C \qquad \forall\  z \in \mathcal{D} \text{ and } r>0,
 \]
 then $f \in C^{0,1} (\mathcal{D})$ with respect to the metric $d (z_1,z_2) := \max \{ |x_1-x_2|, \sqrt{\ga^{-1} |t_1-t_2| }\}$.
\end{lemma}

Finally, let us recall the parabolic version of the well known Gehring's lemma (for example, see \cite[Lemma 6.4]{bogelein2009very} for the details).
\begin{lemma}
\label{gehring_lemma}
 Let $\al_0 \geq 1$, $\ka \geq 1$, $\ve_0>0$, $p>1$ and $\be_{gh}>0$ be given. Let $q$ be given such that  $1<p-\varepsilon_0 \leq q < p-2\be<p-\be$ for some $\be \in (0,\be_{gh})$.  Furthermore, for a cylinder $Q_2 = Q_{2\rho,2s^2}$, let  $f \in L^{p-\be}(Q_2)$ and $g \in L^{\tp}(Q_2)$ for some $\tp \geq p$ be given. Suppose for each $\la \geq \al_0$ and almost every $z \in Q_2$ with $f(z) >\la$, there exists a parabolic cylinder $Q = Q_{\rho,s}(z)\subset Q_2$ such that 
 \[
  \frac{\la^{p-\be}}{\ka} \leq \fiint_Q f^{p-\be}(\tz) \lomt \ d\tz \leq \ka \lbr \fiint_Q f^q(\tz) \lomt \ d\tz \rbr^{\frac{p-\be}{q}} + \ka \fiint_Q g^{p-\be}(\tz) \lomt \ d\tz \leq \ka^2 \la^{p-\be},
 \]
then there exists $\de_0 = \de_0(\ka,p,\be,q,\ve_0)$ and $C = C(\ka,p,\be,q,\ve_0)$, such that $f \in L^{p-\be + \de_1}(Q_2)$ with $\de_1 = \min \{ \de_0, \tp-p+\be\}$. This improved higher integrability comes with the following bound:
\[
 \iint_{Q_2} f^{p-\be + \de}(\tz) \lomt \ d\tz \apprle \al_0^{\de} \iint_{Q_2} f^{p-\be}(\tz) \lomt \ d\tz + \iint_{Q_2} g^{p-\be+\de}(\tz) \lomt\ d\tz \quad \text{for all} \ \de \in (0,\de_1]. 
\]
\end{lemma}

\subsection{Crucial lemma}

The lemma concerns very weak solutions of \cref{main-1} which will be used in Section \cref{section5} and Section \cref{section6}.
\begin{lemma}
\label{lemma_crucial_2_app}
 Let $u \in L^2(0,T;L^2(\Om))\cap L^{p-\be} (0,T; w+ W_0^{1,p-\be}(\Om))$  be a very weak solution of \cref{main-1} for some $0 \leq \be \leq \min\{1,p-1\}$ and  $h \in (0,T)$. Let $\phi(x) \in C_c^{\infty}({\Om})$ and $\varphi(t) \in C^{\infty}(\RR)$ be non-negative functions and $[u]_h,[w]_h$ be the Steklov average as defined in \cref{time_average}. Furthermore, let $w$ satisfy the hypothesis in \cref{thm3.1-hyp} (see \cref{imp_rmk}), then   the following estimate holds for any time interval $(t_1,t_2) \subset (0,T)$:
 \begin{equation*}
  \label{lemma_crucial_2_est_app}
  \begin{array}{ll}
  |\avgs{{[u-w]_h\varphi}}{\phi} (t_2) - \avgs{{[u-w]_h\varphi}}{\phi}(t_1)| & \leq  C(\La_1,p)\|\nabla \phi\|_{L^{\infty}{({{\Om}})}} \|\varphi\|_{L^{\infty}(t_1,t_2)} \iint_{{{\Om}} \times (t_1,t_2)} {[(|\nabla u|+|h_0|)^{p-1}]_h} \ dz \\
   &  \qquad + \|\varphi\|_{L^{\infty}(t_1,t_2)} \int_{t_1}^{t_2}  \left|\left[ \int_{{{\Om}}} \iprod{\vec{w}}{\nabla \phi}\ dx \right]_h\right| \ dt \\
   & \qquad + \|\phi\|_{L^{\infty}{({{\Om}})}} \|\varphi'\|_{L^{\infty}(t_1,t_2)} \iint_{{{\Om}} \times (t_1,t_2)} |[u-w]_h| \ dz.
  \end{array}
 \end{equation*}
\end{lemma}
\begin{proof}
%
%
Let us use $\phi(x)\varphi(t)$ as a test function in  \cref{very_weak_solution} solving \cref{main-1} to get 
 \begin{gather*}
  \int_{\Om\times\{t\}} \ddt{[u]_h} (x,t) \phi(x)\varphi(t) \ dx +  \iprod{[\aa(x,t,\nabla u)]_h}{\nabla \phi}(x,t)\varphi(t) \ dx =0.
 \end{gather*}
Integrating over $(t_1,t_2)$, we get
\begin{align*}
     \int_{t_1}^{t_2} \int_{\Om} \frac{d}{dt} \left( [u-w]_h \phi(x) \varphi(t) \right) \ dx \ dt =& -\int_{t_1}^{t_2} \int_{\Om} \iprod{[\aa(x,t,\nabla u)]_h}{\nabla \phi}(x,t) \varphi(t) \ dx \ dt \\
     & - \int_{t_1}^{t_2} \int_{\Om} \ddt{[w]_h}(x,t)  \phi(x) \ dx \varphi(t)  \ dt\\
     &+ \int_{t_1}^{t_2} \int_{\Om} [u-w]_h(x,t) \phi(x) \ddt{\varphi(t)} \ dx \ dt.
\end{align*}
 We estimate each of the terms as follows:
 \begin{equation*}
 \label{7.7.15_app}
  \begin{array}{rcl}
   |\avgs{{[u-w]_h\varphi}}{\phi} (t_2) - \avgs{{[u-w]_h\varphi}}{\phi}(t_1)|
   &\overset{\redlabel{lemma3.4.1a}{a}}{\leq}&  C(\La_1,p)\|\nabla \phi\|_{L^{\infty}{({{\Om}})}} \|\varphi\|_{L^{\infty}(t_1,t_2)} \iint_{{{\Om}} \times (t_1,t_2)}  [(|\nabla u|+|h_0|)^{p-1}]_h \ dz \\
   & & \quad + \|\varphi\|_{L^{\infty}(t_1,t_2)} \int_{t_1}^{t_2} \left|\left[ \iprod{\ddt{w}}{\phi}_{({W^{-1,\pbp}(\Om)},W_0^{1,\pbo}(\Om))}(t)\right]_h\right|  \ dt\\
   && \quad + \|\phi\|_{L^{\infty}{({{\Om}})}} \|\varphi'\|_{L^{\infty}(t_1,t_2)} \iint_{{{\Om}} \times (t_1,t_2)} |[u-w]_h| \ dz \\
      &\overset{\redlabel{lemma3.4.1b}{b}}{=}&  C(\La_1,p)\|\nabla \phi\|_{L^{\infty}{({{\Om}})}} \|\varphi\|_{L^{\infty}(t_1,t_2)} \iint_{{{\Om}} \times (t_1,t_2)} {[(|\nabla u|+|h_0|)^{p-1}]_h} \ dz \\
   & & \quad + \|\varphi\|_{L^{\infty}(t_1,t_2)} \int_{t_1}^{t_2}  \left|\left[ \int_{{{\Om}}} \iprod{\vec{w}}{\nabla \phi}\ dx \right]_h\right| \ dt \\
   && \quad + \|\phi\|_{L^{\infty}{({{\Om}})}} \|\varphi'\|_{L^{\infty}(t_1,t_2)} \iint_{{{\Om}} \times (t_1,t_2)} |[u-w]_h| \ dz.
  \end{array}
 \end{equation*}
 To obtain \redref{lemma3.4.1a}{a}, we made use of assumptions \cref{abounded} and \cref{est3.3} and to obtain \redref{lemma3.4.1b}{b}, we made use of \cref{lemma_lihe_wang} to obtain the representation $\ddt{w} = \dv \vec{w}$ for some $\vec{w} \in L^{\pbp}(0,T; L^{\pbp}(\Om,\RR^n))$. This completes the proof of the lemma.
\end{proof}

\section{A priori estimates for the homogeneous problem}
\label{section5}

In this section, let us fix an exponent $q$ such that 
\begin{equation}
\label{def_q_1}
1<p-\ve_0 < q \leq p-2\be < p-\be < p,
\end{equation}
where $\ve_0 = \ve_0(n,p, b_0,r_0)$ is the exponent described in \ref{not8} and $\be$ is a constant to be chosen sufficiently small later on.

Let us define the following new maximal function for any $\vartheta \in [1,\infty)$ and any $Q \subset \RR^{n+1}$:
\begin{equation}
\label{7.7.4}
\mmn{\vartheta} (f\lsb{\chi}{Q})(x,t) := \sup_{(a,b)\ni t} \sup_{B \ni x} \frac{1}{(b-a)} \int_{a}^{b}  \frac{1}{|B|^\frac{1}{\vartheta}} \|f\lsb{\chi}{Q}\|_{W^{-1,\vartheta}(B)} \lsb{\chi}{Q} \ dt,
\end{equation}
where $\|\cdot\|_{W^{-1,\vartheta}(B)}$ is defined in \cref{lemma_lihe_wang} and the supremum is taken over all parabolic cylinders of the form $B \times (a,b)$ containing the point $(x,t)$.

Given any $f \in L^{{\frac{q}{p-1}}}(0,T; W^{-1,{\frac{q}{p-1}}}(\Om))$, from \cref{lemma_lihe_wang}, we see that there exists a vector field $\vec{f} \in L^{{\frac{q}{p-1}}}(\Om_T,\RR^n)$ satisfying $\|f\|_{L^{{\frac{q}{p-1}}}(0,T; W^{-1,{\frac{q}{p-1}}}(\Om))} = \|\vec{f}\|_{L^{{\frac{q}{p-1}}}(\Om_T,\RR^n)}$. Using this, we get
\begin{equation}
\label{est5.3}
\begin{array}{rcl}
\left\|\mmn{\frac{q}{p-1}} (f\lsb{\chi}{\Om_T})\right\|_{L^{\pbp}(\RR^{n+1})} & \le & \left\|\sup_{(a,b)\ni t}  \sup_{B \ni x}\fint_a^b\frac{1}{|B|^{\frac{p-1}{q}}} \|\vec{f} \lsb{\chi}{\Om}\|_{L^{\frac{q}{p-1}}(B,\RR^n)}  \lsb{\chi}{[0,T]}\ ds\right\|_{L^{\pbp}(\RR^{n+1})}\\
& \le & \left\|\sup_{(a,b)\ni t}  \sup_{B \ni x}\left(\fint_a^b\frac{1}{|B|} \|\vec{f} \lsb{\chi}{\Om_T}\|_{L^{\frac{q}{p-1}}(B,\RR^n)}^\frac{q}{p-1}  \ ds\right)^\frac{p-1}{q}\right\|_{L^{\pbp}(\RR^{n+1})}\\
& = & \left\|\mm(|\vec{f}|^{\frac{q}{p-1}}\lsb{\chi}{\Om_T})^{\frac{p-1}{q}} \right\|_{L^{\pbp}(\RR^{n+1})} \\
& \overset{\redlabel{eq5.3.a}{a}}{\apprle} & \left\| \vec{f}\lsb{\chi}{\Om_T}\right\|_{L^{\pbp}(\Om_T,\RR^n)}\\
& \overset{\redlabel{eq5.3.b}{b}}{=} & \left\|f\right\|_{L^{\pbp}(0,T;W^{-1,\pbp}(\Om))}.
\end{array}
\end{equation}
To obtain \redref{eq5.3.a}{a}, we made use of the standard strong maximal function bound from \cref{max_bound} along with the observation $\frac{p-\be}{q} > 1$ and to obtain \redref{eq5.3.b}{b}, we made use of \cref{lemma_lihe_wang}.

Let $u$ and $w$ be as in \cref{main_theorem_1}, then define the following functions:
\begin{gather}
v(z) =  u(z) - w(z) \txt{and} \vh(z) = [u-w]_h(z), \txt{for} z \in \RR^{n+1}, \label{def_v_app}
\end{gather}
where $[\cdot]_h(z)$ denotes the usual Steklov average defined in \cref{stek1}. From \cref{main-1} and \cref{time_average},  we see that 
\[\vh \xrightarrow{h \searrow 0} v \txt{and}\  v(z) = 0 \ \ \text{for} \ z \in\pa_p (\Om \times (0,T)).\]
In subsequent calculations, we extend $\vh$ by zero outside $\Om \times [0,\infty)$.

\subsection{Construction of test function} 

Let us define the following function:
\begin{equation}
\label{def_g_app}
\begin{array}{ll}
g(z) & := \max \left\{\mm\lbr\left[ |\nabla v|^q + (|\nabla u|+|h_0|)^q  +|\nabla w|^q\right] \lsb{\chi}{\Om_T} \rbr^{\frac1{q}}(z), \mmn{\frac{q}{p-1}}(w' \lsb{\chi}{\Om_T})^{\frac{1}{p-1}}(z) \right\},
\end{array}	
\end{equation}
where $v$ is defined in \cref{def_v_app}, $u$ and $w$ are as in the hypothesis of \cref{main_theorem_1} and $q$ is from \cref{def_q_1}. We then have the following estimate for $g$:
\begin{equation}
\label{max_g_fde.3}
\begin{array}{rcl}
\| g\|_{L^{p-\be}(\RR^{n+1})}^{p-\be} & \overset{\text{\cref{max_bound}}}{\apprle} &  \|(|\nabla u|+|h_0|)\lsb{\chi}{\Om_T} \|_{L^{p-\be}(\RR^{n+1})}^{p-\be}+ \| |\nabla w|\lsb{\chi}{\Om_T}\|_{L^{p-\be}(\RR^{n+1})}^{p-\be} \\
&& \hfill+ \|\mmn{\frac{q}{p-1}} (w'\lsb{\chi}{\Om_T}) \|_{L^{\pbp}(\RR^{n+1})} \\
& \overset{\cref{est5.3}}{\apprle} & \|(|\nabla u| +|h_0|)\lsb{\chi}{\Om_T} \|_{L^{p-\be}(\RR^{n+1})}^{p-\be}+ \| |\nabla w|\lsb{\chi}{\Om_T}\|_{L^{p-\be}(\RR^{n+1})}^{p-\be} \\
&& \hfill + \|w'\|_{L^{\pbp}(0,T; W^{-1,\pbp}(\Om))}^{\pbp}.
\end{array}
\end{equation}

For a fixed $\la >0$, let us define the \emph{good set} by 
\begin{equation*}
\label{elambda}
\elam := \{ (x,t) \in \RR^{n+1} : g(x,t) \leq \la\},
\end{equation*}
and apply \cref{whitney_decomposition} and \cref{partition_unity} with $\mathbb{E} = \elam^c$ to get a covering of $\elam^c$.  Recall that the intrinsic scaling is of the form 
\begin{equation}
\label{intrinsic_scaling}
\gamma:=\la^{2-p} \txt{and} Q_j(x,t) = B_{r_j}(x) \times (t_j - \gamma r_j^2, t_j + \gamma r_j^2). 
\end{equation}

Now we define the following Lipschitz extension function as follows:
\begin{equation}
\label{lipschitz_extension}
\vlh(z) := v_h(z) - \sum_i \Psi_i(z) (v_h(z) - v_h^i),
\end{equation}
where
\begin{equation}
\label{lipschitz_extension_one}
v_h^i := \left\{ \begin{array}{ll}
                  \frac1{\|\Psi_i\|_{L^1(\frac34Q_i)}}\iint_{\frac34Q_i} v_h(z) \Psi_i (z) \lsb{\chi}{[0,T]} \ dz & \text{if} \ \frac34Q_i \subset \Om \times [0,\infty), \\
                   0 & \text{else}.
                  \end{array}\right.
\end{equation}
Note that $u-w = 0$ on $\pa \Om \times [0,T]$, which enables us to  switch between $\lsb{\chi}{[0,T]}$ and $\lsb{\chi}{\Om_T}$ without affecting the calculations.

Even though $v_h(x,0) \neq 0$ in general, nevertheless the following observations regarding the initial boundary values hold:
\begin{itemize}
\item The initial condition $(u-w)(x,0) = 0$ is to be understood in the sense $[u-w]_h(\cdot,0) \xrightarrow{h \searrow 0} 0\  \text{in}\  L^2 (\Om).$
\item For $(x,0) \in \elam$, we have $\vlh(x,0) =v_h(x,0)$.
\item For $(x,0) \notin \elam$, we have $\vlh(x,0) = 0$ by using \cref{lipschitz_extension_one}.
\item From \cref{time_average}, we see that $\vlh(z) \xrightarrow{h \searrow 0} \vl(z)$ almost everywhere. 
\end{itemize}

For the rest of this section, let us denote \[2\rho := \diam(\Om).\]

\subsection{Bounds of \texorpdfstring{$\vlh$}.}
Let us now prove some estimates on the test function constructed in \cref{lipschitz_extension}.

\begin{lemma}
\label{lemma3.7}
 For any $z \in \elam^c$, we have
 \begin{equation}
  \label{bound_v_l_h}
  |\vlh(z)| \apprle_{(n,p,q,\Lambda_1,b_0,r_0)} \rho \la.
 \end{equation}
\end{lemma}
\begin{proof}
By construction of the extension in \cref{lipschitz_extension},  for $z \in \elam^c$, $\vlh(z) = \sum_j \Psi_j(z) v_h^j$ with $v_h^j = 0$ whenever $\frac34Q_j \nsubseteq \Om \times [0,\infty)$. Making use of \descref{W12}, that \cref{bound_v_l_h} follows if there holds
\begin{equation}
\label{claim_bound}
|v_h^j| \apprle_{(n,p,q,\Lambda_1,b_0,r_0)}\rho\la. 
\end{equation}
Note that we only have to consider the case $\frac34Q_j \subset \Om \times [0,\infty)$, which automatically implies $\frac34r_{j} \leq \rho$. Thus we proceed as follows:
 let us define the following constant $k_0 := \min\{ \tk_1,\tk_2\}$ where $\tk_1$ and $\tk_2$ satisfy
   \begin{equation}\label{def_k_0}
   \begin{array}{c}
2^{\tk_1 - 1} r_j < \rho \leq 2^{\tk_1} r_j, \\ 
2^{\tk_2 -1} Q_j \subset  \Om \times [0,\infty)  \ \text{  but  } \ 2^{\tk_2}Q_j \nsubseteq \Om \times [0,\infty).
\end{array}
   \end{equation} 
   The idea is to gradually enlarge $\frac34Q_j$ until it goes outside $\Om \times [0,\infty)$. As a consequence, we consider the following two subcases, first case where $2^{\tk_1}Q_j$ crosses the lateral boundary first, and second case when $2^{\tk_2} Q_j$ crosses the initial boundary first. 
   Note that $k_0$ denotes the first scaling exponent under which either  $2^{k_0}r_j \geq  \rho$ occurs  or $2^{k_0}Q_j$ goes outside $\Om \times [0,\infty)$. 

   Since we only consider the case $\frac34Q_i \subset \Om \times [0,\infty)$, using triangle inequality, we get
   \begin{equation}
   \label{est_2_1} 
   \begin{array}{rcl}
    |v_h^j| & {\apprle} & \sum_{m=0}^{k_0 -2} \lbr[[] \avgs{{[u-w]_h}\lsb{\chi}{Q_{\rho,s}(\mfz)}}{2^m Q_j} - \avgs{{[u-w]_h\lsb{\chi}{Q_{\rho,s}(\mfz)}}}{2^{m+1}Q_j}\rbr[]] + \avgs{{[u-w]_h\lsb{\chi}{Q_{\rho,s}(\mfz)}}}{2^{k_0-1}Q_j} \\ 
    &=: & \sum_{m=0}^{k_0-2}S_1^m + S_2.
   \end{array}
  \end{equation}
    We shall now estimate $S_1^m$ and $S_2$  as follows:
  \begin{description}[leftmargin=*]
  \item[Estimate of $S_1^m$:] Note that  $2^{m+1}Q_j \subset \Om\times [0,\infty)$. Thus applying \cref{lemma_crucial_1} for any $\mu \in C_c^{\infty}(B_{2^{m+1}r_j}(x_j))$ with $| \mu(x)| \leq \frac{C(n)}{(2^{m+1}r_j)^{n}}$ and $|\nabla \mu(x)| \leq \frac{C(n)}{(2^{m+1}r_j)^{n+1}}$, we get
   \begin{equation}
   \label{S_1_1.1}
    \begin{array}{rcl}
     S_1^m   
     & \apprle &  (2^{m+1}r_j) \lbr \fiint_{2^{m+1}Q_j} |\nabla [u-w]_h|^q \lsb{\chi}{\Om_T}\ dz\rbr^{\frac1{q}} \\
     & & \quad + (2^{m+1}r_j) \lbr \sup_{t_1,t_2 \in {2^{m+1}I_j\cap [0,T]}} \left|\frac{\avgs{{[u-w]_h}}{\mu}(t_2)-\avgs{{[u-w]_h}}{\mu}(t_1)}{2^{m+1}r_j} \right|^q\rbr^{\frac1{q}}\\
     & \overset{\text{\descref{W4}}}{\apprle} & (2^{m+1}r_j) \la + (2^{m+1}r_j) \lbr \sup_{t_1,t_2 \in {2^{m+1}I_j\cap [0,T]}} \left|\frac{\avgs{{[u-w]_h}}{\mu}(t_2)-\avgs{{[u-w]_h}}{\mu}(t_1)}{2^{m+1}r_j} \right|^q\rbr^{\frac1{q}}.
    \end{array}
   \end{equation}
   
      Since $B_{2^{m+1}r_j}(x_j) \subset \Om$, we can apply \cref{lemma_crucial_2_app} with the test function $\phi(x)=\mu(x)$ and $\varphi(t) = 1$, which implies that for any $t_1,t_2 \in \frac34 I_j\cap[0,T]$, 
      \begin{equation}
      \label{est5.15}
      \begin{array}{r@{}c@{}l}
      \hspace*{-0.5cm}|\avgs{{[u-w]_h}}{\mu}(t_2)-\avgs{{[u-w]_h}}{\mu}(t_1)| &\  {\apprle} & \   \|\nabla \mu \|_{L^{\infty}} \iint_{2^{m+1}Q_j}   \left[(|\nabla u|+|h_0|)^{p-1}\right]_h   \lsb{\chi}{[0,T]}\ dz \\
      && \ \ \ + \underbrace{\int_{t_j - \gamma (2^{m+1} r_j)^2}^{t_j + \gamma (2^{m+1} r_j)^2} \left|\left[ \int_{B_{2^{m+1}r_j}(x_j)} \iprod {\vec{w}(x,t)}{ \nabla \mu(x)} \ dx \right]_h \right|\ dt}_{J}.
    \end{array}
      \end{equation}
The first term on the right hand side of \cref{est5.15} can be controlled using \descref{W4} and \cref{intrinsic_scaling} as
\begin{equation}
\label{est5.16}
\|\nabla \mu \|_{L^{\infty}} \iint_{2^{m+1}Q_j}   \left[(|\nabla u|+|h_0|)^{p-1}\right]_h   \lsb{\chi}{[0,T]}\ dz \apprle 2^{m+1} r_j \la.
\end{equation}
We now estimate $J$ as follows:
\begin{equation}
\label{S_1_2.1_app}
 \begin{array}{ll}
  J & \overset{\redlabel{7.8.7.a}{a}}{\apprle} \|\nabla \mu\|_{L^{\infty}} |B_{2^{m+1}r_j}|^{\frac{q-p+1}{q}} \int_{t_j - \gamma (2^{m+1} r_j)^2}^{t_j + \gamma (2^{m+1} r_j)^2} \left[ \sum_{i=1}^n\lbr\int_{B_{2^{m+1}r_j}(x_j)} |\vec{w}_i(x,t)|^{\frac{q}{p-1}} \ dx \rbr^{\frac{p-1}{q}}\right]_h\ dt \\
    & \overset{\redlabel{7.8.7.b}{b}}{\apprle}  \|\nabla \mu\|_{L^{\infty}} |B_{2^{m+1}r_j}|^{\frac{q-p+1}{q}} \int_{t_j - \gamma (2^{m+1} r_j)^2}^{t_j + \gamma (2^{m+1} r_j)^2} \left[ \left\| \ddt{\vec{w}(\cdot,t)}\right\|_{W^{-1,\frac{q}{p-1}}(B_{2^{m+1}r_j})}\right]_h\ dt \\
  & \overset{\redlabel{7.8.7.c}{c}}{\apprle}  \gamma 2^{m+1} r_j  \fint_{2^{m+2}I_j} \frac{1}{|B_{2^{m+1}r_j}|^{\frac{p-1}{q}}} \left\| \ddt{\vec{w}(\cdot,t)}\right\|_{W^{-1,\frac{q}{p-1}}(B_{2^{m+1}r_j})}\ dt \\
  & \overset{\redlabel{7.8.7.d}{d}}{\apprle} 2^{m+1} r_j \gamma \la^{p-1} = 2^{m+1} r_j \la.
 \end{array}
\end{equation}
To obtain \redref{7.8.7.a}{a}, we made use of H\"older's inequality, to obtain \redref{7.8.7.b}{b}, we used \cref{lemma_lihe_wang}, to obtain \redref{7.8.7.c}{c}, we used \cref{intrinsic_scaling} along with \cref{time_average} and finally to obtain \redref{7.8.7.d}{d}, we made use of \descref{W4}.

   Thus combining \cref{S_1_2.1_app} and \cref{est5.16} into \cref{est5.15} and finally making use of \cref{S_1_1.1} along with \descref{W4} and $\gamma := \la^{2-p}$, we get
\begin{equation}
 \label{bound_S_1_m}
 S_1^m \apprle 2^{m+1} r_j \lbr \la^q + (\la^{p-1}\gamma)^q   \rbr^{\frac{1}{q}} \apprle 2^{m+1}r_j \la.
\end{equation}

  \item[Estimate of $S_2$:] For this term, we know that $2^{k_0-1}Q_j \notin \Om \times [0,\infty)$, which implies $2^{k_0-1}Q_j$ crosses either the lateral boundary $\pa \Om \times [0,\infty)$ or the initial boundary $\Om \times \{t=0\}$. We will consider both the cases separately and estimate $S_2$ as follows:

  \hspace{\parindent}\emph{In the case $2^{k_0-1}Q_j$ crosses the lateral boundary $\pa \Om \times [0,\infty)$ first}, we can directly apply Poincar\'e's inequality to obtain 
  \begin{equation}
\label{second_case.1}
          \fiint_{2^{k_0-1}Q_j} [u-w]_h\lsb{\chi}{\Om_T} \ dz  \apprle ( 2^{k_0} r_j)\lbr \fiint_{ 2^{k_0}Q_j} |\nabla [u-w]_h|^q \lsb{\chi}{\Om_T} \ dz \rbr^{1/q} 
           \overset{\redlabel{4.20.a}{a}}{\apprle} \rho \la.
\end{equation}
   To obtain \redref{4.20.a}{a}, we made use of \descref{W4} along with $2^{k_0-2}r_j \leq \rho$ given by \cref{def_k_0}.

 \hspace{\parindent} \emph{In the case $2^{k_0} Q_j$ crosses the initial boundary $\Om \times \{t=0\}$ first}, by enlarging the cylinder to $2^{k_1+1}Q_j$, we can find a cut-off function $\xi(x,t)$ such that 
 \[
 \spt\xi(x,t) \subset 2^{k_1+1}Q_j \cap \RR^n \times (-\infty,0) \txt{and} \|\xi \|_{L^{\infty}} \apprle_n \frac{\|\xi\|_{L^1}}{|2^{k_1+1}Q_j|},
 \]
by which and along with the fact that $v_h(z) \lsb{\chi}{[0,T]} = 0$ on $\RR^n \times (-\infty,0)$, we get $\avgs{v_h\lsb{\chi}{[0,T]}}{\xi}=0$. Thus applying \cref{lemma_crucial_1}, we get
\begin{equation}
\label{second_case}
     \begin{array}{rcl}
          \fiint_{2^{k_0+1}Q_j} |v_h (z)|\lsb{\chi}{[0,T]} \ dz\  & = &\fiint_{ 2^{k_0+1}Q_j} \abs{v_h(z)\lsb{\chi}{[0,T]} - \avgs{v_h\lsb{\chi}{[0,T]}}{\xi}} \ dz\\ 
     & \apprle&  (2^{k_0+1}r_j) \lbr \fiint_{2^{k_0+1}Q_j} |[\nabla (u-w)]_h|^q  \lsb{\chi}{[0,T]}\ dz  \rbr^{\frac1{q}} \\
     &  & + (2^{k_0+1}r_j)\lbr  \sup_{t_1,t_2 \in {2^{k_0+1}I_j\cap [0,T]}} \left|\frac{\avgs{{[u-w]_h }}{\mu}(t_2)-\avgs{{[u-w]_h }}{\mu}(t_1)}{2^{k_0+1}r_j} \right|^q\rbr^{\frac1{q}}\\
     & \overset{\redlabel{4.21.a}{a}}{\apprle}& 2^{k_0+1}r_j \la \overset{\redlabel{4.21.b}{b}}{\apprle} \rho \la.
     \end{array}
\end{equation}
To bound the first term in \redref{4.21.a}{a}, we made use of \descref{W1} along with \descref{W4} and to bound the second term in \redref{4.21.a}{a}, we proceed exactly as in \cref{est5.15} and finally to 
 to obtain \redref{4.21.b}{b}, we made use of \cref{def_k_0}. 
 
 Combining \cref{second_case.1} and \cref{second_case}, we get
\begin{equation}
\label{bound_S_2}
S_2 \apprle \rho \la. 
\end{equation}
\end{description}
Thus combining \cref{bound_S_1_m} and \cref{bound_S_2} into \cref{est_2_1}, we get
\begin{equation*}
|v_h^j| \leq \sum_{m=0}^{k_0-2} S_1^m + S_2  \apprle  \la \lbr \sum_{m=0}^{k_0-2} 2^{m+1}r_j + \rho\rbr \overset{\cref{def_k_0}}{\apprle} \rho \la.
\end{equation*}
This completes the proof of the Lemma. 
\end{proof}

Now we prove a sharper estimate. 
\begin{lemma}
\label{lemma3.8}
 For any $j \in A_i$, the following improved estimate holds:
 \begin{equation*}
  \label{3.26}
  |v_h^i - v_h^j| \apprle_{(n,p,q,\Lambda_1,b_0,r_0)} \min\{\rho, r_i\} \la.
 \end{equation*}
\end{lemma}

\begin{proof}
We only have to consider the case $r_i \leq \rho$ because if $\rho \leq r_i$, we can directly use \cref{lemma3.7} to get the required conclusion.
\begin{description}
\item[We first consider the case that $\frac34Q_i$  intersect the initial or lateral boundary.]

     \emph{Initial Boundary Case $\frac34Q_i \subset \Om \times \RR$:} Without loss of generality, we can assume $2Q_i \subset \Om \times \RR$. We now pick a 
     \[\xi(x,t)\in C_c^{\infty}(\RR^{n+1}) \txt{with} \spt(\xi) \subset 2B_i \times (-\infty,0).\] We extend  $u-w = 0$ on $2B_i \times (-\infty,0)$, which implies $\avgs{{[u-w]_h\lsb{\chi}{[0,T]}}}{\xi}=0$. Thus we get
     \begin{equation*}\label{lemma3.8.0}
          \begin{array}{rcl}
               |v_h^i| & \apprle & \fiint_{2Q_i} \abs{[u-w]_h\lsb{\chi}{[0,T]}- \avgs{{[u-w]_h\lsb{\chi}{[0,T]}}}{\xi}} \ dz \\
               & \overset{\redlabel{3.80.a}{a}}{\apprle} & r_i \lbr \fiint_{2Q_i} |\nabla v_h|^q \lsb{\chi}{[0,T]}\ dz + \sup_{t_1,t_2 \in 2I_i\cap [0,T]} \left| \frac{\avgs{{v_h\lsb{\chi}{[0,T]}}}{\mu}(t_2)-\avgs{{v_h\lsb{\chi}{[0,T]}}}{\mu}(t_1)}{r_i}\right|^q \rbr^{\frac{1}{q}} \\
& \overset{\redlabel{3.80.b}{b}}{\apprle} & r_i \la. 
          \end{array}
     \end{equation*}
To obtain \redref{3.80.a}{a}, we made use of \cref{lemma_crucial_1} and to obtain \redref{3.80.b}{b}, we proceed similarly to how \cref{est5.15} was estimated.

     \emph{Lateral Boundary Case $\frac34Q_i \cap  (\Om\times \RR)^c \neq \emptyset$:} In this case, using  \cref{sobolev-poincare} along with \descref{W4}, we get
     \begin{equation}
\label{lemma3.8.1}
 \begin{array}{ll}
  |v_h^i| & \apprle r_i \lbr \fiint_{2Q_i} \left| \frac{[u-w]_h \lsb{\chi}{[0,T]}}{r_i} \right|^q \ dz \rbr^{\frac{1}{q}} \apprle r_i \lbr \fiint_{2Q_i} \left| {\nabla [u-w]_h} \right|^q\lsb{\chi}{[0,T]} \ dz \rbr^{\frac{1}{q}} \apprle r_i \la. 
 \end{array}
\end{equation}
From  \cref{lemma3.8.1} and \cref{claim_bound}, we see that the lemma is proved  provided either $v_h^j=0$ or $v_h^i=0$.

\item[Now let us consider the case $\frac34Q_i \subset \Om\times [0,\infty)$.]

 From the definition of $v_h^i$ in \cref{lipschitz_extension_one}, triangle inequality and \descref{W10}, we get
\begin{equation}
 \label{3.29}
 \begin{array}{rcl}
 |v_h^i - v_h^j| 
 & \apprle&  \frac{|\frac34Q_i|}{|\frac34Q_i \cap \frac34Q_j|}\fiint_{\frac34Q_i} \abs{v_h(z)\lsb{\chi}{[0,T]} - v_h^i} \ dz + \frac{|\frac34Q_j|}{|\frac34Q_i \cap \frac34Q_j|}\fiint_{ \frac34Q_j}\abs{v_h(z)\lsb{\chi}{[0,T]} - v_h^j} \ dz \\
 & \apprle &\fiint_{\frac34Q_i} \abs{v_h(z)\lsb{\chi}{[0,T]} - v_h^i} \ dz + \fiint_{ \frac34Q_j} \abs{v_h(z)\lsb{\chi}{[0,T]} - v_h^j} \ dz.
 \end{array}
\end{equation}

We  now apply H\"older's inequality followed by \cref{lemma_crucial_1} with $\mu \in C_c^{\infty}\lbr\frac34B_i\rbr$ satisfying $|\mu(x)| \apprle \frac{1}{r_i^n}$ and $|\nabla \mu(x)| \apprle \frac{1}{r_i^{n+1}}$ to estimate \cref{3.29} as follows:
\begin{equation}
 \label{3.30}
 \begin{array}{rcl}
 \fiint_{\frac34Q_i} | v_h(z)\lsb{\chi}{[0,T]} - v_h^i| \ dz 
 & \apprle &  r_i \lbr \fiint_{\frac34Q_i} |\nabla v_h|^q \lsb{\chi}{[0,T]}\ dz \rbr^{\frac{1}{q}} \\
 && +\  r_i \lbr\sup_{t_1,t_2 \in \frac34I_i\cap[0,T]} \left|\frac{\avgs{{[u-w]_h}}{\mu}(t_2) - \avgs{{[u-w]_h}}{\mu}(t_1)}{r_i} \right|^q \rbr^{\frac{1}{q}}.
 \end{array}
\end{equation}
The first term on the right of \cref{3.30} can be controlled using \descref{W4} and the second term can be controlled similarly as \cref{est5.15}. Thus we get
\[
\fiint_{\frac34Q_i} | v_h(z)\lsb{\chi}{[0,T]} - v_h^i| \ dz\apprle r_i \la.
\]
\end{description}
This completes the proof of the Lemma.
\end{proof}

\subsection{Bounds on derivatives of \texorpdfstring{$\vlh$}.}

Using \cref{lemma3.7} and \cref{lemma3.8}, in the same spirit of \cite{adimurthi2018gradient}, we can obtain the following estimates:
\begin{lemma}
\label{lemma3.9}
 Given any  $z \in \elam^c$, we have $z \in \frac34Q_i$ for some $i \in \NN$. Then there holds
 \begin{equation}
  \label{3.34}
  |\nabla \vlh(z)| \leq C_{(n,p,q,\Lambda_1,b_0,r_0)} \la.
 \end{equation}
\end{lemma}

\begin{lemma}
 \label{lemma3.10.1}
 Let $z \in \elam^c$ and $\ve \in (0,1]$ be any number, then $z \in \frac34Q_i$ for some $i \in \NN$ from \descref{W1}. There exists a constant $C =  C_{(n)}$ such that the following holds:
 \begin{gather}
  |\vlh(z)|  \leq C \fiint_{4Q_i}|v_h(\tz)| \lsb{\chi}{[0,T]}\ d\tz \leq  \frac{Cr_i\la}{\varepsilon} + \frac{C\varepsilon}{\la r_i} \fiint_{4Q_i}|v_h(\tz)|^2 \lsb{\chi}{[0,T]}\ d\tz, \nonumber\\ 
  |\nabla \vlh(z)| \leq C \frac{1}{r_i} \fiint_{4Q_i}|v_h(\tz)| \lsb{\chi}{[0,T]}\ d\tz \leq  \frac{C \la}{\varepsilon} + \frac{C\varepsilon}{\la r_i^2} \fiint_{4Q_i}|v_h(\tz)|^2 \lsb{\chi}{[0,T]}\ d\tz, \nonumber \\
  |\vlh(z)| \leq C \lbr \min\{ \rho, r_i\} \la + |v_h^i| \rbr \leq  C \lbr \frac{ r_i\la}{\ve} + \frac{\ve}{r_i \la} |v_h^i|^2 \rbr,  \label{lemma3.10_bound3}\\
  |\nabla \vlh(z)| \leq C \frac{\la}{\ve}.\label{lemma3.10_bound4}
 \end{gather}
\end{lemma}


\begin{lemma}
\label{lemma3.11}
Let $z \in \elam^c$ and $\ve \in (0,1]$ be any number, then $z \in \frac34Q_i$ for some $i \in \NN$ from \descref{W1}. There exists a constant $C =  C_{(n,p,q,\Lambda_1,b_0,r_0)}$ such that the following estimates for the time derivative of $\vlh$ holds:
 \begin{equation}\label{lemma3.11.bound2}\begin{array}{rcl}
 |\pa_t \vlh(z)| & \leq & C\frac{1}{\la^{2-p} r_i^2} \fiint_{4Q_i} |v_h(\tz)| \lsb{\chi}{[0,T]}\ d\tz, \\
 |\pa_t \vlh(z)| & \leq & C \frac{1}{\la^{2-p} r_i^2} \min\{r_i,\rho\} \la. 
 \end{array}\end{equation}
\end{lemma}
\begin{lemma}
\label{lemma3.12}
For any $\vartheta \geq 1$,  we have the following bound:
 \begin{equation*}
  \label{3.56}
  \iint_{\Om_T\setminus\elam } |\vlh(z)|^{\vartheta} \ dz \apprle_{(n,p,q,\Lambda_1,b_0,r_0,\vartheta)} \iint_{\Om_T\setminus\elam } |v_h(z)|^{\vartheta} \lsb{\chi}{[0,T]}\ dz.
 \end{equation*}
\end{lemma}

\begin{lemma}
\label{lemma3.14}
 For any $1\le \vartheta \leq q$, there holds 
 \begin{equation*}
  \iint_{\Om_T \setminus \elam} |\pa_t \vlh(z)  (\vlh(z) - v_h(z))|^{\vartheta} \ dz \apprle_{(n,p,q,\Lambda_1,b_0,r_0, \vartheta)} \la^{\vartheta p} |\RR^{n+1} \setminus \elam|.
 \end{equation*}
\end{lemma}

\subsection{Lipschitz continuity of the test function}

We shall now prove the Lipschitz continuity of $\vlh$ on $\mch := \RR^n \times [0,T]$. 

\begin{lemma}
\label{lemma3.15}
The function $\vlh$ from  \cref{lipschitz_extension} is $C^{0,1}(\mch)$ with respect to the parabolic metric given by 
\[
d_{\la}(z_1,z_2)  := \max \left\{ |x_2-x_1|, \sqrt{\la^{p-2} |t_2-t_1|} \right\}.
\]

\end{lemma}

\begin{proof} 
Let us consider a parabolic cylinder $Q_{r}(z) := Q_{r, \gamma r^2} (z) := Q = B \times I$ for some $z \in \mch$ and $r>0$ (recall $\gamma = \la^{2-p}$ which is the intrinsic scaling from \cref{whitney_decomposition}).  To prove the lemma, we make use of \cref{metric_lipschitz} and obtain the following bound:
 \begin{equation*}
 \label{bound_I_r}
I_r(z) :=   \fiint_{Q \cap \mch} \left|\frac{\vlh(\tz) - \avgs{\vlh}{Q\cap\mch}}{r}\right| \ d\tz \leq {o}(1),
 \end{equation*}
where ${o}(1)$ denotes a constant independent of $z \in \mch$ and $r>0$ only. We will split the proof into several subcases and proceed as follows:

\begin{description}
\item[Case $2Q \subset \elam^c$:]  In this case, from \descref{W3}, we see that $z \in \frac34Q_i$ for some $i \in \NN$. From the construction in \cref{lipschitz_extension}, we see that $\vlh\in C^{\infty}(\elam^c)$ which combined with the mean value theorem gives
  \begin{equation*}
   \label{3.71}
I_r(z)  \apprle   \frac{1}{r} \fiint_{Q \cap \mch} \fiint_{Q \cap \mch}  |\vlh(\tz_1) - \vlh(\tz_2)| \ d\tz_1 \ d\tz_2
 \apprle  \sup_{\tz \in Q \cap \mch}\lbr  |\nabla\vlh(\tz)| + \la^{2-p} r |\pa_t \vlh(\tz)|\rbr .
  \end{equation*}
  Let us pick some  $\tz_0 \in Q \subset \elam^c$,  then $\tz_0 \in Q_j$ for some $j \in \NN$. Thus we can make use of \cref{3.34} and \cref{lemma3.11.bound2} to get
  \begin{equation}
  \label{3.72}
   |\nabla\vlh(\tz_0)| + \la^{2-p} r |\pa_t \vlh(\tz_0)| \apprle \la + \la^{2-p} r \frac{1}{\la^{2-p} r_j^2} r_j \la.
  \end{equation}

  In \cref{3.72}, we need to understand the relation between $r_j$ and $r$. To this end, from $2Q \subset \elam^c$, we see that 
\begin{equation}
 \label{3.74}
 r \leq  d_{\la} (\tz_0, \elam) \leq d_{\la} (\tz_0, z_j) + d_{\la} (z_j , \elam) \leq r_j + 16r_j = 17r_j. 
\end{equation}

Combining \cref{3.72} and \cref{3.74}, we get
\begin{equation*}
\label{3.73}
|\nabla\vlh(\tz_0)| + \la^{2-p} r |\pa_t \vlh(\tz_0)| \apprle \la.
\end{equation*}

\item[Case $2Q \nsubseteq \elam^c$:] In this case, we shall split the proof into three subcases:
\begin{description}[leftmargin=*]
\item[Subcase $2Q \subset \RR^n \times {(-\infty,T]}$ or  $2Q \subset \RR^n \times [0,\infty)$:]
In this situation, it is easy to see that the following holds:
\begin{equation}
  \label{3.70}
  |Q \cap \mch| \apprge |Q|. 
 \end{equation}
 We apply   triangle inequality and estimate $I_r(z)$ by 
 \begin{equation}
  \label{3.81}
  \begin{array}{ll}
   I_r(z) & \leq \fiint_{Q \cap \mch} \left| \frac{\vlh(\tz) - v_h(\tz)}{r}\right|+ \left| \frac{v_h(\tz) - \avgs{v_h}{Q\cap \mch}}{r}\right| + \left| \frac{\avgs{v_h}{Q\cap \mch} - \avgs{\vlh}{Q \cap \mch}}{r}\right| \ d\tz \\
   & \leq 2J_1 + J_2,
  \end{array}
 \end{equation}
 where we have set
 \begin{gather}
  J_1:= \fiint_{Q \cap \mch} \left| \frac{\vlh(\tz) - v_h(\tz)}{r}\right| \ d\tz  \txt{and} J_2 := \fiint_{Q \cap \mch} \left| \frac{v_h(\tz) - \avgs{v_h}{Q\cap \mch}}{r}\right|\ d\tz.\label{def_J_1_2}
 \end{gather}
We now estimate each of the terms of \cref{def_J_1_2} as follows:
\begin{description}[leftmargin=*]
\item[Estimate for $J_1$:]  From \cref{lipschitz_extension}, we get
\begin{equation}
\label{3.82}
\begin{array}{ll}
J_1 &\apprle \sum_{i\in \NN} \frac{1}{|Q\cap\mch|} \iint_{Q \cap\mch\cap \frac34Q_i} \left| \frac{v_h(\tz)\lsb{\chi}{[0,T]} - v_h^i}{r}\right| \ d\tz.
\end{array}
\end{equation}
Let us fix an $i \in \NN$ and take two points $\tz_1 \in Q \cap \frac34Q_i$ and $\tz_2 \in \elam \cap 2Q$. Let $z_i$ denote the center of $\frac34Q_i$,  making use of  \descref{W2} along with the trivial bound $ d_{\la}(\tz_1, \tz_2) \leq  4r$ and $d_{\la} (z_i, \tz_1) \leq 2r_i$,  we get
\begin{equation}
\label{3.83.1}
16r_i =d_{\la}(z_i,\elam) \leq d_{\la} (z_i, \tz_1) + d_{\la} (\tz_1, \tz_2) \leq 2r_i + 4r  \ \Longrightarrow \ 2r_i \leq r.
\end{equation}

Note that \cref{3.70} holds and thus summing over all $i \in \NN$ such that  $Q \cap\mch\cap \frac34Q_i \neq \emptyset$ in \cref{3.82} and making use of  \cref{3.83.1}, we get
\begin{equation*}
 \label{bound_J_1}
 J_1 \apprle \sum_{i\in\NN } \frac{|\frac34Q_i|}{|Q\cap\mch|} \fiint_{\frac34Q_i} \left| \frac{v_h(\tz)\lsb{\chi}{[0,T]} - v_h^i}{r}\right| \ d\tz
  \overset{\redlabel{4.53.a}{a}}{\apprle}  \sum_{i\in \NN}  \fiint_{\frac34Q_i} \left| \frac{v_h(\tz)\lsb{\chi}{[0,T]} - v_h^i}{r_i}\right| \ d\tz 
  \overset{\redlabel{4.53.b}{b}}{\apprle}  \la.
\end{equation*}
To obtain \redref{4.53.a}{a}, we made use of \cref{3.70} and \cref{3.83.1}, to obtain \redref{4.53.b}{b}, we follow the calculation from bounding \cref{3.30}.

\item[Estimate for $J_2$:] Note that $Q \cap \mch$ is another cylinder. \emph{In the case $Q \subset \Om \times \RR$,}  choose a cut-off function $\mu \in C_c^{\infty}(\Om)$ and apply \cref{lemma_crucial_1} to get
\begin{equation*}
\begin{array}{rcl}
J_2  
& \apprle & \lbr \fiint_{Q \cap \mch} |\nabla v_h|^q \lsb{\chi}{\Om_T} + \sup_{t_1,t_2 \in  I} \left| \frac{\avgs{v_h\lsb{\chi}{[0,T]}}{\mu}(t_2) - \avgs{v_h\lsb{\chi}{[0,T]}}{\mu}(t_1)}{r} \right|^q\rbr^{\frac{1}{q}}.
\end{array}
\end{equation*}
Recall that we are in the case $2Q \cap \elam \neq \emptyset$ and $2Q \cap \elam^c \neq \emptyset$. We can now proceed as in \cref{est5.15} to get
\begin{equation}
\label{J_2_bound_first_case}
J_2 \apprle \la. 
\end{equation}

\emph{On the other hand, if $Q \nsubseteq \Om \times \RR$,} then we can apply \cref{sobolev-poincare} directly and make use of the fact that $2Q \cap \elam \neq \emptyset$ to get
\begin{equation*}
\label{bound_J_2_second_case}
J_2 
 \apprle \lbr \fiint_{Q \cap \mch} \left| \nabla v_h(\tz)\lsb{\chi}{[0,T]}\right|^q\ d\tz\rbr^{\frac{1}{q}} \apprle \la.
\end{equation*}

\end{description}

\item[Subcase $2Q \cap \RR^n \times {(-\infty,0]} \neq \emptyset$ and $2Q \cap \RR^n \times [T,\infty)\neq \emptyset$ AND $\gamma r^2 \leq T$:] In this case, we see that $$|Q \cap \mch| \ge |B_1|r^n \times \frac{T}{2}.$$
 We apply   triangle inequality and estimate $I_r(z)$ as we did in \cref{3.81} to get
 \begin{equation*}
  \label{3.81_n}
   I_r(z) 
   \leq 2J_1 + J_2,
 \end{equation*}
 where we have set
 \begin{gather*}
  J_1:= \fiint_{Q \cap \mch} \left| \frac{\vlh(\tz) - v_h(\tz)}{r}\right| \ d\tz  \txt{and} J_2 := \fiint_{Q \cap \mch} \left| \frac{v_h(\tz) - \avgs{v_h}{Q\cap \mch}}{r}\right|\ d\tz.\label{def_J_1_2_n}
 \end{gather*}
 
 We estimate $J_1$ as follows
 \begin{equation*}
 \label{3.82.n}
 \begin{array}{rcl}
 J_1 & \apprle &\sum_{i \in \NN} \frac{|\frac34Q_i|}{|Q\cap\mch|} \fiint_{\frac34Q_i} \left| \frac{v_h(\tz)\lsb{\chi}{[0,T]} - v_h^i}{r}\right| \ d\tz \\
 & \apprle& \frac{r_i^{n+2} \gamma}{r^{n} T} \sum_{i\in \NN}  \fiint_{\frac34Q_i} \left| \frac{v_h(\tz)\lsb{\chi}{[0,T]} - v_h^i}{r_i}\right| \ d\tz\\
 & \overset{\cref{3.83.1}}{\apprle}& \frac{r^{n+2} \gamma}{r^{n} T} \sum_{i\in \NN}  \fiint_{\frac34Q_i} \left| \frac{v_h(\tz)\lsb{\chi}{[0,T]} - v_h^i}{r_i}\right| \ d\tz\\
 & \overset{\redlabel{4.58.a}{a}}{\apprle} &\frac{r^{2} \gamma}{ T} \la\\
 & \overset{\redlabel{4.58.b}{b}}{\apprle} & \la.
 \end{array}
 \end{equation*}
To obtain \redref{4.58.a}{a}, we proceeded similarly to \cref{3.30} and to obtain \redref{4.58.b}{b}, we made use of  $\gamma r^2 \leq T$.

 The estimate of $J_2$ is already obtained in \cref{J_2_bound_first_case} which shows
 \begin{equation*}
 \label{bound_J_2_third_case}
 J_2 \apprle \la. 
 \end{equation*}
 
 \item[Subcase $2Q \cap \RR^n \times {(-\infty,0]} \neq \emptyset$ and $2Q \cap \RR^n \times [T,\infty)\neq \emptyset$ AND $\gamma r^2 > T$:] Using triangle inequality, we get
 \begin{equation*}
  \label{3.91}
  \begin{array}{ll}
  \fiint_{Q \cap \mch}  \left| \frac{\vlh(\tz) - \avgs{\vlh}{Q \cap \mch}}{r} \right|\ d\tz & \apprle \frac{1}{|Q \cap \mch|} \iint_{Q \cap \mch} |\vlh(\tz)| \ d\tz \\
  & \apprle \frac{1}{|Q \cap \mch|} \iint_{Q \cap \mch \cap \elam} |\vlh(\tz)| \ d\tz + \frac{1}{|Q \cap \mch|} \iint_{Q \cap \mch \setminus \elam} |\vlh(\tz)| \ d\tz.
  \end{array}
 \end{equation*}
We have $\vlh = v_h$ on $\elam$.  On $\Om_T \setminus \elam$, we can apply Lemma \ref{lemma3.7} to obtain the following bound:
\begin{equation*}
  \begin{array}{ll}
   \fiint_{Q \cap \mch}  \left| \frac{\vlh(\tz) - \avgs{\vlh}{Q \cap \mch}}{r} \right|\ d\tz & \apprle  \frac{1}{r^n T} \iint_{\Om_T} |v_h(\tz)| \ d\tz +   \frac{1}{|Q \cap \mch|} \iint_{Q \cap \mch	 \setminus \elam} \rho \la\ d\tz \\
   & \apprle \lbr \frac{\gamma}{T}\rbr^{\frac{n}{2}} \frac{1}{T} \| v_h \|_{L^1(\Om_T)} + \rho \la.
  \end{array}
 \end{equation*}
\end{description}
\end{description}
This completes the proof of the Lipschitz regularity of $v_{\lambda,h}$.
\end{proof}

\subsection{Two crucial estimates for the test function}
We shall now prove the first crucial estimate which holds on each time slice. 
\begin{lemma}
 \label{pre_crucial_lemma}
 For any $i \in \NN$ and any $0 < \ve \leq 1$, there exists a positive constant $C{(n,p,q, \Lambda_1,b_0,r_0)}$ such that for  almost every $t \in [0,T]$, there holds
 \begin{equation}
 \label{3.120}
  \left| \int_{\Om} (v(x,t) - v^i) \vl(x,t) \Psi_i(x,t) \ dx \right| \leq C \lbr  \frac{\la^p}{\ve} |4Q_i| + \ve |4B_i| |v^i|^2\rbr. 
 \end{equation}
 
\end{lemma}

\begin{proof}
Let us fix any $t \in [0,T]$,  $i \in \NN$  and take $\Psi_i(y,\tau) \vlh(y,\tau)$ as a test function in \cref{main-1}. Further integrating the resulting expression over $ \left(t_i - \gamma \left(\frac34 r_i\right)^2 , t\right)$ or $(0,t)$ depending on the location of $\frac34Q_i$, along with making use of  the fact that $\Psi_i(y,t_i - \gamma (3r_i/4)^2) = 0$ or $\vlh(y,0) =0$ for $(y,0) \in \elam^c$, we get for  any $a\in \RR$, the equality
\begin{equation}
 \label{3.123}
 \begin{array}{ll}
  \int_{\Om} \lbr[(]  (v_h - a)  \Psi_i \vlh \rbr (y,t) \ dy 
  & = \int_{\max \lbr[\{]t_i - \gamma \left(\frac34 r_i\right)^2,0\rbr[\}]}^t \int_{\Om} \pa_t \left( [u-w]_h  \Psi_i \vlh  - a \Psi_i \vlh \right) (y,\tau) \ dy \ d\tau \\
  & = \int_{\max \lbr[\{]t_i - \gamma \left(\frac34 r_i\right)^2,0\rbr[\}]}^t \int_{\Om} \iprod{[\aa(y,\tau,\nabla u)]_h}{\nabla (\Psi_i \vlh)} \ dy \ d\tau  \\
  &\qquad - \int_{\max \lbr[\{]t_i - \gamma \left(\frac34 r_i\right)^2,0\rbr[\}]}^t \int_{\Om} \ddt{[w]_h} (\Psi_i \vlh) \ dy \ d\tau  \\
  & \qquad +\ \int_{\max \lbr[\{]t_i - \gamma \left(\frac34 r_i\right)^2,0\rbr[\}]}^t \int_{\Om} [u-w]_h \pa_t \lbr \Psi_i \vlh\rbr \ dy \ d\tau  \\
  & \qquad -\  \int_{\max \lbr[\{]t_i - \gamma \left(\frac34 r_i\right)^2,0\rbr[\}]}^t \int_{\Om} a \pa_t \lbr \Psi_i \vlh\rbr \ dy \ d\tau.
 \end{array}
\end{equation}
We can estimate $|\nabla ( \Psi_i \vl)|$ using the chain rule and \descref{W13}, to get
\begin{equation}
 \label{3.126}
 |\nabla ( \Psi_i \vlh)|  \apprle \frac{1}{r_i} |\vl| + |\nabla \vl|.
\end{equation}
Similarly, we can estimate $\left|\pa_t\lbr \Psi_i \vl \right)\right|$ using the chain rule and \descref{W13}, to get
\begin{align}
 \left| \pa_t \lbr \Psi_i \vl\rbr\right| & \apprle  \frac{1}{\gamma r_i^2} |\vl| + |\pa_t \vl|.\label{3.128}
\end{align}
%
%
 Let us take $a=v_h^i$ in the \cref{3.123} followed by letting $h \searrow 0$ and making use of \cref{3.126}, \cref{abounded} and crucially the assumption from \cref{imp_rmk} (more specifically \cref{est3.3}),  we get
 \begin{equation}
  \label{first_1}
  \begin{array}{ll}
   \left| \int_{\Om} \lbr[(] (v - v^i) \Psi_i \vl \rbr (y,t) \ dy \right| & \apprle J_1 + J_2 + J_3,
  \end{array}
 \end{equation}
 where we have set 
 \begin{align}
  J_1& :=  \iint_{\Om_T} (|\nabla u|+|h_0|)^{p-1} \lbr \frac{1}{r_i} |\vl| + |\nabla \vl|\rbr   \lsb{\chi}{\frac34Q_i\cap \Om_T} \ dy \ d\tau, \nonumber \\
  J_2& :=  \int_{0}^T \abs{\iprod{\pa_t w}{\Psi_i \vl}_{(W^{-1,\pbp}(\Om), W_0^{1,\pbo}(\Om))}} d\tau = \int_0^T \left|\int_{\Om} \iprod{\vec{w}}{\nabla(\Psi_i \vl)}\lsb{\chi}{\frac34Q_i\cap \Om_T}\ dx\right| \ d\tau,\nonumber \\
  J_3&:= \iint_{\Om_T} |v-v^i| | \pa_t (\Psi_i \vl)|  \lsb{\chi}{\frac34Q_i\cap \Om_T} \ dy \ d\tau.\label{bound_J_3_1}
 \end{align}

 Let us now estimate each of the terms as follows: 
 \begin{description}
  \item[Bound for $J_1$:] If $\rho \leq r_i$,  we can directly use H\"older's inequality, \cref{lemma3.7}, \cref{lemma3.9} and \descref{W4}, to find that for any $\ve \in (0,1]$, there holds 
  \begin{equation}
   \label{bound_I_1_rho_leq_r_i}
   \begin{array}{ll}
    J_1 & \apprle \lbr \frac{1}{r_i} \rho \la + \la \rbr |Q_i| \lbr \fiint_{16Q_i}  (|\nabla u| + |h_0|)^{q}  \lsb{\chi}{\Om_T}  \ dy \ d\tau\rbr^{\frac{p-1}{q}} \apprle  \frac{\la^p}{\ve} |4Q_i|.
   \end{array}
  \end{equation}
  In the case $r_i \leq \rho$, we make use of \cref{lemma3.10_bound3}, \descref{W4} and \cref{lemma3.9} along with the fact $|Q_i| = |B_i| \times 2\la^{2-p} r_i^2$, to get
  \begin{equation}
  \label{bound_I_1_rho_geq_r_i}
   \begin{array}{ll}
    J_1
    & {\apprle}\left[ \frac{1}{r_i} \lbr \frac{r_i \la}{\ve} + \frac{\ve}{\la r_i} |v^i|^2 \rbr + \la \right] |4Q_i|\lbr \fiint_{4Q_i} (|\nabla u| + |h_0|)^{q}  \lsb{\chi}{\Om_T}\ dy \ d\tau \rbr^{\frac{p-1}{q}} \\
    & {\apprle} \left[ \frac{1}{r_i} \lbr \frac{r_i \la}{\ve} + \frac{\ve}{\la r_i} |v^i|^2 \rbr + \la \right] |4Q_i|\la^{p-1} 
     {\apprle} \frac{\la^p}{\ve} |4Q_i| +  \ve |4B_i| |v^i|^2.	
   \end{array}
  \end{equation}

Thus combining \cref{bound_I_1_rho_geq_r_i} and \cref{bound_I_1_rho_leq_r_i}, we get
\begin{equation}
 \label{bound_I_1}
 J_1 \apprle \frac{\la^p}{\ve} |4Q_i| +  \lsb{\chi}{r_i \leq \rho}\ve |4B_i| |v^i|^2,
\end{equation}
where we have set $\lsb{\chi}{r_i \leq \rho} = 1$ if  $r_i \leq \rho$ and $\lsb{\chi}{r_i \leq \rho}=0$ else.

  \item[Bound for $J_2$:] In this case, we can directly use \cref{lemma3.9} and \descref{W4} to get for any $\ve \in (0,1]$, the bound
  \begin{equation}
   \label{bound_I_22}
   \begin{array}{rcl}
    J_2 & \apprle & \int_{I_i} \int_{Q_i} \sum_{j=1}^n|\vec{w}_j| |\nabla_j (\Psi_i\vl)| \ dx \ dt \\
    & \apprle & \sum_{j=1}^n \int_{I_i} \lbr  \int_{B_i} \sum_{j=1}^n|\vec{w}_j|^{\pbp} \ dx \rbr^{\frac{1}{\pbp}} \lbr \int_{B_i}  |\nabla_j (\Psi_i\vl)|^{\pbo} \ dx\rbr^{\frac{1}{\pbo}}  \ dt \\
    & \apprle & |\nabla (\Psi_i \vl)| |B_i| |I_i|\fint_{I_i} \frac{1}{|B_i|^{\frac{1}{\pbp}}}\left\| \ddt{{w}(\cdot,t)}\right\|_{W^{-1,\pbp}(B_i)}  \ dt \\
    & \apprle & |\nabla (\Psi_i \vl)| |Q_i| \la^{p-1}.
   \end{array}
  \end{equation}
  
  In the case $\rho \leq r_i$, we follow the idea from \cref{bound_I_1_rho_leq_r_i} and in the case $r_i \leq \rho$, we follow the strategy from \cref{bound_I_1_rho_geq_r_i} to get
  \begin{equation*}
  \label{bound_J_2}
  J_2 \apprle \frac{\la^p}{\ve} |4Q_i| +  \lsb{\chi}{r_i \leq \rho}\ve |4B_i| |v^i|^2.
  \end{equation*}

  \item[Bound for $J_3$:] Substituting  \cref{lemma3.10_bound4}, \cref{lemma3.11.bound2} and \descref{W13} into \cref{3.128},  for any $\ve \in (0,1]$, there holds
  \begin{equation}
   \label{bound_3_1}
   \begin{array}{ll}
    |\pa_t(\Psi_i \vl)(z)| 
    & \apprle \frac{1}{\gamma r_i^2} \lbr \frac{r_i \la}{\ve}  + \frac{\ve}{r_i \la} |v^i|^2 \rbr + \frac{1}{\gamma r_i^2} \min\{r_i,\rho\} \la \approx \frac{1}{\gamma r_i^2} \lbr \frac{r_i \la}{\ve}  + \frac{\ve}{r_i \la} |v^i|^2 \rbr. 
   \end{array}
  \end{equation}
Making use of  \cref{bound_3_1} in the expression for $J_3$ in \cref{bound_J_3_1}, we get
  \begin{equation*}
   \label{bound_I_3}
   \begin{array}{ll}
    J_3 & \apprle \frac{1}{\gamma r_i^2} \lbr \frac{r_i \la}{\ve}  + \frac{\ve}{r_i \la} |v^i|^2 \rbr\iint_{\frac34Q_i} |v-v^i|  \lsb{\chi}{\Om_T}  \ dy \ d\tau.
   \end{array}
  \end{equation*}
  We can now proceed similarly to \cref{3.30} to get
\begin{equation}
 \label{bound_I33}
 \begin{array}{ll}
 J_3 & \apprle \frac{1}{\gamma r_i^2} \lbr \frac{r_i \la}{\ve}  + \frac{\ve}{r_i \la} |v^i|^2 \rbr r_i \la |Q_i| 
   \apprle \frac{\la^p }{\ve} |4Q_i| + \ve |4B_i| |v^i|^2.
 \end{array}
\end{equation}

 \end{description}

 Substituting the estimates \cref{bound_I_1}, \cref{bound_I_22} and \cref{bound_I33} into \cref{first_1} completes the proof of the lemma.
\end{proof}

We now come to  essentially the most important estimate which will be needed to prove the difference estimate:
\begin{lemma}
 \label{crucial_lemma}
 There exists a positive constant $C=C_{(n,p,q, \Lambda_1,b_0,r_0)}$ such that the following estimate holds for every $t \in [0,T]$:
 \begin{equation}
 \label{3122}
  \int_{\Om \setminus \elam^t} (|v|^2 - |v - \vl|^2)(x,t) \ dx \geq - C  \la^p |\RR^{n+1} \setminus \elam|.
 \end{equation}
\end{lemma}
\begin{proof}
 Let us fix any $t\in [0,T]$ and any point $x \in \Om \setminus \elam^t$.  Now define
 \begin{equation*}
  \tTh := \left\{  i \in \NN: \spt(\Psi_i) \cap \Om \times \{t\} \neq \emptyset \quad \text{and} \quad |v| + |\vl| \neq 0 \quad  \text{on}\ \spt(\Psi_i)\cap (\Om \times \{t\}) \right\}.
 \end{equation*}
 Hence we only need to consider $i \in \tTh$.  
 Noting that $\sum_{i \in \tTh} \Psi_i(\cdot, t) \equiv 1$ on $\RR^n \cap \elam^t$, we can rewrite the left-hand side of \cref{3122} as 
 \begin{equation*}
  \label{I_1}
  \begin{array}{ll}
  \int_{\Om \setminus \elam^t} (|v|^2 - |v - \vl|^2)(x,t) \ dx & = \sum_{i \in \tTh} \int_{\Om} \Psi_i (|v|^2 - |v - \vl|^2) \ dx \\
  & = \sum_{i \in \tTh} \int_{\Om} \Psi_i(z) \lbr |v^i|^2  + 2 \vl (v - v^i) \rbr \ dx - \sum_{i \in \tTh} \int_{\Om} \Psi_i(z) |\vl - v^i|^2 \ dx\\
  & := J_1 - J_2.
  \end{array}
 \end{equation*}
 
 \begin{description}[leftmargin=*]
  \item[Estimate of $J_1$:] Using \cref{3.120}, we get
  \begin{equation}
   \label{I_1_1}
   J_1 \apprge \sum_{i \in \tTh} \int_{\Om} \Psi_i(z)  |v^i|^2 \ dz - \ve \sum_{i \in \tTh} |4B_i| |v^i|^2 - \sum_{i \in \tTh} \frac{\la^p}{\ve} |4Q_i|.
  \end{equation}
  From \cref{lipschitz_extension_one}, we have $v^i = 0$ whenever $\spt(\Psi_i) \cap \Om^c \neq \emptyset$. Hence we only have to sum over all those $i \in \tTh$ for which $\spt(\Psi_i) \subset \Om\times [0,\infty)$.  In this case, we make use of a suitable choice for $\ve \in (0,1]$, and use \descref{W7} along with \descref{W12} and estimate \cref{I_1_1} from below to get
  \begin{equation}
 \label{bound_I1}
 J_1 \apprge -\la^p |\RR^{n+1} \setminus \elam|. 
\end{equation}

  \item[Estimate of $J_2$:] For any $x \in \Om \setminus \elam^t$, we have from \descref{W14} that $\sum_{j} \Psi_j(x,t) = 1$, which gives
  \begin{equation}
   \label{I_2_1}
    \Psi_i(z) |\vl(z) - v^i|^2 \apprle \Psi_i(z)  \sum_{j \in A_i}   \lbr v^j - v^i\rbr^2  \overset{\redlabel{3.104.a}{a}}{\apprle} \min\{ \rho, r_i\}^2 \la^2.
  \end{equation}
To obtain \redref{3.104.a}{a} above, we made use of \cref{lemma3.8}  along with \descref{W8}.  Substituting \cref{I_2_1} into the expression for $J_2$ and using $|Q_i| = |B_i| \times 2\gamma r_i^2$, we get
\begin{equation}
 \label{bound_I_2}
J_2 \apprle \sum_{i \in \tTh_1} \left|B_i\right| \frac{\gamma r_i^2}{\gamma} \la^2 \apprle \la^p |\RR^{n+1} \setminus \elam|.
\end{equation}
 \end{description}
Substituting \cref{bound_I1} and \cref{bound_I_2} into \cref{I_1}, we get the desired estimate. This completes the proof of the lemma.
\end{proof}

\subsection{A priori estimate - Proof of \texorpdfstring{\cref{main_theorem_1}}.}
Consider the following cut-off function $\zv \in C^{\infty}(0,\infty)$ such that $0 \leq \zv(t) \leq 1$ and 
\begin{equation*}
\label{def_zv.3}
\zv(t) = \left\{ \begin{array}{ll}
                1 & \text{for} \ t \in (0+\ve,T-\ve)\\
                0 & \text{for} \ t \in (-\infty,0)\cup (T,\infty).
                \end{array}\right.
\end{equation*}	
It is easy to see that 
\begin{equation*}
\label{bound_zv.3}
\begin{array}{c}
\zv'(t) = 0 \ \txt{for} \  t \in (-\infty,0) \cup (0+\ve,T-\ve)\cup (T,\infty), \\
|\zv'(t)| \leq \frac{c}{\ve}\  \txt{for} \  t \in (0,0+\ve) \cup (T-\ve,0).
\end{array}
\end{equation*}

Let $h \in (0,T)$ be the Steklov averaging exponent. Without loss of generality, we shall always take $h \ge 2\ve$ since we will take limits in the following order $\lim_{h \rightarrow 0} \lim_{\ve \rightarrow 0}$. 

Let us use $\vlh \zv$ constructed in \cref{lipschitz_extension} as a test function in \cref{main-1} to get
\begin{equation}
\label{5.13.3}
\begin{array}{rcl}
\iint_{\Om_T} \ddt{[u-w]_h} \vlh \zv \ dx \ dt + \iint_{\Om_T} \iprod{[\aa(x,t,\nabla u)]_h}{\nabla \vlh} \zv \ dx \ dt & =& \iint_{\Om_T} \ddt{[w]_h} \vlh \zv \ dx \ dt, \\
\end{array}
\end{equation}
which we express as 
\[
L_1 + L_2 = L_3.
\]
\begin{description}
\item[Estimate of $L_1$:] 
Recalling \cref{def_v_app},  we get
\begin{equation}
   \label{6.38.3}
   \begin{array}{ll}
    L_1     & = \int_{0}^{T} \int_{\Om} \dds{v_h(y,s)}  \vlh(y,s)\zv(s) \ dy\ ds \\
    & = \frac{1}{2}  \int_{0}^{T} \int_{\Om}  \frac{d{\lbr \lbr[[](v_h)^2 - (\vlh - v_h)^2\rbr[]]\zv(s) \rbr }}{ds} \ dy \ ds+ \int_{0}^{T} \int_{\Om\setminus \elam^{s}}  \dds{\vlh}  (\vlh-v_h) \zv(s)\ dy \ ds  \\
    & \qquad - \frac{1}{2}\int_0^{T} \int_{\Om} \dds{\zv} \lbr v_h^2 - (\vlh - v_h)^2 \rbr \ dy \ ds\\
    & := J_1(T) - J_1(0)+J_2  - J_3,
   \end{array}
  \end{equation}
  where we have set 
  \begin{equation*}
  \label{def_i_1.3}J_1(s) := \frac12 \int_{\Om} ( (v_h)^2 - (\vlh - v_h)^2 ) (y,s) \zv(s) \ dy.
  \end{equation*}
 Since $\zv(0) =\zv(T) =  0$, we have 
 \begin{equation}\label{bond_J_1} J_1(0) = J_1(T) =0.\end{equation}
 Form \cref{lemma3.14} applied with $\vt =1$, we have the bound
\begin{equation}
    \label{6.39.3}
    \begin{array}{ll}
     |J_2| & \apprle \iint_{\Om_T\setminus \elam}   \left| \dds{\vlh}  (\vlh-v_h)\right| \ dy \ ds \apprle\la^p |\RR^{n+1} \setminus \elam| .
    \end{array}
   \end{equation}
   
   In order to estimate $- \int_{0}^{T} \int_{\Om} \dds{\zv} \lbr v_h^2 - (\vlh - v_h)^2 \rbr \ dy \ ds$, we take limits first in $\ve \searrow 0$ followed by $h\searrow 0$ to get
\begin{equation}
\label{5.22.3}
\begin{array}{ll}
- \int_{0}^{T} \int_{\Om} \dds{\zv} \lbr v_h^2 - (\vlh - v_h)^2 \rbr \ dy \ ds  \xrightarrow{\lim\limits_{h \searrow 0}\lim\limits_{\ve \searrow 0}} & \int_{\Om} (v^2 - (\vl - v)^2 )(x,T) \ dx \\
& - \int_{\Om} (v^2 - (\vl - v)^2 )(x,0) \ dx.
\end{array}
\end{equation}

For the second term on the right of \cref{5.22.3}, we observe that $\vl = v$ on $\elam$. Note that $\vl(\cdot,0) =0 = v(\cdot,0)$ by the initial condition. Thus, the second term on the right of \cref{5.22.3} vanishes and so
\begin{equation}
\label{5.23.3}
- \int_{0}^{T} \int_{\Om} \dds{\zv} \lbr v_h^2 - (\vlh - v_h)^2 \rbr \ dy \ ds  \xrightarrow{ \lim\limits_{h \searrow 0}\lim\limits_{\ve \searrow 0}}  \int_{\Om} (v^2 - (\vl - v)^2 )(x,T) \ dx.
\end{equation}

\item[Estimate of $L_2$:] We split $L_2$ and make use of the fact that $\vlh(z) = v_h(z)  \overset{\cref{def_v_app}}{=} [u-w]_h(z)$ for all $z\in \elam\cap \Om_T$ to get 
\begin{equation}
\label{5.17.3}
\begin{array}{ll}
L_2 
& = \iint_{\Om_T\cap \elam} \iprod{[\aa(x,t,\nabla u)]_h}{\nabla [u-w]_h} \zv\ dz + \iint_{\Om_T\setminus \elam} \iprod{[\aa(x,t,\nabla u)]_h}{\nabla \vlh} \zv\ dz\\
& =: L_2^1 + L_2^2.
\end{array}
\end{equation}
   
\begin{description}[leftmargin=*]
\item[Estimate of $L_2^1$:] Using ellipticity from \cref{abounded}, we get
\begin{equation}
\label{5.18.3}
\begin{array}{rcl}
L_2^1 
& \apprge & \iint_{\Om_T \cap \elam} \left[|\nabla u|^p  -\lvert h_0 \rvert^p\right]_h \zv \ dz - \iint_{\Om_T \cap \elam} [|\nabla u|^{p-1}+ |h_0|^{p-1}]_h |\nabla w| |\zv| \ dx \ dt \\
& \overset{\lim\limits_{h \searrow 0}\lim\limits_{\ve \searrow 0}}{=} & \iint_{\Om_T \cap \elam} |\nabla u|^p-\lvert h_0\rvert^p \ dz - \iint_{\Om_T \cap \elam}\lbr |\nabla u|^{p-1}+ |h_0|^{p-1}\rbr |\nabla w|  \ dx \ dt.
\end{array}
\end{equation}

\item[Estimate of $L_2^2$:] Using the bound  from \cref{lemma3.9}, we get
\begin{equation}
\label{5.19.3}
\lvert L_2^2\rvert \apprle \la \iint_{\Om_T\setminus \elam} [|\nabla u|^{p-1}+ |h_0|^{p-1}]_h\ dz \overset{\lim\limits_{h \searrow 0}\lim\limits_{\ve \searrow 0}}{=}  \la \iint_{\Om_T\setminus \elam} \lbr|\nabla u|^{p-1}+ |h_0|^{p-1} \rbr \ dz.
\end{equation}

Combining \cref{5.18.3} and \cref{5.19.3} with \cref{5.17.3}, we get
\begin{equation}
\label{est_L_2}
\hspace*{-0.2cm}L_2 \apprge  \iint_{\Om_T \cap \elam} |\nabla u|^p-\lvert h_0\rvert^p \ dz - \iint_{\Om_T \cap \elam}\lbr |\nabla u|^{p-1}+ |h_0|^{p-1}\rbr |\nabla w|  \ dz - \la \iint_{\Om_T\setminus \elam} \lbr|\nabla u|^{p-1}+ |h_0|^{p-1} \rbr \ dz.
\end{equation}

\end{description}

\item[Estimate of $L_3$:] Analogous to estimate for $L_2$, we split $L_3$ into integrals over $\elam$ and $\elam^c$, to find
\begin{equation}
\label{5.20.3}
\begin{array}{rcl}
L_3 & \overset{\redlabel{5.63.a}{a}}{=}  & \int_0^T \left[ \int_{\Om} \iprod{\ddt{w}}{ \vlh}_{(W^{-1,\pbp}(\Om),W_0^{1,\pbo}(\Om))}(x,t) \ dx \right]_h \zv(t) \ dt \\
& \overset{\text{\cref{lemma_lihe_wang}}}{=}& \int_0^T \left[ \int_{\Om} \iprod{\vec{w}}{\nabla \vlh}(x,t) \ dx \right]_h \zv(t) \ dt \\
& \overset{\redlabel{5.63.b}{b}}{\apprle} & \int_0^T \left[ \int_{\Om\cap \elam^t} \iprod{\vec{w}}{\nabla [u-w]_h}(x,t) \ dx \right]_h \zv(t) \ dt  \\
&& \qquad + \la \int_0^T \left[ \int_{\Om\setminus \elam^t} |\vec{w}| \ dx \right]_h \zv(t) \ dt \\
& \overset{\lim\limits_{h \searrow 0}\lim\limits_{\ve \searrow 0}}{\apprle}& \iint_{\Om_T \cap \elam} |\vec{w}| |\nabla (u-w)| \ dz + \la \iint_{\Om_T \setminus \elam} |\vec{w}| \ dz.
\end{array}
\end{equation}
To obtain \redref{5.63.a}{a}, we made use of the weaker assumption \cref{est3.3} (see \cref{imp_rmk} for the details) and to obtain \redref{5.63.b}{b},  we made use of  \cref{lipschitz_extension} and  \cref{lemma3.9}.
\end{description}

Combining \cref{6.38.3}, \cref{bond_J_1}, \cref{6.39.3}, \cref{5.23.3}, \cref{est_L_2} and  \cref{5.20.3} with \cref{5.13.3}, we get
\begin{equation}
\label{combined_1_1.3}
\begin{array}{rcl}
 \int_{\Om} (v^2 - (\vl - v)^2 )(x,T) \ dx +  \iint_{\Om_T \cap \elam} |\nabla u|^p -\lvert h_0\rvert^p \ dz & \apprle & \iint_{\Om_T \cap \elam} \lbr |\nabla u|^{p-1}+|h_0|^{p-1} \rbr |\nabla w| \ dz\\
 && + \la \iint_{\Om_T\setminus \elam} \lbr|\nabla u|^{p-1}+ |h_0|^{p-1} \rbr \ dz\\
& & + \iint_{\Om_T \cap \elam} \abs{{\vec{w}}} |\nabla (u-w)|  \ dz \\ 
& & + \la \iint_{\Om_T\setminus \elam}  \abs{{\vec{w}}}\ dz \\
& &  + \la^p |\RR^{n+1} \setminus \elam|.
\end{array}
\end{equation}

{\color{black}In fact, if we consider a cut-off function $\zv^{t_0} (\cdot)$ for some $t_0 \in (0,T)$, where 
\begin{equation*}
\zv^{t_0}(t) = \left\{ \begin{array}{ll}
                1 & \text{for} \ t \in (0+\ve,t_0-\ve)\\
                0 & \text{for} \ t \in (-\infty,0)\cup (t_0,\infty),
                \end{array}\right.
\end{equation*}
we get the following analogue of \cref{combined_1_1.3}:
\begin{equation}
\label{combined_1_new_11.3}
\begin{array}{rcl}
\int_{\Om} (v^2 - (\vl - v)^2 )(x,t_0) \ dx +  \int_0^{t_0}\int_{\Om \cap \elam^t} |\nabla u|^p-\lvert h_0\rvert^p  \ dx \ dt & \apprle & \iint_{\Om_T \cap \elam} \lbr |\nabla u|^{p-1}+|h_0|^{p-1} \rbr |\nabla w| \ dz\\
 && + \la \iint_{\Om_T\setminus \elam} \lbr|\nabla u|^{p-1}+ |h_0|^{p-1} \rbr \ dz\\
& & + \iint_{\Om_T \cap \elam} \abs{{\vec{w}}} |\nabla (u-w)|  \ dz \\ 
& & + \la \iint_{\Om_T\setminus \elam}  \abs{{\vec{w}}}\ dz \\
& &  + \la^p |\RR^{n+1} \setminus \elam|.
\end{array}
\end{equation}}

Using  \cref{crucial_lemma},  we get for almost every $t \in (0,T)$,
\begin{equation}
 \label{4.13.3}
    \int_{\Om} | (v)^2 - (\vl - v)^2 | (y,t) \ dy  \apprge \int_{\elam^t} | v (x,t)|^2 \ dx   - \la^p |\RR^{n+1} \setminus \elam|.
 \end{equation}
 
 Thus combining \cref{4.13.3} with \cref{combined_1_new_11.3}, we get
 \begin{equation}
\label{fully_combined.3}
\begin{array}{lll}
\sup_{t \in (0,T)}\int_{\elam^t} | v (x,t)|^2 \ dx +  \int_0^{t_0}\int_{\Om \cap \elam^t} |\nabla u|^p-\lvert h_0 \rvert^p  \ dx \ dt & \apprle & \iint_{\Om_T \cap \elam} \lbr |\nabla u|^{p-1}+|h_0|^{p-1} \rbr |\nabla w| \ dz\\
 && + \la \iint_{\Om_T\setminus \elam} \lbr|\nabla u|^{p-1}+ |h_0|^{p-1} \rbr \ dz\\
& & + \iint_{\Om_T \cap \elam} \abs{{\vec{w}}} |\nabla (u-w)|  \ dz \\ 
& & + \la \iint_{\Om_T\setminus \elam}  \abs{{\vec{w}}}\ dz \\
& &  + \la^p |\RR^{n+1} \setminus \elam|.
\end{array}
\end{equation} 
Since $\int_{\elam^t} | v (x,t)|^2 \ dx$ occurs on the left hand side of \cref{fully_combined.3} and is positive, we can ignore this term.  
Let us now multiply \cref{fully_combined.3}  with $\la^{-1-\be}$ and integrating over $(0,\infty)$ with respect to $\la$, we get
 \begin{equation}
 \label{K_expression.3}
K_1 \apprle K_2+ K_3 + K_4 + K_5+ K_6,
 \end{equation}
 where we have set
 \begin{equation*}
  \begin{array}{rcl}
  K_1 \ &:=& \ \int_{0}^{\infty} \la^{-1-\be} \iint_{\Om_T \cap \elam} \lbr |\nabla u|^p -\lvert h_0\rvert^p \rbr \ dz \ d\la,\\
  K_2 \ &:=& \ \int_{0}^{\infty} \la^{-1-\be} \iint_{\Om_T \cap \elam} \lbr |\nabla u|+|h_0|\rbr^{p-1} |\nabla w|  \ dz \ d\la,  \\
  K_3 \ &:=& \ \int_{0}^{\infty} \la^{-\be} \iint_{\Om_T \setminus \elam} \lbr |\nabla u|+|h_0|\rbr^{p-1}   \ dz \ d\la,  \\
  K_4 \ &:=& \ \int_{0}^{\infty} \la^{-1-\be} \iint_{\Om_T \cap \elam} \abs{{\vec{w}}} |\nabla (u-w)|  \ dz \ d\la,  \\
  K_5 \ &:=& \ \int_0^{\infty} \la^{-\be} \iint_{\Om_T\setminus \elam} \abs{{\vec{w}}}\ dz \ d\la, \\
  K_6\  &:=& \ \int_{0}^{\infty} \la^{-1-\be}  \la^p |\RR^{n+1} \setminus \elam| \  d\la.
  \end{array}
 \end{equation*}

 Let us now estimate each of the $\{K_i\}_{i=1}^6$ as follows:
  \begin{description}
 \item[Estimate of $K_1$:] Applying Fubini, we get
 \begin{equation*}
 \label{3.13.3}
 K_1 \apprge \frac{1}{\be} \iint_{\Om_T} g(z)^{-\be}(|\nabla u|^p-\lvert h_0 \rvert^p) \ dz.
 \end{equation*}
 Using Young's inequality along with \cref{max_g_fde.3}, we get for any $\ep_1 >0$,  \begin{equation}
 \label{bound_K_1.3}
 \def\arraystretch{1.3}
 \begin{array}{rcl}
 \iint_{\Om_T} |\nabla u|^{p-\be} \ dz & \apprle&   C(\ep_1) \be K_1  + (\ep_1 +\beta)\iint_{\Om_T} (|\nabla u|+|h_0|)^{p-\be} \ dz  + \ep_1 \iint_{\Om_T} |\nabla w|^{p-\be}\, dz \\
 && \quad + \ep_1\|w' \|_{L^{\pbp}(0,T; W^{-1,\pbp}(\Om))}^{\pbp}.
 \end{array}
 \end{equation}

  \item[Estimate of $K_2$:] Again by Fubini, we get
 \begin{equation*}
 \label{3.18.3}
 K_2  = \frac{1}{\be} \iint_{\Om_T} g(z)^{-\be} \lbr |\nabla u| + |h_0|\rbr^{p-1} \abs{\nabla w} \ dz.
 \end{equation*}
From the definition of $g(z)$ in \cref{def_g_app}, we see that for $z \in \Om_T$, we have $g(z) \geq (|\nabla u|+|h_0|)(z)$ which implies $g(z)^{-\be} \leq (|\nabla u|+|h_0|)^{-\be}(z)$. Applying Young's inequality, for any $\ep_2 >0$, we get
\begin{equation}
\label{3.19.3}
K_2 
 \apprle \frac{C(\ep_2)}{\be} \iint_{\Om_T}  |\nabla w|^{p-\be} \ dz + \frac{\ep_2}{\be} \iint_{\Om_T}\lbr |\nabla u| + |h_0|\rbr^{p-\be} \ dz.
\end{equation}

\item[Estimate of $K_3$:] Again applying Fubini, we get
 \begin{equation}
 \label{3.21.3}
 \begin{array}{rcl}
 K_3 & = &  \frac{1}{1-\be} \iint_{\Om_T} g(z)^{1-\be}\lbr |\nabla u|+|h_0|\rbr^{p-1} \ dz\\
& \apprle & \iint_{\Om_T} g(z)^{p-\be} \ dz + \iint_{\Om_T} \lbr |\nabla u|+|h_0|\rbr^{p-\be} \ dz \\
& \overset{\cref{max_g_fde.3}}{\apprle} & \iint_{\Om_T}\lbr |\nabla u|+|h_0|\rbr^{p-\be} \ dz+ \iint_{\Om_T}|\nabla w|^{p-\be}\ dz  + \|w' \|_{L^{\pbp}(0,T; W^{-1,\pbp}(\Om))}^{\pbp}.
 \end{array}
 \end{equation}
 
\item[Estimate of $K_4$:] Again by Fubini, we get
 \begin{equation*}
 \label{3.18.3.3}
 K_4  = \frac{1}{\be} \iint_{\Om_T} g(z)^{-\be} |\nabla u-\nabla w| \abs{\vec {w}} \ dz.
 \end{equation*}
From the definition of $g(z)$ in \cref{def_g_app}, we see that for $z \in \Om_T$, we have $g(z) \geq |\nabla u-\nabla w|(z)$ which implies $g(z)^{-\be} \leq |\nabla u-\nabla w|^{-\be}(z)$. Applying Young's inequality, for any $\ep_3 >0$, we get
\begin{equation}
\label{3.19.3.3}
\begin{array}{rcl}
K_4 & \apprle &  \frac{1}{\be} \iint_{\Om_T} |\nabla u - \nabla w|^{1-\be} |\vec{w}| \ dz \\
& \apprle & \frac{C(\ep_3)}{\be} \int_0^T \int_{\Om} |\vec{w}|^{\pbp} \ dx \ dt + \frac{\ep_3}{\be} \iint_{\Om_T} |\nabla u- \nabla w|^{p-\be} \ dx \ dt\\
& \overset{\text{\cref{lemma_lihe_wang}}}{\apprle} & \frac{C(\ep_3)}{\be}\|w' \|_{L^{\pbp}(0,T; W^{-1,\pbp}(\Om))}^{\pbp} + \frac{\ep_3}{\be} \iint_{\Om_T} |\nabla w|^{p-\be} \ dz \\
&& \quad + \frac{\ep_3}{\be} \iint_{\Om_T} |\nabla u|^{p-\be} \ dz.
\end{array}
\end{equation}

\item[Estimate of $K_5$:] Again applying Fubini, we get
 \begin{equation}
 \label{3.21.3.3}
 \begin{array}{rcl}
 K_5 & = &  \frac{1}{1-\be} \iint_{\Om_T} g(z)^{1-\be}|\vec{w}| \ dz\\
& \apprle & \iint_{\Om_T} g(z)^{p-\be} \ dz + \iint_{\Om_T} |\vec{w}|^{\frac{p-\be}{p-1}} \ dz \\
& \overset{\redlabel{3.20.3.a}{a}}{\apprle} & \iint_{\Om_T}\lbr |\nabla u|+|h_0|\rbr^{p-\be} \ dz+ \iint_{\Om_T}|\nabla w|^{p-\be}\ dz  + \|w' \|_{L^{\pbp}(0,T; W^{-1,\pbp}(\Om))}^{\pbp}.
 \end{array}
 \end{equation}
 To obtain \redref{3.20.3.a}{a}, we made use of  \cref{max_g_fde.3} and \cref{lemma_lihe_wang}.

\item[Estimate of $K_6$:] Applying the layer cake representation followed by \cref{max_g_fde.3}, we get
 \begin{equation}
 \label{3.22.3}
 \begin{array}{rcl}
 K_6  & = &  \frac{1}{p-\be} \iint_{\RR^{n+1}} g(z)^{p-\be} \ dz  \\
 & \overset{\cref{max_g_fde.3}}{\apprle} & \iint_{\Om_T}\lbr |\nabla u|+|h_0|\rbr^{p-\be} \ dz+ \iint_{\Om_T}|\nabla w|^{p-\be}\ dz  + \|w' \|_{L^{\pbp}(0,T; W^{-1,\pbp}(\Om))}^{\pbp}.
 \end{array}
 \end{equation}
 \end{description}
 
 Combining \cref{bound_K_1.3}, \cref{3.19.3}, \cref{3.21.3}. \cref{3.19.3.3}, \cref{3.21.3.3} and \cref{3.22.3} with \cref{K_expression.3}, we get
\begin{equation*}
 \label{3.23.3}
 \begin{array}{rcl}
  \iint_{\Om_T} |\nabla u|^{p-\be} \ dz & \apprle &  \lbr[[] C(\ep_1) \lbr \ep_2+\be + \ep_3\rbr + \ep_1  \rbr[]] \iint_{\Om_T}  |\nabla u|^{p-\be}  \ dz    \\
  &&\qquad  +\lbr[[]C(\ep_1)(\ep_2+\be + 1) + \ep_1\rbr[]] \iint_{\Om_T} |h_0|^{p-\be}\, dz \\
  &&\qquad  + \lbr[[] C(\ep_1)(C(\ep_2) + \be+\ep_3) + \ep_1\rbr[]] \iint_{\Om_T} |\nabla w|^{p-\be} \ dz  \\
 && \qquad +  C(\ep_1)(C(\ep_3) + \be)  \|w' \|_{L^{\pbp}(0,T; W^{-1,\pbp}(\Om))}^{\pbp}.
 \end{array}
 \end{equation*}
 Choosing $\ep_1$ small followed by choosing $\ep_2,\ep_3,\be$ small, we get the desired estimate. This completes the proof of the theorem. \hfill \qed

\section{Higher integrability at the initial boundary}
\label{section6}

Let us consider a cylinder $Q= Q_{\rho,s}(z_0)$ centered at a point $z_0 = (x_0,t_0) \in \Om \times \RR$ such that \[Q \cap \Om \times [0,\infty) \neq \emptyset\txt{and}\ \  Q \cap \Om \times (-\infty,0] \neq \emptyset.\] In particular, we take a cylinder that crosses the initial time slice $\{t=0\}$. Furthermore, assume that the cylinder $Q$ satisfies
\begin{equation}\label{Q_rest}
Q = B_{\rho} \times I_s \txt{and}\ \ \qquad 16Q \subset \Om \times \RR.
\end{equation}
We shall suppress writing the center of the cylinder $z_0 = (x_0,t_0)$ henceforth unless necessary. Corresponding to this cylinder, let us take the following test functions:
\begin{gather}
\eta(x) \in C_c^{\infty}(B_{8\rho}): \quad \eta(x) \equiv 1 \quad \text{on} \ B_{4\rho}, \quad |\nabla \eta| \leq \frac{C}{\rho}, \label{def_eta}\\
\zeta(t) \in C_c^{\infty}(I_{8s}): \quad \zeta(t) \equiv 1 \quad \text{on} \ I_{4s}, \quad |\zeta'| \leq \frac{C}{s}.\label{def_zeta}
\end{gather}
For any given $(x,t) \in \Om \times [0,T)$, we define the following function and its corresponding Steklov average to be
\begin{equation}
\label{def-v-ini}
v(x,t) :=  (u(x,t) - w(x,t)) \eta(x) \zeta(t) \txt{and} \ \ \vh(x,t) := [u-w]_h(x,t)\eta(x) \zeta(t).
\end{equation}
From \cref{main-2} and \cref{time_average},  we see that 
\[\vh \xrightarrow{h \searrow 0} v \txt{and}\  v(z) = 0 \ \ \text{for} \ z \in \omtz.\]
In subsequent calculations, we extend $v$ by zero to $\Om \times (-\infty,0]$.

\subsection{Construction of test function}
In what follows, we shall use \cref{lemma_lihe_wang} to obtain a representation 
\begin{equation}
\label{def_vec_w}
\ddt{w} = \dv \vec{w} \txt{in} \Om,
\end{equation}
for some $\vec{w} \in L^{\frac{p}{p-1}}(0,T; L^{\frac{p}{p-1}}(\Om))$.

Again, let us fix the following choice of exponents:
\begin{equation}
1< q \leq p-2\be < p-\be < p,
\end{equation}
where $\be$ is a constant to be chosen sufficiently small later on.
Define the following function
\begin{equation}
\label{def_g_ini}
g(z) := \max\left\{G_1(z), G_2(z), G_3(z), G_4(z), G_5(z) \right\}.
\end{equation}
with 
\begin{equation}
\label{eq6.7}
\begin{array}{rcl}
G_1(z) & := & \mm((|\nabla u|+|h_0|)^{q}\lsb{\chi}{\qomt})^{\frac{1}{q}}(z), \\ 
G_2(z) & := & \mm(|\nabla v|^q\lsb{\chi}{\qomt})^\frac{1}{q}(z), \\
G_3(z) & := & \mm\lbr\frac{|u-w|^q}{\rho^q}\lsb{\chi}{\qomt}\rbr^\frac{1}{q}(z), \\
G_4(z) & := &  \mm(|\nabla w|^q \lsb{\chi}{\qomt})^\frac{1}{q}(z),\\
G_5(z) & := & \mmn{1}(w'\lsb{\chi}{16Q\cap\Om_T} )(z).
\end{array}
\end{equation}
where $v$ is defined in \cref{def-v-ini} and $u$ is as defined in the hypothesis of \cref{main_theorem_2} and $G_5$ is defined using \cref{def_vec_w} as follows (see also \cref{7.7.4}):
\[
\mmn{1}(w'\lsb{\chi}{16Q\cap\Om_T} )(x,t) := \sup_{(a,b) \ni t} \sup_{B \ni x} \fint_{a}^b \fint_{B} |\vec{w}|\lsb{\chi}{\qomt} \ dx \ dt.
\]

 Let us now define the good set to be 
\begin{equation*}
\label{elam-ini}
\elam = \{ z \in \RR^{n+1} : g(z) \leq \la\},
\end{equation*}
and apply \cref{whitney_decomposition} and \cref{partition_unity} with $\mathbb{E} = \elam^c$ to get a covering of $\elam^c$.  Recall that the intrinsic scaling is of the form 
\begin{equation*}
\label{intrinsic_scaling-ini}
\gamma:=\la^{2-p} \txt{and} Q_j(x,t) = B_{r_j}(x) \times (t_j - \gamma r_j^2, t_j + \gamma r_j^2). 
\end{equation*}

Now we define the following Lipschitz extension function as follows:
\begin{equation}
\label{lipschitz-extension-ini}
\vlh(z) := v_h(z) - \sum_i \Psi_i(z) (v_h(z) - v_h^i).
\end{equation}
where
\begin{equation*}
\label{lipschitz-extension-one-ini}
v_h^i := \left\{ \begin{array}{ll}
                  \frac1{\|\Psi_i\|_{L^1(\frac34Q_i)}}\iint_{\frac34Q_i} v_h(z) \Psi_i (z) \lsb{\chi}{[0,T]} \ dz & \text{if} \ \frac34Q_i \subset 8B \times [0,\infty), \\
                   0 & \text{else}.
                  \end{array}\right.
\end{equation*}
\emph{Note that $u = 0$ on $\omtz$, which along with \cref{Q_rest} enables us to  switch between $\lsb{\chi}{[0,T]}$ and $\lsb{\chi}{\Om_T}$ without affecting the calculations.}

\begin{assumption}
\label{alpha_0}
Let $\al_0 \in \RR^+$ be such that the following is satisfied:
\begin{equation*}\label{hypothesis_1-ini}
\begin{array}{c}
\al_0^{p-\be} \apprle \fiint_{Q} \lbr|\nabla u|+|h_0|\rbr^{p-\be} \lsb{\chi}{\qomt} \ dz + \fiint_Q |\nabla w|^{p-\be} \lsb{\chi}{\qomt} \ dz + \left\| \lvert\ddt{w}\rvert\frac{\lsb{\chi}{16Q\cap\Om_T}}{|16Q|} \right\|_{L^{\frac{p-\be}{p-1}}(0,T; W^{-1.\frac{p-\be}{p-1}}(\Om))}^{\frac{p-\be}{p-1}},\vspace{1em}\\
\fiint_{16Q} \lbr|\nabla u|+|h_0|\rbr^{p-\be} \lsb{\chi}{\qomt} \ dz
+ \fiint_{16Q} |\nabla w|^{p-\be} \lsb{\chi}{\qomt} \ dz + \left\| \lvert\ddt{w}\rvert\frac{\lsb{\chi}{16Q\cap\Om_T}}{|16Q|} \right\|_{L^{\frac{p-\be}{p-1}}(0,T; W^{-1.\frac{p-\be}{p-1}}(\Om))}^{\frac{p-\be}{p-1}} \apprle \al_0^{p-\be},
\end{array}
\end{equation*}
where $Q =Q_{\rho,s} = Q_{\rho, \al_0^{2-p} \rho^2}$ and constant depends on universal constants. Noting \cref{def_g_ini}, it is easy to see that there exists a universal positive constant $c_e = c_e(n,p,\lamot)$ such that for all $\la \geq c_e \al_0$, we have  $\elam \neq \emptyset.$

Note here that we made use of \cref{def_vec_w} and denoted
\[
\frac{1}{|16Q|}\left\| \lvert\ddt{w}\rvert\lsb{\chi}{16Q\cap\Om_T} \right\|_{L^{\frac{p-\be}{p-1}}(0,T; W^{-1.\frac{p-\be}{p-1}}(\Om))}^{\frac{p-\be}{p-1}} := \fiint_{16Q} |\vec{w}|^{\frac{p-\be}{p-1}} \lsb{\chi}{\qomt} \ dz. 
\]
\end{assumption}

Let us first prove an important bound for $g$ as defined in \cref{def_g_ini}:
\begin{lemma}
\label{bound_g_x_t-ini}
Let \cref{alpha_0} be in force and let $\al_0 \in \RR^+$ and $c_e$ be as in \cref{alpha_0} and $g(z)$ be as in \cref{def_g_ini}, then the following holds for any $q < \vartheta \leq p-\be$:
\[
\iint_{\RR^{n+1}} |g(z)|^\vartheta \ dz \apprle |Q| \al_0^{\vartheta}.
\]
\end{lemma}
\begin{proof}
We proceed as follows:
\begin{equation}
\label{lemma62-6.8}
\begin{array}{rcl}
\iint_{\RR^{n+1}} |g(z)|^\vartheta \ dz & \overset{\redlabel{6.2a}{a}}{\apprle} & \iint_{16Q \cap \Om_T} \lbr|\nabla u|+|h_0|\rbr^\vartheta + |\nabla w|^\vartheta + |\nabla v|^\vartheta + \lbr \frac{|v|}{\rho} \rbr^\vartheta \ dz  \\
&& \quad + \left\| \lvert\ddt{w}\rvert\lsb{\chi}{16Q\cap\Om_T} \right\|_{L^{\frac{\vartheta}{p-1}}(0,T; W^{-1.\frac{\vartheta}{p-1}}(\Om))}^{\frac{\vartheta}{p-1}}\\
& \overset{\redlabel{6.2b}{b}}{\apprle} & \iint_{16Q \cap \Om_T} \lbr|\nabla u|+|h_0|+|\nabla w|\rbr^\vartheta \ dz + \iint_{16Q \cap \Om_T} \lbr \frac{|u-w|}{\rho} \rbr^\vartheta \ dz \\
&& \quad + \left\| \lvert\ddt{w}\rvert\lsb{\chi}{16Q\cap\Om_T} \right\|_{L^{\frac{\vartheta}{p-1}}(0,T; W^{-1.\frac{\vartheta}{p-1}}(\Om))}^{\frac{\vartheta}{p-1}}\\
& \overset{\redlabel{6.2c}{c}}{\apprle} & |Q| \al_0^\vartheta + \iint_{16Q \cap \Om_T} \lbr \frac{|u-w|}{\rho} \rbr^\vartheta \ dz.
\end{array}
\end{equation}
To obtain \redref{6.2a}{a}, we applied the standard maximal function estimate along with the bound from \cref{est5.3}, to obtain \redref{6.2b}{b}, we made use of \cref{def-v-ini} which gives $\nabla v = (u-w)\zeta \nabla \eta + \eta\zeta \nabla (u-w)$ and finally to obtain \redref{6.2c}{c}, we made use of the hypothesis from \cref{alpha_0}. 

To control the last term of \cref{lemma62-6.8}, let us use the cut-off function $\xi\in C_c^{\infty}(16Q\cap\Om_T\cap\{t\le0\})$ from \cref{lemma_crucial_1} and note that $\avgs{u-w}{\xi} = 0$. Thus we can apply \cref{lemma_crucial_1} to get
\begin{equation}
\label{lemma62-6.9}
\begin{array}{rcl}
\fiint_{16Q \cap \Om_T} \lbr \frac{|u-w|}{\rho} \rbr^\vartheta \ dz & = & \fiint_{16Q} \lbr \frac{|(u-w) - \avgs{u-w}{\xi}|}{\rho} \rbr^\vartheta \lsb{\chi}{[0,T]}\ dz\\
& \apprle & \fiint_{16Q} |\nabla u-\nabla w|^\vartheta \lsb{\chi}{[0,T]} \ dz \\
&& \quad +  \sup_{t_1,t_2 \in 16I \cap [0,T]} \abs{\frac{\avgs{u-w}{\mu}(t_2) - \avgs{u-w}{\mu}(t_1)}{\rho}}^\vartheta\\
& \overset{\text{\cref{alpha_0}}}{\apprle} & \al_0^\vartheta + \sup_{t_1,t_2 \in 16I \cap [0,T]} \abs{\frac{\avgs{u-w}{\mu}(t_2) - \avgs{u-w}{\mu}(t_1)}{\rho}}^\vartheta.
\end{array}
\end{equation}

To control the last term of \cref{lemma62-6.9}, we apply \cref{lemma_crucial_2_app} with $\varphi \equiv 1$ and $\phi(x) = \mu(x)$ to get for any $t_1,t_2 \in 16I \cap [0,T]$,
\begin{equation}
\label{lemma62-6.10}
\begin{array}{rcl}
|\avgs{u-w}{\mu}(t_2) - \avgs{u-w}{\mu}(t_1)| & \overset{\redlabel{612a}{a}}{\apprle} &  \|\nabla \mu\|_{L^{\infty}(16B_\rho)} \iint_{8Q}\lbr |\nabla u|+ |h_0|\rbr^{p-1} \lsb{\chi}{[0,T]} \ dz  \\
&& + \|\nabla \mu\|_{L^{\infty}(16B_\rho)}  |Q|\fiint_{16Q\cap\Om_T}|\vec{w}| \lsb{\chi}{\qomt} \ dz  \\
&\overset{\redlabel{612b}{b}}{\apprle} &  \frac{|Q|}{\rho^{n+1}} \al_0^{p-1}=\rho\alpha_0.
\end{array}
\end{equation}
To obtain \redref{612a}{a}, we made use of \cref{lemma_crucial_2_app} and to obtain \redref{612b}{b}, we made use of \cref{alpha_0}.

Combining \cref{lemma62-6.8}, \cref{lemma62-6.9} and \cref{lemma62-6.10}, we get the desired estimate.
\end{proof}

%
 By the choice of the cylinder, we see that $t_0-s < 0 < t_0+s$. As a consequence, let us define the followings :
\begin{gather}
\mch := \RR^n \times [0,t_0+s] \cap \Om_T, k\mch := \RR^n \times [0,t_0+k^2s]\cap \Om_T, \label{h-ini}\\
\Theta:= \{ i\in \NN: \frac34Q_i \cap 2\mch \neq \emptyset\},\label{theta-ini}\nonumber\\
\Theta_1:= \{ i\in \NN: 8Q_i \subset   \RR^n \times (-\infty,t_0+16s]\},\label{theta1-ini}\\
\Theta_2 := \Theta \setminus \Theta_1.\label{theta2-ini}\nonumber
\end{gather}
Recall from \cref{def_eta} that on $[t_0-16s,t_0+16s]$,  we have $\zeta(t) = 1$ and in particular on $Q_i$ for any $i \in \Th_1$.

\subsection{Bounds of \texorpdfstring{$\vlh$}.}
The first lemma is a rough bound of $\vlh$:
\begin{lemma}
\label{lemma6.1-ini}
Let $z \in \elam^c$, then 
\begin{equation*}\label{claim_bound-ini}|\vlh(z)| \apprle_{(n)} \rho \la.\end{equation*}
\end{lemma}
\begin{proof}
From \cref{lipschitz-extension-ini}, we see that $\vlh(z) = \sum_i \Psi_i(z) \vh^i$, which along with \cref{time_average} and \descref{W4} gives the following bound:
\begin{equation*}
|\vh^i|  \leq \rho \fiint_{4Q_i} \frac{|\vh(\tz)|}{\rho}\lsb{\chi}{[0,T]} \ d\tz \apprle \rho \fiint_{16Q_i} \frac{|v(\tz)|}{\rho}\lsb{\chi}{[0,T]} \ d\tz \apprle \rho \la.
\end{equation*}
\end{proof}

Let us now prove an improved bound of $\vlh$.
\begin{lemma}
\label{lemma6.2-ini}
For any $i \in \Th_1$ and any $j \in A_i$ where $\Th_1$ is from \cref{theta1-ini} and  $A_i$ is from \cref{whitney_decomposition}, there holds
\[
|\vh^i - \vh^j| \apprle_{(n,p,\Lambda_1)} \min\{\rho,r_i\} \la.
\]
\end{lemma}

\begin{proof}
If $\rho \leq r_i$, then the result follows from \cref{lemma6.1-ini}. Hence, we only have to consider the case $r_i \leq \rho$.

{\bf Suppose } either $\frac34Q_i \cap \{t\leq 0\} \neq \emptyset$ or $\frac34Q_i \cap (8B)^c \times [0,\infty) \neq \emptyset$, i.e., $\frac34Q_i$ either crosses the initial boundary or the lateral boundary.

\begin{itemize}
\item In the case $\frac34Q_i \cap (8B)^c \times [0,\infty) \neq \emptyset$, then from the fact $j \in A_i$, we can assume $\frac34Q_j \subset 8B \times [0,\infty)$, otherwise, $\vh^j=0$ and there is nothing to prove. Thus, we can apply Poincar\'{e}'s inequality after enlarging to $4Q_j$ which gives $\frac34Q_i \subset 4Q_j$ from \descref{W11}
\begin{equation*}
|\vh^j|  \leq  \fiint_{4Q_j} |\vh(\tz)| \lsb{\chi}{[0,T]}\ d\tz = r_j \fiint_{4Q_j} \frac{|\vh(\tz)|}{r_j} \lsb{\chi}{[0,T]}\ d\tz \apprle r_j \fiint_{4Q_j} |\nabla \vh(\tz)| \lsb{\chi}{[0,T]}\ d\tz \overset{\text{\descref{W4}}}{\apprle} r_j \la.
\end{equation*}

\item In the case $\frac34Q_i \cap \{t\leq 0\} \neq \emptyset$, then again we can assume $\frac34Q_j \subset 8B \times [0,\infty)$, otherwise $\vh^j=0$ and there is nothing to prove. From \descref{W11}, we see that $\frac34Q_i \subset 4Q_j$, which along with \cref{lipschitz-extension-ini} and  a cut-off function $\xi \in C_c^{\infty}(4Q_j \cap \{t\leq 0\})$ implies  $\avgs{\vh\lsb{\chi}{[0,T]}}{\xi}=0$. Thus, for any $\mu \in C_c^{\infty}(4B_j)$ with $|\mu| \leq \frac{C}{r_j^n}$ and $|\nabla \mu| \leq \frac{C}{r_j^{n+1}}$, we get
\begin{equation*}
\begin{array}{rcl}
|\vh^j| & \leq & r_j \fiint_{4Q_j} \frac{\abs{\vh(\tz)\lsb{\chi}{[0,T]} - \avgs{\vh\lsb{\chi}{[0,T]}}{\xi}}}{r_j} \ d\tz  \\
& \apprle & r_j\fiint_{4Q_j} |\nabla \vh|\lsb{\chi}{[0,T]} \ d\tz + \sup_{t_1,t_2 \in 4I_j \cap [0,T]} \abs{\avgs{\vh \lsb{\chi}{[0,T]}}{\mu}(t_2) - \avgs{\vh \lsb{\chi}{[0,T]}}{\mu}(t_1)}\\
& \overset{\redlabel{621a}{a}}{\apprle} & r_j \la.
\end{array}
\end{equation*}
To obtain \redref{621a}{a}, we controlled the first term using \descref{W4} and the second term is controlled similar to \cref{6.20-ini}.
\end{itemize}

{\bf Let us come to the case where $\frac{3}{4}Q_i\subset 8B\times[0,\infty)$.} We see that (see \cite[Lemma 3.7]{AdiByun2} for the details)
\begin{equation}
\label{3.30-ini}
\begin{array}{rcl}
|\vh^i - \vh^j| \apprle \fiint_{Q_i} |\vh(z)\lsb{\chi}{[0,T]} - \vh^i| \ dz + \fiint_{Q_j} |\vh(z)\lsb{\chi}{[0,T]} - \vh^j| \ dz.
\end{array}
\end{equation}
Since $i \in \Th_1$, we must have $\zeta \equiv 1$ on $Q_i$, thus applying \cref{lemma_crucial_1}, for any $\mu \in C_c^{\infty}(B_i)$ with $|\mu| \leq \frac{C}{r_i^n}$ and $|\nabla \mu| \leq \frac{C}{r_i^{n+1}}$, we get
\begin{equation}
\label{6.23-ini}
\fiint_{Q_i} |\vh(z)\lsb{\chi}{[0,T]} - \vh^i| \ dz \apprle  r_i \fiint_{Q_i} |\nabla \vh|\lsb{\chi}{[0,T]} \ dz + \sup_{t_1,t_2 \in I_i} \abs{\avgs{v_h\lsb{\chi}{[0,T]}}{\mu}(t_2) - \avgs{v_h\lsb{\chi}{[0,T]}}{\mu}(t_1)}.
\end{equation}
The first term on the right hand side of \cref{6.23-ini} can be controlled easily using \descref{W4} to get
\begin{equation}
\label{6.24-ini}
\fiint_{Q_i} |\nabla \vh|\lsb{\chi}{[0,T]} \ dz \apprle \la.
\end{equation}

To control the second term on the right hand side of \cref{6.23-ini}, let us  apply \cref{lemma_crucial_2_app} with $\phi(x) = \eta(x) \mu(x)$ and $\varphi(t) \equiv \zeta(t) \equiv 1$ along with  to get
\begin{equation*}
\label{6.17-ini}
\begin{array}{rcl}
\sup_{t_1,t_2 \in I_i\cap [0,T]} \abs{\avgs{v_h}{\mu}(t_2) - \avgs{v_h}{\mu}(t_1)} & \apprle & \lbr \frac{1}{\rho r_i^n} + \frac{1}{r_i^{n+1}} \rbr \iint_{Q_i}  [|\nabla u|+|h_0|]_h^{p-1}\lsb{\chi}{[0,T]}\ dz \\
&& + \lbr \frac{1}{\rho r_i^n} + \frac{1}{r_i^{n+1}} \rbr  \int_{I_i} \left[ \int_{B_i} |\vec{w}| \lsb{\chi}{\qomt} \ dx \right]_h \ dt.
\end{array}
\end{equation*}

Since $|B_i| = c(n) r_i^n$, $|I_i| = \la^{2-p} r_i^2$,  after using the fact that $r_i \leq \rho$, we make use of \cref{imp_rmk} to  get
\begin{equation}
\label{6.20-ini}
\begin{array}{r@{}c@{}l}
\sup_{t_1,t_2 \in I_i} \abs{\avgs{v_h}{\mu}(t_2) - \avgs{v_h}{\mu}(t_1)}  & \apprle &    r_i \la^{2-p} \lbr  \fiint_{16Q_i} (|\nabla u|+|h_0|)^{p-1}\lsb{\chi}{[0,T]}\ dz  + \fiint_{16Q_i} |\vec{w}| \lsb{\chi}{\qomt}\ dz\rbr \\
& \overset{\text{\descref{W4}}}{\apprle} &  r_i \la.
\end{array}
\end{equation}
Thus combining \cref{6.20-ini} and \cref{6.24-ini} into \cref{6.23-ini} gives the desired estimate which proves the lemma.
\end{proof}


\subsection{Bounds on derivatives of \texorpdfstring{$\vlh$}.}
Using bounds from \cref{lemma6.1-ini} and \cref{lemma6.2-ini}, following the calculations in \cite{AdiByun2}, we can obtain the following lemmas which estimate the derivatives of $\vlh$:

\begin{lemma}
\label{lemma6.3-ini}
Given any $z \in \elam^c$, from \cref{whitney_decomposition}, we have $z \in \frac34Q_i$ for some $i \in \NN$. If either $i \in \Th_1$ or $i \in \Th_2$ with $\rho \leq r_i$, then 
\begin{equation*}\label{3.34-ini}
|\nabla \vlh| \apprle_{(n,p,\Lambda_1)} \la.
\end{equation*}
\end{lemma}

\begin{lemma}
 \label{lemma3.10.1-ini}
 Let $z \in \elam^c$ and $\ve >0$ be any number, then there exists a constant $C(n)$ such that the following holds:
 \begin{align*}
  |\vlh(z)| & \leq C \fiint_{4Q_i}|\vh(\tz)|\lsb{\chi}{[0,T]} \ d\tz \leq  \frac{Cr_i\la}{\varepsilon} + \frac{C\varepsilon}{\la r_i} \fiint_{4Q_i}|\vh(\tz)|^2 \lsb{\chi}{[0,T]}\ d\tz, \\
  |\nabla \vlh(z)| &\leq C \frac{1}{r_i} \fiint_{4Q_i}|\vh(\tz)| \lsb{\chi}{[0,T]}\ d\tz \leq  \frac{C \la}{\varepsilon} + \frac{C\varepsilon}{\la r_i^2} \fiint_{4Q_i}|\vh(\tz)|^2 \lsb{\chi}{[0,T]}\ d\tz. 
 \end{align*}
\end{lemma}

\begin{lemma}
 \label{lemma3.10.2-ini}
 Let $z \in \elam^c$ and $\ve \in (0,1)$ be given, from \cref{whitney_decomposition}, we have  $z \in \frac34Q_i$ for some $i \in \NN$. Suppose $i \in \Th_1$, then there holds:
 \begin{align*}
  |\vlh(z)|& \leq C_{(n,p,\Lambda_1)} \lbr \min\{ \rho, r_i\} \la + |\vh^i| \rbr \leq  C_{(n,p,b_0,r_0,r_0,\Lambda_1)} \lbr \frac{ r_i\la}{\ve} + \frac{\ve}{r_i \la} |\vh^i|^2 \rbr  \\
  |\nabla \vlh(z)| &\leq C_{(n)} \frac{\la}{\ve}.
 \end{align*}
\end{lemma}

\begin{lemma}
 \label{lemma3.10.3-ini}
 Let $z \in \elam^c$, then from \cref{whitney_decomposition}, we have $z \in \frac34Q_i$ for some $i \in \NN$. Suppose $i \in \Th_2$, then there holds:
  \begin{align*}  
  | \vlh(z)| &\leq C_{(n)}\lbr  r_i \la + \frac{\la^{1-p}r_i}{s} \fiint_{4Q_i}|\vh(\tz)|^2 \lsb{\chi}{[0,T]}\ d\tz \rbr,\\
  |\nabla \vlh(z)| &\leq C_{(n)}\lbr  \la + \frac{\la^{1-p}}{s} \fiint_{4Q_i}|\vh(\tz)|^2 \lsb{\chi}{[0,T]}\ d\tz\rbr.
 \end{align*}
\end{lemma}

\begin{lemma}
\label{lemma3.11-ini}
Let $z \in \elam^c$, then from \cref{whitney_decomposition}, there exists an $i \in \NN$ such that $z \in \frac34Q_i$. Then the following estimates for the time derivative of $\vlh$ holds:
 \begin{align*}
 |\pa_t \vlh(z)| & \leq C_{(n)} \frac{1}{\ga r_i^2} \fiint_{4Q_i} |\vh(\tz)| \lsb{\chi}{[0,T]}\ d\tz. \label{lemma3.11.bound1-ini} 
 \end{align*}
 If $i \in \Th_1$, then we have
 \begin{align*}
   |\pa_t \vlh(z)| & \leq C_{(n,p,\Lambda_1)} \frac{1}{\ga r_i^2} \min\{r_i,\rho\} \la. 
  \end{align*}
  If $i \in \Th_2$, then there holds
  \begin{align*}
   |\pa_t \vlh(z)| & \leq C_{(n)} \frac{\rho \la}{s}.
 \end{align*}

\end{lemma}

As a consequence of the above lemmas, we have the following lemma which controls the integral of $\vlh$.
\begin{lemma}
\label{lemma3.12-ini}
From \cref{lipschitz-extension-ini}, we extend $\vlh$ by zero in the region $8Q \setminus \elam \cap \{t \leq 0\}$. From this,  for any $\vartheta \in [1,p-\be]$,  we have the following bound:
 \begin{equation*}
  \iint_{8Q\setminus\elam } |\vlh(z)|^{\vartheta} \ dz \leq C_{(n,p,\Lambda_1)} \iint_{8Q\setminus\elam } |\vh(z)|^{\vartheta} \lsb{\chi}{[0,T]}\ dz.
 \end{equation*}

\end{lemma}

\subsection{Two important intermediate estimates}

In order to prove the Lipschitz continuity of $\vlh$, we need to obtain a suitable control of integrals over $Q_i$, which will be done in the following two lemmas. The first one is an estimate for cylinders in $\Th_1$.
\begin{lemma}
\label{lemma3.13-ini}
 For any  $i \in \Th_1$,  we have the following estimate: 
 \begin{equation*}
 \label{3.59-ini}
  \fiint_{\frac34Q_i} \left|\frac{\vh(z) - \vlh(z)}{r_i}\right|^q \lsb{\chi}{[0,T]}\ dz \leq C_{(n,p,q,\lamot)} \la^q.
 \end{equation*}

\end{lemma}

\begin{proof}
 For any  $z \in \frac34Q_i$, using \cref{lipschitz-extension-ini} along with triangle inequality and  \descref{W13}, we get
 \begin{equation}
 \label{3.60-ini}
    \fiint_{\frac34Q_i}|\vh(z) - \vlh(z)|^q\lsb{\chi}{[0,T]} \ dz 
   \leq \fiint_{\frac34Q_i}|\vh(z)\lsb{\chi}{[0,T]} - \vh^i|^q \ dz +  \sum_{j:j \in A_i}\fiint_{\frac34Q_i} \left|\vh^j - \vh^i \right|^q\  dz := J_1 + J_2.
 \end{equation}

 We shall estimate each of the terms of \cref{3.60-ini} as follows (note that $i \in \Th_1$). Note that $J_1$ is exactly as in \cref{3.30-ini}, which implies 
  \begin{equation}
  \label{3.61-ini}
   J_1 \apprle (r_i \la)^q. 
  \end{equation}
  In order to estimate $J_2$, we can directly use \cref{lemma6.2-ini} to get 
  \begin{equation}
  \label{3.62-ini}
   J_2 \apprle (r_i \la)^q. 
  \end{equation}
  Substituting \cref{3.61-ini} and \cref{3.62-ini} into \cref{3.60-ini} and making use of  \descref{W8}, the lemma follows.
\end{proof}

The second lemma is more involved to prove, as it concerns cylinders in $\Th_2$. 
\begin{lemma}
\label{lemma2.5-ini}
Given any  $i \in \Th_2$ and $8Q_i \subset \RR^n \times \RR^+$, let $\al_0$ and $c_e$ be as in \cref{alpha_0},  then for any $\la \geq c_e \al_0$, there holds
 \[\begin{array}{rcl}
   \fiint_{\frac34Q_i} \left| \frac{\vh(z)\lsb{\chi}{[0,T]} - \vh^i}{r_i} \right|^q \ dz & \apprle_{(n,p,q,\lamot,c_e)}  &\la^q \lbr 1+ \lbr \frac{\la^{2-p}}{\al_0^{2-p}}\rbr+ \lbr \frac{\la^{2-p}}{\al_0^{2-p}}\rbr^{\frac{n+1}{2}}  + \lbr \frac{\la^{2-p}}{\al_0^{2-p}}\rbr^{\frac{n-1}{2}}\rbr^q.
 \end{array}
 \]
\end{lemma}

\begin{proof} 
Let us first note that  without loss of generality, we can take $r_i \leq \rho$, otherwise we can directly apply \cref{lemma6.2-ini} in the case of $\rho \leq r_i$ to get
\begin{equation}
\label{6.42-ini}
\fiint_{\frac34Q_i} \left| \frac{\vh(z)\lsb{\chi}{[0,T]} - \vh^i}{r_i} \right|^q \ dz  =  \fiint_{\frac34Q_i} \left| \frac{\vh(z)\lsb{\chi}{[0,T]}}{r_i} \right|^q \ dz  \overset{\text{\cref{lemma6.1-ini}}}{\apprle}  \lbr \frac{\rho}{r_i}\rbr^q \la^q \apprle \la^q.
\end{equation}

$\bullet$\emph{If $\vh^i =0$, then either $\frac34Q_i \cap (8B)^c \times [0,\infty) \neq \emptyset$ or $\frac34Q_i \cap \{t\leq 0\} \neq \emptyset$.} In each of those cases, we can estimate as follows:
\begin{description}
\item[In the case $\frac34Q_i$ crosses the lateral edge,]  we can directly apply \cref{sobolev-poincare} with $\ka =1$ to get
\begin{equation*}
\fiint_{\frac34Q_i} \left| \frac{\vh(z)\lsb{\chi}{[0,T]}}{r_i} \right|^q \ dz  \apprle \fiint_{\frac34Q_i} |\nabla \vh(z)|^q \lsb{\chi}{[0,T]} \ dz \overset{\text{\descref{W4}}}{\apprle} \la^q.
\end{equation*}

\item[In the case $\frac34Q_i$ crosses the initial edge,] we take a cut-off function $\xi \in C_c^{\infty}(Q_j \cap \{t\leq 0\})$ from which we see that $\avgs{\vh}{\xi} =0$. Making use of \cref{lemma_crucial_1}, for any $\mu \in C_c^{\infty}(B_i)$ with $|\mu| \leq \frac{C}{r_i^n}$ and $|\nabla \mu| \leq \frac{C}{r_i^{n+1}}$, we get
\begin{equation}
\label{lemma2.5.8.1-ini.1}
\begin{array}{r@{}c@{}l}
\hspace*{-1.0cm}\fiint_{\frac34Q_i} \left| \frac{\vh(z)\lsb{\chi}{[0,T]}}{r_i} \right|^q \ dz &\ \apprle \ & \fiint_{Q_i} \left| \frac{\vh(z)\lsb{\chi}{[0,T]} - \avgs{\vh}{\xi}}{r_i} \right|^q \ dz\\
& {\apprle}& \fiint_{Q_i} |\nabla \vh|^q\lsb{\chi}{[0,T]} \ dz + \sup_{t_1,t_2 \in I_i} \left| \frac{\avgs{\vh\lsb{\chi}{[0,T]}}{\mu}(t_1)-\avgs{\vh\lsb{\chi}{[0,T]}}{\mu}(t_2)}{r_i} \right|^q .
\end{array}
\end{equation}
The first term on the right hand side of \cref{lemma2.5.8.1-ini.1} can be estimated using \descref{W4},  to estimate the second term on the right hand side of \cref{lemma2.5.8.1-ini.1}, let us apply \cref{lemma_crucial_2_app} with $\phi(x) = \mu(x) \eta(x)$ and $\varphi(t) = \zeta(t)$, which gives for any $t_1,t_2 \in I_i\cap[0,T]$, the following sequence of estimates:
  \begin{equation}
  \label{lemma2.5.8.1}
  \begin{array}{rcl}
  |\avgs{\vh}{\mu}(t_2)- \avgs{\vh}{\mu}(t_1)| & \overset{\redlabel{645a}{a}}{\apprle} & \lbr \frac{1}{\rho r_i^n} + \frac{1}{r_i^{n+1}} \rbr \int_{t_1}^{t_2}\int_{\frac34B_i} [|\nabla u|+|h_0|]_h^{p-1} \ dz \\
&&+\lbr \frac{1}{\rho r_i^n} + \frac{1}{r_i^{n+1}} \rbr \int_{t_1}^{t_2}\left[\int_{\frac34B_i}|\vec{w}| \lsb{\chi}{\qomt} \ dx \right]_h \ dt\\
&&+ \frac{1}{r_i^n} \frac{1}{s} \int_{t_1}^{t_2}\int_{\frac34B_i}|[u-w]_h| \lsb{\chi}{[0,T]}\ dz\\
    & \overset{\redlabel{645b}{b}}{\apprle} & r_i \la^{2-p} \fiint_{4Q_i} [|\nabla u|+|h_0|]_h^{p-1} \lsb{\chi}{[0,T]}\ dz\\
&&+r_i\la^{2-p}\fint_{t_1}^{t_2}\left[\fint_{\frac34B_i}|\vec{w}| \lsb{\chi}{\qomt} \ dx \right]_h \ dt \\
&&+ \frac{r_i^3 \la^{2-p}}{s} \fiint_{4Q_i}|[u-w]_h| \lsb{\chi}{[0,T]}\ dz.
\end{array}
\end{equation}
To obtain \redref{645a}{a} and \redref{645b}{b}, we made use of \cref{def_eta}, \cref{def_zeta} and the bounds for $\mu$ along with \descref{W1} and \cref{imp_rmk}.
The first term and second term on the right hand side of \cref{lemma2.5.8.1} can be controlled using \descref{W4} to get
\begin{equation}
\label{6.46-ini}
r_i \la^{2-p} \fiint_{4Q_i} [|\nabla u|+|h_0|]_h^{p-1} \lsb{\chi}{[0,T]}\ dz+r_i\la^{2-p}\fint_{t_1}^{t_2}\left[\fint_{\frac34B_i}|\vec{w}| \lsb{\chi}{\qomt} \ dx \right]_h \ dt \apprle r_i \la.
\end{equation}
To estimate the third term on the right hand side of \cref{lemma2.5.8.1}, let us take a cut-off function $\xi \in C_c^{\infty}(4Q_i \cap \{t\leq 0\})$, from which we note that for $h$ sufficiently small, using \cref{lemma_crucial_1} and the hypothesis that $Q_i$ crosses the initial boundary, we observe $\avgs{[u-w]_h}{\xi}=0$. Thus we get 
\begin{equation}
\label{6.47-ini}
\begin{array}{rcl}
  \fiint_{4Q_i}\frac{|[u-w]_h|}{r_i} \lsb{\chi}{[0,T]}\ dz  & = &   \fiint_{4Q_i}\frac{|[u-w]_h\lsb{\chi}{[0,T]} - \avgs{[u-w]_h}{\xi}|}{r_i} \ dz\\
    & \overset{\redlabel{6451a}{a}}{\apprle} &   \fiint_{4Q_i}|[\nabla u-\nabla w]_h| \lsb{\chi}{[0,T]} \ dz\vspace{1em}\\
&& +  \sup_{t_1,t_2 \in 4I_i \cap [0,T]}\frac{|\avgs{(u-w)_h}{\mu}(t_1) - \avgs{(u-w)_h}{\mu}(t_2)|}{r_i} \ dz\vspace{1em}\\
    & \overset{\redlabel{6451b}{b}}{\apprle} &  \la + \frac{\|\nabla \mu\|_{L^{\infty}} |4Q_i|}{r_i}\fiint_{4Q_i} [|\nabla u|+|h_0|]_h^{p-1} \lsb{\chi}{[0,T]} \ dz\\
&&\hspace{1em}+\frac{\|\nabla \mu\|_{L^{\infty}} |4Q_i|}{r_i}\fint_{4I_i\cap[0,T]}\left[\fint_{\frac34B_i}|\vec{w}| \lsb{\chi}{\qomt} \ dx \right]_h \ dt\\
    & \overset{\redlabel{6451c}{c}}{\apprle} & \la.
  \end{array}
 \end{equation}
 To obtain \redref{6451a}{a}, we made use of \cref{lemma_crucial_1}, to obtain \redref{6451b}{b}, we made use of \cref{lemma_crucial_2_app} along with \descref{W4} and finally to obtain \redref{6451c}{c}, we used \descref{W1} and \descref{W4} along with \cref{imp_rmk}.

  Combining \cref{6.47-ini}, \cref{6.46-ini} and \cref{lemma2.5.8.1} into \cref{lemma2.5.8.1-ini.1} followed by the restriction $r_i \leq \rho$ and \cref{alpha_0}, we get
  \begin{equation*}
\fiint_{\frac34Q_i} \left| \frac{\vh(z)\lsb{\chi}{[0,T]}}{r_i} \right|^q \ dz  \apprle \la^q + \la^q \lbr \frac{\la^{2-p}}{\al_0^{2-p}} \rbr^q.
  \end{equation*}
\end{description}

$\bullet$\emph{Now we consider the case when $\vh^i \neq 0$.} Again without loss of generality, we can assume $r_i \leq \rho$ because if $\rho \leq r_i$, we can proofed as in \cref{6.42-ini}. Since $i \in \Th_2$, we have $\ga r_i^2 \apprge s$. Now applying \cref{lemma_crucial_1} with $\mu(x) \in C_c^{\infty} (B_i)$ satisfying $|\mu| \leq \frac{C(n)}{r_i^n}$ and $|\nabla \mu| \leq \frac{C(n)}{r_i^{n+1}}$,  we get
 \begin{equation}\label{lemma2.5.8.1-ini}
  \begin{array}{rcl}
   \fiint_{\frac34Q_i} \left| \frac{\vh(z)\lsb{\chi}{[0,T]} - \vh^i}{r_i} \right|^q \ dz 
   & \apprle& \fiint_{Q_i} |\nabla \vh|^q\lsb{\chi}{[0,T]} \ dz + \sup_{t_1,t_2 \in I_i} \left| \frac{\avgs{\vh\lsb{\chi}{[0,T]}}{\mu}(t_1)-\avgs{\vh\lsb{\chi}{[0,T]}}{\mu}(t_2)}{r_i} \right|^q .
  \end{array}
 \end{equation}
 The first term on the right hand side of \cref{lemma2.5.8.1-ini} can be estimated using \descref{W4}. To estimate the second term on the right hand side of \cref{lemma2.5.8.1-ini}, let us apply \cref{lemma_crucial_2_app} with $\phi(x) = \mu(x) \eta(x)$ and $\varphi(t) = \zeta(t)$, which gives for any $t_1,t_2 \in I_i\cap[0,T]$, the following sequence of estimates:
   \begin{equation}
  \label{lemma2.5.8-11}
  \begin{array}{rcl}
  |\avgs{\vh}{\mu}(t_2)- \avgs{\vh}{\mu}(t_1)| & {\apprle}& \|\nabla (\eta \mu)\|_{L^{\infty}(\frac34B_i)} \|\zeta\|_{L^{\infty}(t_1,t_2)}\int_{t_1}^{t_2}\int_{\frac34B_i}[|\nabla u|+|h_0|]_h^{p-1} \lsb{\chi}{[0,T]}\ dz\\
&&+\|\nabla (\eta \mu)\|_{L^{\infty}(\frac34B_i)} \|\zeta\|_{L^{\infty}(t_1,t_2)}\int_{t_1}^{t_2}\left[\int_{\frac34B_i}\lvert \vec{w}\rvert\ dx \right]_h\ dt\\
  && + \|(\eta\mu)\|_{L^{\infty}(\frac34B_i)}\|\zeta'\|_{L^{\infty}(t_1,t_2)}\int_{t_1}^{t_2}\int_{8B}|[u-w]_h| \lsb{\chi}{[0,T]}\ dz\\
  & \overset{\redlabel{650a}{a}}{\apprle} & \lbr \frac{1}{\rho r_i^n} + \frac{1}{r_i^{n+1}} \rbr \int_{t_1}^{t_2}\int_{\frac34B_i} [|\nabla u|+|h_0|]_h^{p-1} \ dz\\
  &&+ \lbr \frac{1}{\rho r_i^n} + \frac{1}{r_i^{n+1}}\rbr\int_{t_1}^{t_2}\left[\int_{\frac34B_i}\lvert \vec{w}\rvert\ dx \right]_h\ dt+ \frac{1}{r_i^n} \frac{1}{s} \int_{t_1}^{t_2}\int_{8B}|[u-w]_h| \lsb{\chi}{[0,T]}\ dz\\
    & \overset{\redlabel{650b}{b}}{\apprle} & r_i \la^{2-p} \fiint_{4Q_i} [|\nabla u|+|h_0|]_h^{p-1} \lsb{\chi}{[0,T]}\ dz \\
&&+r_i\la^{2-p}\fint_{t_1}^{t_2}\left[\fint_{\frac34B_i}\lvert \vec{w}\rvert\ dx \right]_h\ dt+ \frac{1}{r_i^n} \frac{\rho |Q|}{s} \fiint_{8Q}\frac{|[u-w]_h|}{\rho} \lsb{\chi}{[0,T]}\ dz.
\end{array}
\end{equation}
To obtain \redref{650a}{a} and \redref{650b}{b}, we made use of \cref{def_eta}, \cref{def_zeta} and the bounds for $\mu$ along with \descref{W1}. The first term on the right hand side of \cref{lemma2.5.8-11} can be controlled by \descref{W4}  to get an estimate analogous to \cref{6.46-ini}. To control the second term on the right hand side of \cref{lemma2.5.8-11}, we note that for $h$ sufficiently small,  \cref{lemma_crucial_1} and the hypothesis that $8Q$ crosses the initial boundary imply $\avgs{[u-w]_h}{\xi}=0$. Then 
  \begin{equation}
  \label{lemma2.5.8}
  \begin{array}{rcl}
  \fiint_{8Q}\frac{|[u-w]_h|}{\rho} \lsb{\chi}{[0,T]}\ dz & = &  \fiint_{8Q}\frac{|[u-w]_h \lsb{\chi}{[0,T]} - \avgs{[u-w]_h}{\xi}|}{\rho}\ dz\\
  & \overset{\redlabel{651a}{a}}{\apprle} & \fiint_{8Q}|[\nabla u-\nabla w]_h| \lsb{\chi}{[0,T]}\  dz\\
&& +  \sup_{t_1,t_2 \in 8I \cap [0,T]} \frac{|\avgs{[u-w]_h}{\mu}(t_2) - \avgs{[u-w]_h}{\mu}(t_1)|}{\rho}\\
    & \overset{\redlabel{651b}{b}}{\apprle} &  \la   +  \frac{\|\nabla \mu\|_{L^{\infty}(8B)}|Q|}{\rho} \fiint_{8Q} [|\nabla u|+|h_0|]_h^{p-1} \lsb{\chi}{[0,T]} \ dz\\
&&\hspace{1em}+\frac{\|\nabla \mu\|_{L^{\infty}(8B)}|Q|}{\rho}\fint_{8I\cap[0,T]}\left[\fint_{8B}\lvert \vec{w}\rvert\ dx\right]_h\ dt\\
    &\overset{\redlabel{651c}{c}}{\apprle} & \la + \frac{s}{\rho^2} \la^{p-1}.
  \end{array}
 \end{equation}
 To obtain \redref{651a}{a}, we made use of \cref{lemma_crucial_1}, to obtain \redref{651b}{b}, we made use of \cref{lemma_crucial_2_app} along with \descref{W4} and finally to obtain \redref{651c}{c}, we made use of \descref{W4} along with \cref{def_eta}.
 
 Combining \cref{lemma2.5.8}, \cref{lemma2.5.8-11} and \cref{lemma2.5.8.1-ini} along with  the bound $\la^{2-p} r_i^2 \apprge s$ and \cref{hypothesis_1-ini}, we get
 \begin{equation*}
  \fiint_{\frac34Q_i} \left| \frac{\vh(z)\lsb{\chi}{[0,T]} - \vh^i}{r_i} \right|^q \ dz \apprle \la^q \lbr 1+  \lbr \frac{\la^{2-p}}{\al_0^{2-p}}\rbr^{\frac{n+1}{2}}  + \lbr \frac{\la^{2-p}}{\al_0^{2-p}}\rbr^{\frac{n-1}{2}}\rbr^q.
 \end{equation*}
 This completes the proof of the Lemma. 
 \end{proof}

 \begin{remark}
 In the case $p \geq 2$, the estimate in \cref{lemma2.5-ini} takes the form
 \[
 \fiint_{\frac34Q_i} \left| \frac{\vh(z)\lsb{\chi}{[0,T]} - \vh^i}{r_i} \right|^q \ dz\apprle_{(n,p,q,\lamot,c_e)} \la^q.
 \]
To obtain analogous cleaner estimate  in the case $p< 2$, we can use the unified intrinsic scaling approach developed in \cite{adimurthi2018interior}. This cleaner estimate will not be needed in this paper and hence we leave the details to the interested reader.
 \end{remark}

\subsection{Lipschitz continuity of the test function}
\begin{lemma}
\label{lemma3.15-ini}
The extension $\vlh$ from \cref{lipschitz-extension-ini} is $C^{0,1}(2\mch)$ with respect to the parabolic metric \cref{parabolic_metric}. Here $\mch$ is as defined in \cref{h-ini}.
\end{lemma}

\subsection{Two crucial estimates for the test function}
Before we state the two crucial lemmas, let us collect a few consequences of the estimates proved in the previous subsections. The first estimate is very similar to \cite[Lemma 3.16]{AdiByun2}.
\begin{lemma}
\label{lemma3.14-ini}
For any $1\le\vartheta\le q$, there exists a positive constant $C(n,p,\Lambda_1,\vartheta)>0$ such that the following holds:
 \begin{equation*}
  \iint_{2\mch \setminus \elam} |\pa_t \vlh(z)  (\vlh(z) - \vh(z))|^\vartheta \ dz \apprle_C  \la^{\vartheta p} |\RR^{n+1} \setminus \elam| + \frac{1}{s} \iint_{8Q} |\vh(z)|^{2\vartheta} \lsb{\chi}{[0,T]}\ dz. 
 \end{equation*}

\end{lemma}

Analogous to \cite[Lemma 3.19]{AdiByun2}, we have the following lemma:
\begin{lemma}
 \label{lemma3.18-ini}
 For any $i \in \Th$ and $k \in \{0,1\}$, there exists a positive  constant $C{(n,p,q,\lamot)}$ such that there holds:
 \begin{equation*}
 \label{lemma3.18.bound_1-ini}
  \iint_{\frac34Q_i} [|\nabla u|+|h_0|]_h^{p-1} |\nabla^k \vlh|\lsb{\chi}{[0,T]} \ dz \leq C \rho^{1-k} \lbr \la^p |4Q_i| + \frac{\lsb{\chi}{i \in \Th_2}}{s} \iint_{4Q_i} |\vh|^2\lsb{\chi}{[0,T]} \ dz \rbr. 
 \end{equation*}
Here we have used the notation  $\lsb{\chi}{i\in \Th_2} = 1$ if $i \in \Th_2$ and $\lsb{\chi}{i\in \Th_2} = 0$ if $i \in \Th_1$ and $\nabla^0 \vlh := \vlh$. 

We also have the estimate
\[
\iint_{\frac34Q_i} [|\vec{w}|]_h |\nabla^k \vlh|\lsb{\chi}{[0,T]} \ dz \leq C \rho^{1-k} \lbr \la^p |4Q_i| + \frac{\lsb{\chi}{i \in \Th_2}}{s} \iint_{4Q_i} |\vh|^2\lsb{\chi}{[0,T]} \ dz \rbr. 
\]
\end{lemma}

\begin{corollary}
 \label{corollary3.20-ini}
There exists a positive constant $C{(n,p,q,\lamot)}$ such that  the following estimate holds for any $k \in \{0,1\}$:
 \begin{equation*}
  \label{corollary3.20.1-ini}
  \iint_{8B \times 2\tm \setminus \elam} [|\nabla u|+|h_0|]_h^{p-1} |\nabla^k \vlh| \lsb{\chi}{[0,T]}\ dz \leq C \rho^{1-k} \lbr \la^p |\RR^{n+1} \setminus \elam| + \frac{1}{s} \iint_{8Q} |\vh|^2 \lsb{\chi}{[0,T]}\ dz \rbr. 
 \end{equation*}
 We also have the estimate
 \[
 \iint_{8B \times 2\tm \setminus \elam} [|\vec{w}|]_h |\nabla^k \vlh| \lsb{\chi}{[0,T]}\ dz \leq C \rho^{1-k} \lbr \la^p |\RR^{n+1} \setminus \elam| + \frac{1}{s} \iint_{8Q} |\vh|^2 \lsb{\chi}{[0,T]}\ dz \rbr. 
 \]
\end{corollary}
The first crucial estimate on each time slice follows analogous to \cite[Lemma 3.21]{AdiByun2} and takes the form.
\begin{lemma}
 \label{pre_crucial_lemma-ini}
 For any $i \in \Th_1$ and any $0 < \ve \leq 1$, for  almost every $t \in (0,t_0+4s)=: 2I \cap [0,T]$, there holds
 \begin{equation*}
 \label{3.120-ini}
  \left| \int_{16B} (v(x,t) - v^i) \vl(x,t) \Psi_i(x,t) \ dx \right| \apprle_{(n,p,\Lambda_1)}  \lbr  \frac{\la^p}{\ve} |4Q_i| + \ve |4B_i| |v^i|^2\rbr. 
 \end{equation*}
 
 In the case $i \in \Th_2$,  for  almost every $t \in (t_0+0,4s)$, there holds 
 \begin{equation*}
 \label{3.121-ini}
  \left| \int_{16B} v(x,t) \vl(x,t) \Psi_i(x,t) \ dx \right| \apprle_{(n,p,\Lambda_1)} \lbr  {\la^p} |4Q_i| + \frac{1}{s}\iint_{\frac34Q_i} |u-w|^2 \lsb{\chi}{[0,T]}\ dz\rbr. 
 \end{equation*}

\end{lemma}
We now come to  essentially the most important estimate which will be used to obtain the Caccioppoli inequality. The proof is very similar to \cite[Lemma 3.22]{AdiByun2} and will be omitted.
\begin{lemma}
 \label{crucial_lemma-ini}
 There exists a positive constant $C{(n,p,q, \lamot)}$ such that the following estimate holds for every $t \in [0,t_0+4s] =: 2I \cap [0,T]$:
 \begin{equation*}
 \label{3122-ini}
  \int_{8B \setminus \elam^t} (|\vh|^2 - |\vh - \vl|^2)(x,t) \ dx \geq C \lbr - \la^p |\RR^{n+1} \setminus \elam|  - \frac{1}{s} \iint_{8Q} |u-w|^2 \lsb{\chi}{[0,T]}\ dz \rbr.
 \end{equation*}
\end{lemma}

\subsection{Caccioppoli type inequality}

We shall prove the Caccioppoli inequality in this subsection.
\begin{lemma}
 \label{caccioppoli}
 Let $\al_0$ and  $c_e$ be as in \cref{alpha_0}, then there exists   constants $C = C(n,p,q,\La_0,\La_1)$ and $\be_0 = \be_0(n,p,\lamot) \in (0,1)$ small  such that the following holds.
For some $\be \in (0,\be_0)$, suppose that  $u \in L^2(0,T;L^2(\Om)) \cap L^{p-\be}(0,T; W_{\loc}^{1,p-\be}(\Om))$  is  any very weak solution of \cref{main-2} in the sense of \cref{very_weak_solution}, then there holds
\begin{equation*}
\label{conclusion_1}
 \begin{array}{ll}
        &\al_0^{p-\be} + \sup_{t \in \tm \cap \{t \geq 0\}} \al_0^{p-2} \hint_{B} \mathcal{M}(x,t)^{-\be} \left| \frac{ u-w}{\rho}\right|^2(x,t) \ dx   \vspace{1em}\\
&\hspace*{4cm}\apprle  \fiint_{8Q}\left[ \al_0^{p-2-\be}   \lbr \frac{|u-w|}{\rho}\rbr^2 +  \lbr \frac{|u-w|}{\rho} \rbr^{p-\be} \right]\lsb{\chi}{[0,T]}\ dz\vspace{1em}\\
        &\hspace*{4.5cm} + \fiint_{8Q} |h_0|^{p-\be}\lsb{\chi}{[0,T]} \ dz +  \fiint_{8Q}|\nabla w|^{p-\be} \lsb{\chi}{[0,T]}\ dz  +  \fiint_{8Q}   |\vec{w}|^{\frac{p-\be}{p-1}} \lsb{\chi}{\qomt}\ dz,
 \end{array}
\end{equation*} 
 where we have set $\mathcal{M}(x,t):=\max\{g(x,t), \al_0\}$.
\end{lemma}

\begin{proof}
Pick any $t_1 \in (0,t_0+s)$ and consider the cut-off function $\chi_{0,t_1}^{\ve} \in C_c^{\infty}(0,t_1)$ such that
\begin{equation}
\label{zeta_t_1}
\chi_{0,t_1}^{\ve}(t) = \left\{ \begin{array}{ll}
                1 & \text{for} \ t \in (0+\ve,t_1-\ve)\\
                0 & \text{for} \ t \in (-\infty,0)\cup (t_1,\infty).
                \end{array}\right.
\end{equation}
Let us use $ \vlh(x,t)\eta(x) \chi_{0,t_1}^{\ve}$ as a test function in \cref{main-2} where $\vlh$ is from \cref{lipschitz-extension-ini} and $\eta$ is from \cref{def_eta}. Integrating over $(0,t_1)$,  we get
\begin{equation}
 \label{cac1-ini}
  L_1 + L_2:=\int_{0}^{t_1} \left[ \int_{16B} \frac{d{[u]_h}}{dt} \eta(x) \vlh(x,t) + \iprod{[\aa(x,t,\nabla u)]_h}{\nabla (\eta \vlh)} \ dx\right] \chi_{0,t_1}^{\ve}(t)\ dt = 0.
 \end{equation}
\begin{description}
\item[Estimate of $L_1$:] Note $\zeta \equiv 1$ on $(0,t_1)$, from which we get
 \begin{equation}
 \label{6.51-ini}
 \begin{array}{rcl}
 \int_0^{t_1} \int_{16B} \ddt{[u]_h} \vlh \eta \chi_{0,t_1}^{\ve}(t)\ dz 
 & = & \int_0^{t_1} \int_{16B} \ddt{\vlh} (\vlh - v_h)  \chi_{0,t_1}^{\ve}(t)\ dz  \\
 && +\frac{1}{2} \int_0^{t_1} \int_{16B} \ddt{\lbr (v_h^2) - (\vlh-v_h)^2\rbr\chi_{0,t_1}^{\ve}(t)}   \ dz \\
 && - \frac{1}{2}\int_0^{t_1} \int_{16B} {\lbr (v_h^2) - (\vlh-v_h)^2\rbr}   \ddt{\chi_{0,t_1}^{\ve}(t)}\ dz \\
 && + \int_0^{t_1} \int_{16B} \ddt{[w]_h} \eta \vlh \chi_{0,t_1}^{\ve}(t)   \ dz.
 \end{array}
 \end{equation}
 From \cref{zeta_t_1}, we see that  $\int_0^{t_1} \int_{16B} \ddt{\lbr (v_h^2) - (\vlh-v_h)^2\rbr\chi_{0,t_1}^{\ve}(t)}   \ dx \ dt = 0$ since $\chi_{0,t_1}^{\ve}(0) = 0$ and $\chi_{0,t_1}^{\ve}(t_1) = 0$.
 
Letting $\ve \rightarrow 0$ in \cref{6.51-ini}, we  get
\begin{equation}\label{6.38-ini}
 \begin{array}{rcl}
 \int_0^{t_1} \int_{16B} \ddt{[u]_h} \vlh \eta \ dz & = &\int_0^{t_1} \int_{16B} \ddt{\vlh} (\vlh - \vh)  \ dz  \\
 && -  \frac{1}{2}\int_{16B} {\lbr (\vh^2) - (\vlh-\vh)^2\rbr} (x,0)  \ dx  + \\
 && +  \frac{1}{2}\int_{16B} {\lbr (\vh^2) - (\vlh-\vh)^2\rbr} (x,t_1)  \ dx  \\
 && + \int_0^{t_1} \int_{16B} \ddt{[w]_h} \eta(x) \vlh(z)    \ dz \\
 & = & J_2 - J_1(0) + J_1(t_1) + J_3.
 \end{array}
 \end{equation}
Let us now estimate each of the terms as follows:
\begin{description} 
   \item[Estimate of $J_2$:] Taking absolute values and making use of \cref{lemma3.14-ini}, we get
   \begin{equation}
    \label{6.39-ini}
    \begin{array}{r@{}c@{}l}
    \hspace*{-1cm}|J_2| 
      \apprle \iint_{2\mathcal{H}\setminus \elam}   \left| \ddt{\vlh}  (\vlh-\vh)\right| \lsb{\chi}{[0,T]} \ dz 
       &\apprle & \la^p |\RR^{n+1} \setminus \elam| + \frac{1}{s} \iint_{8Q} |[u-w]_h|^2 \lsb{\chi}{[0,T]} \ dz\\
       &\overset{\lim_{h \searrow 0}}{=} & \la^p |\RR^{n+1} \setminus \elam| + \frac{1}{s} \iint_{8Q} |u-w|^2 \lsb{\chi}{[0,T]} \ dz.
    \end{array}
   \end{equation}
   
   \item[Estimate of $J_1(0)$:] Since we have $v=0$ on $\{t=0\}$, we see that $v = \vl = 0$ on $\elam \cap \{t=0\}$ and on $\elam^c\cap\{t=0\}$, we have $\vl = 0$ from \cref{lipschitz-extension-ini} and hence 
   \begin{equation}
   \label{est_J_1-ini}
   \lim_{h \searrow 0} J_1(0) = 0.
   \end{equation}
 \item[Estimate for $J_1(t_1)$:] We can take $h \searrow 0$ followed by  making use of \cref{crucial_lemma-ini},  we get
 \begin{equation}
 \label{4.13-ini}
  \begin{array}{rcl}
    \lim_{h \searrow 0} J_1(t_0) & = &\frac{1}{2} \int_{16B} | (v)^2 - (\vl - v)^2 | (y,t) \ dy   \\
    & \apprge &\int_{\elam^t} | v (x,t)|^2 \ dx   - \la^p |\RR^{n+1} \setminus \elam| - \frac{1}{s} \iint_{8Q} |u-w|^2 \lsb{\chi}{[0,T]}\ dz. 
  \end{array}
 \end{equation}
 \item[Estimate of $J_3$:] Using \cref{imp_rmk}, we can take $\lim_{h \searrow 0}$ and thus get the following sequence of estimates:
 \begin{equation*}
 \hspace*{-2.3cm}\begin{array}{r@{}c@{}l}
  J_3 & \overset{\text{\cref{lemma_lihe_wang}}}{\leq} &   \int_0^{t_1} \left[\int_{16B\cap \elam^t} \iprod{\vec{w}}{\nabla (\eta v)} \ \lsb{\chi}{[0,T]} \ dx \right]_h\ dt + \int_0^{t_1} \left[\int_{16B\setminus \elam^t} |\vec{w}||\nabla(\eta \vl)| \lsb{\chi}{[0,T]} \ dx\right]_h\ dt \\
  & \overset{\text{\cref{corollary3.20-ini}}}{\apprle} &\int_0^{t_1} \left[\int_{16B\cap \elam^t} \iprod{\vec{w}}{\nabla (\eta v)} \ \lsb{\chi}{[0,T]} \ dx \right]_h\ dt  + \la^p |\RR^{n+1} \setminus \elam| + \frac{1}{s} \iint_{8Q} |[u-w]_h|^2 \lsb{\chi}{[0,T]} \ dz \\
  & \overset{\lim_{h\searrow 0}}{=} & \int_0^{t_1}\int_{16B\cap \elam^t} \iprod{\vec{w}}{\nabla (\eta^2 (u-w))} \ \lsb{\chi}{[0,T]} \ dz  + \la^p |\RR^{n+1} \setminus \elam| + \frac{1}{s} \iint_{8Q} |u-w|^2 \lsb{\chi}{[0,T]} \ dz.
 \end{array}
 \end{equation*}

\end{description}

\item[Estimate of $L_2$:] 
   We decompose the expression as 
   \begin{equation}\label{4.9-ini}
    \begin{array}{ll}
     L_2 & = \left[ \int_0^{t_1} \int_{\elam^{t}} + \int_0^{t_1} \int_{16B \setminus \elam^t} \right] \iprod{[\aa(x,t,\nabla u)]_h}{\nabla (\eta \vlh)} \lsb{\chi}{[0,T]}\ dz  \\
     & := L_2^1 + L_2^2.
    \end{array}
   \end{equation}

   \begin{description}
    \item[Estimate of $L_2^2$:] Using the chain rule, \cref{abounded}, \cref{def_eta} along with \cref{corollary3.20-ini}, we get
    \begin{equation}
    \label{4.10-ini}
     \begin{array}{rcl}
      L^2_2 
      & \leq &\int_{0}^{t_1} \int_{\Om_{8\rho} \setminus E^{\tau}_\la} {[|\nabla u|^{p-1}+|h_0|^{p-1}]_h}{|\nabla (\eta \vlh)|} \ dz \\
      & \apprle &  \sum_{k=0}^1  \rho^{k-1}\iint_{(8B\times 2\tm) \setminus \elam} {[|\nabla u|^{p-1}+|h_0|^{p-1}]_h}{|\nabla^k  \vlh|} \lsb{\chi}{[0,T]} \ dz  \\
&\apprle &  \la^p |\RR^{n+1} \setminus \elam| + \frac{1}{s} \iint_{8Q} |[u-w]_h|^2 \lsb{\chi}{[0,T]} \ dz \\
& \overset{\lim_{h \searrow 0}}{=} & \la^p |\RR^{n+1} \setminus \elam| + \frac{1}{s} \iint_{8Q} |u-w|^2 \lsb{\chi}{[0,T]} \ dz.
     \end{array}
    \end{equation}
\end{description} 
In the above estimate, we made use of the bound $|\vh| \leq |[u-w]_h|$ which follows from \cref{def-v-ini}.
\end{description}

Noting that  $\lim_{h \searrow 0}L_2^1 =\int_0^{t_1} \int_{E_{\tau}(\la)} \iprod{\aa(y,\tau,\nabla u)}{\nabla (\eta \vl)}\lsb{\chi}{[0,T]} \ dz$, we combine \cref{4.10-ini}, \cref{4.9-ini}, \cref{4.13-ini}, \cref{est_J_1-ini}, \cref{6.39-ini} and  \cref{6.38-ini}, followed by making use of \cref{cac1-ini}, we get
 \begin{equation}
 \label{6.45-ini}
  \begin{array}{l}
\int_{\elam^t} |v(x,t)|^2 \ dx  +  \int_0^{t_1} \int_{E^{\tau}_\la} \iprod{\aa(y,\tau,\nabla u)}{\nabla (\eta \vl)}\lsb{\chi}{[0,T]} \ dz  \\
   \hspace*{1cm} \apprle \la^p |\RR^{n+1} \setminus \elam| + \frac{1}{s} \iint_{8Q} |u-w|^2 \lsb{\chi}{[0,T]}\ dz + \int_0^{t_1}\int_{16B\cap \elam^t} \iprod{\vec{w}}{\nabla (\eta^2 (u-w))} \ \lsb{\chi}{[0,T]} \ dz.
  \end{array}
 \end{equation}
Multiplying \cref{6.45-ini} by  $\la^{-1-\be}$ and integrating from $(c_e\al_0,\infty)$ (recall that  $c_e$ is as in \cref{alpha_0}), for almost every $t \in (t_0+ s,t_0+4s)$ (actually holds for any $t \in (0,t_0+4s)$), we get 
 \begin{equation}
 \label{K_expression-ini}
K_1 + K_2 \apprle K_3 + K_4 + K_5,
 \end{equation}
 where we have set
 \begin{equation*}
  \begin{array}{@{}r@{}c@{}l@{}}
   K_1 \ &:=&\  \frac12\int_{c_e \al_0}^{\infty} \la^{-1-\be} \int_{\elam^t} | v(y,t)|^2 \ dy \ d\la, \\
  K_2 \ &:=& \ \int_{c_e \al_0}^{\infty} \la^{-1-\be}\int_{0}^t \int_{\elam^{\tau}} \iprod{\aa(y,\tau,\nabla u)}{\nabla (\eta^2 (u-w))} \lsb{\chi}{[0,T]}\ dy \ d\tau \ d\la, \\
  K_3\  &:=& \ \int_{c_e \al_0}^{\infty} \la^{-1-\be}  \la^p |\RR^{n+1} \setminus \elam| \  d\la, \\
  K_4 \ &:=& \ \frac{1}{s} \int_{c_e \al_0}^{\infty} \la^{-1-\be}   \iint_{8Q} |u-w|^2(y,\tau) \lsb{\chi}{[0,T]} \ dy  \ d\tau \ d\la, \\
  K_5 \ & :=& \ \int_{c_e \al_0}^{\infty} \la^{-1-\be}\int_{0}^t \int_{\elam^{\tau}} \iprod{\vec{w}}{\nabla (\eta^2 (u-w))} \lsb{\chi}{[0,T]}\ dy \ d\tau \ d\la. \\
  \end{array}
 \end{equation*}
 We now define the truncated Maximal function $\mathcal{M}(z) := \max \{ g(z), \al_0\}$ and then estimate each of the $K_i$ for $i \in \{1,2,3,4\}$ as follows:
 \begin{description}
  \item[Estimate of $K_1$:]  By applying Fubini, we get
  \begin{equation}
  \label{4.19-ini}
    K_1 \apprge \frac{1}{\be c_e^{\be}}  \int_{8B} \mathcal{M}(y,t)^{-\be} | v(y,t)|^2 \ dy .
  \end{equation}

  \item[Estimate of $K_2$:] Again applying Fubini, we get
  \begin{equation*}
  \label{4.20-ini}
   \begin{array}{ll}
    K_2 & = \frac{1}{\be c_e^{\be}} \int_{0}^t \int_{8B} \mathcal{M}(y,\tau)^{-\be} \iprod{\aa(y,\tau,\nabla u)}{\nabla (\eta^2 (u-w))} \ dy \ d\tau. 
   \end{array}
  \end{equation*}
  Applying chain rule  along with \cref{abounded}, \cref{bound_b} and the fact that $t \geq t_0+s$ which implies $Q \subset 8B \times (-\infty,t]$, we  get
  \begin{equation}
\label{4.25-ini}
 \begin{array}{ll}
\beta C_e^\beta  K_2 &= \int_{0}^t \int_{8B} \mathcal{M}(y,\tau)^{-\be} \iprod{\aa(y,\tau,\nabla u)}{\nabla u} \eta^2  \ dy \ d\tau   \\
  & \qquad + \int_{0}^t \int_{8B} \mathcal{M}(y,\tau)^{-\be} \iprod{\aa(y,\tau,\nabla u)}{\nabla \eta^2 } (u-w) \ dy \ d\tau  \\
  & \qquad - \int_{0}^t \int_{8B} \mathcal{M}(y,\tau)^{-\be} \iprod{\aa(y,\tau,\nabla u)}{\nabla w} \eta^2  \ dy \ d\tau   \\
  & \apprge  \iint_{Q} \mathcal{M}(y,\tau)^{-\be} |\nabla u|^p \eta^2  \lsb{\chi}{[0,T]} \ dy \ d\tau  - \iint_{8Q} \mathcal{M}(y,\tau)^{-\be} |h_0|^p   \lsb{\chi}{[0,T]} \ dy \ d\tau \\
  & \qquad - \iint_{8Q} \mathcal{M}(y,\tau)^{-\be} \lbr |\nabla u|^{p-1} + |h_0|^{p-1} \rbr \frac{|u-w|}{\rho} \lsb{\chi}{[0,T]} \ dy \ d\tau  \\
  & \qquad - \iint_{8Q}\mathcal{M}(y,\tau)^{-\be} \lbr |\nabla u|^{p-1} + |h_0|^{p-1} \rbr |\nabla w| \lsb{\chi}{[0,T]} \ dy \ d\tau  \\
  & := \aa_1 + \aa_2 + \aa_3 + \aa_4.
 \end{array}
\end{equation}

\begin{description}
 \item[Estimate of $\aa_1$:] Note that $\eta \equiv 1$ on $B$.  Let $S :=\{ z \in Q \cap {\{t\geq 0\}}: |\nabla u(z)| \geq \be g(z)\},$ then we get
 \begin{equation}
 \label{4.26-ini}
  \begin{array}{@{}l@{}c@{}l@{}}
   \iint_Q |\nabla u|^{p-\be}  \lsb{\chi}{[0,T]}\ dz  & =& \iint_S |\nabla u|^{p-\be} \ dz + \iint_{Q\setminus S} |\nabla u|^{p-\be}  \lsb{\chi}{[0,T]} \ dz \\
   & \leq &\be^{-\be} \iint_Q \mathcal{M}(z)^{-\be} |\nabla u|^p  \lsb{\chi}{[0,T]}\ dz + \be^{p-\be} \iint_{Q\setminus S} \mathcal{M}(z)^{p-\be}  \lsb{\chi}{[0,T]} \ dz \\
            & \overset{\text{\cref{bound_g_x_t-ini}}}{\apprle} & \iint_Q \mathcal{M}(z)^{-\be} |\nabla u|^p \lsb{\chi}{[0,T]} \ dz 
             + \be^{p-\be} |Q| \al_0^{p-\be}\\
&\apprle& \iint_Q\lvert\nabla u\rvert^{p-\beta}\lsb{\chi}{[0,T]}\ dz+\beta^{p-\beta}\lvert Q\rvert \alpha_0^{p-\beta}.
      \end{array}
 \end{equation}
  \item[Estimate of $\aa_2$:]  From \cref{def_g_ini}, we see that  $\lsb{\chi}{8Q \cap \{t \geq 0\}}( |\nabla u(z)| + |h_0(z)|)  \leq \mathcal{M}(z)$ for a.e $z \in \RR^n$, which gives
  \begin{equation}
   \label{4.26.1-ini}
 \aa_2 =  \iint_{8Q} \mathcal{M}(z)^{-\be} |h_0|^p  \lsb{\chi}{[0,T]} \ dz \apprle  \iint_{8Q}  |h_0|^{p-\be}  \lsb{\chi}{[0,T]} \ dz.
  \end{equation}

 \item[Estimate of $\aa_3$:]  We use the bound  $\lsb{\chi}{8Q}( |\nabla u(z)| + |h_0(z)|)  \leq \mathcal{M}(z)$ for a.e $z \in \RR^n$, along with Young's inequality and  \cref{alpha_0}, to get
 \begin{equation}
 \label{4.27-ini}
   \aa_3  \apprle \iint_{8Q}  (|\nabla u|+ |h_0|) ^{p-1-\be} \frac{|u-w|}{\rho} \lsb{\chi}{[0,T]}\ dz
   \apprle  \varepsilon |Q| \al_0^{p-\be} + C(\varepsilon) \iint_{8Q} \left| \frac{u-w}{\rho}\right|^{p-\be} \lsb{\chi}{[0,T]}\ dz.
 \end{equation}

 \item[Estimate of $\aa_4$:] Similar to the calculations for $\aa_3$, we get
 \begin{equation*}
 \label{bnd_aa_4}
 \aa_4 \apprle \ve |Q|\al_0^{p-\be} + C(\ve) \iint_{8Q} |\nabla w|^{p-\be} \lsb{\chi}{[0,T]} \ dz.
 \end{equation*}

\end{description}

  \item[Estimate of $K_3$:] Applying the layer-cake representation (see for example \cite[Chapter 1]{Grafakos}), we get
  \begin{equation}
  \label{4.21-ini}
   \begin{array}{rcl}
    K_3  & \le& \frac{1}{p-\be} \iint_{\RR^{n+1}} \mathcal{M}(z)^{p-\be} \ dz
     \overset{\text{\cref{bound_g_x_t-ini}}}{\apprle} |Q| \al_0^{p-\be}.
   \end{array}
  \end{equation}
  \item[Estimate of $K_4$:] Again applying Fubini, we get 
  \begin{equation}
  \label{4.22-ini}
   \begin{array}{ll}
    K_4 & =\frac{1}{s} \int_{{c_e \al_0}}^{\infty} \la^{-1-\be}   \iint_{8Q} |u-w|^2 \lsb{\chi}{[0,T]} \ dz \ d\la  = \frac{1}{\be} \iint_{8Q}  (\al_0)^{-\be}  \frac{|u-w|^2}{s} \lsb{\chi}{[0,T]} \ dz.
   \end{array}
  \end{equation}
\item[Estimate of $K_5$:] Applying Fubini, we get
\[
\begin{array}{rcl}
K_5 & =&  \frac{1}{\be c_e^{\be}} \int_0^t \int_{8B} \mm(y,\tau)^{-\be} \iprod{\vec{w}}{\nabla (\eta^2(u-w))} \lsb{\chi}{[0,T]} \ dz \\
& \apprle&  \frac{1}{\be c_e^{\be}}  \iint_{8Q} \mm(y,\tau)^{-\be} |\vec{w}| \lbr \frac{|u-w|}{\rho} + |\nabla (u-w)| \rbr  \lsb{\chi}{[0,T]} \ dz.
\end{array}
\]
From \cref{eq6.7}, we have $\mm(y,\tau) \geq \frac{|u-w|}{\rho}$ as well as $\mm(y,\tau) \geq |\nabla u|$ and $\mm(y,\tau) \geq |\nabla w|$  which combined with H\"older's inequality and \cref{alpha_0} gives
\begin{equation}
\label{est_K_5}
\begin{array}{rcl}
K_5 &\apprle & \frac{1}{\be c_e^{\be}}  \iint_{8Q}  |\vec{w}| \lbr \abs{\frac{u-w}{\rho}}^{1-\be} + |\nabla u|^{1-\be} + |\nabla w|^{1-\be} \rbr  \lsb{\chi}{[0,T]} \ dz \\
& \apprle & \frac{c(\ve)}{\be c_e^{\be}}  \iint_{8Q}  |\vec{w}|^{\frac{p-\be}{p-1}} \lsb{\chi}{[0,T]}\ dz +  \frac{\ve}{\be c_e^{\be}} \iint_{8Q}  \abs{\frac{u-w}{\rho}}^{p-\be}  \lsb{\chi}{[0,T]} \ dz + \frac{\ve}{\be c_e^{\be}} \al_0^{p-\be}|Q|.
\end{array}
\end{equation}

 \end{description}
%
%
Substituting \cref{4.26-ini}, \cref{4.26.1-ini} and \cref{4.27-ini} into \cref{4.25-ini} followed by combining  \cref{4.19-ini}, \cref{4.21-ini}, \cref{4.22-ini} and \cref{est_K_5} into \cref{K_expression-ini}, we get
\begin{equation*}
\label{4.28-ini}
 \begin{array}{ll}
& \frac{1}{2\be}  \int_{B} \mathcal{M}(y,t)^{-\be} | u-w|^2(y,t) \ dy + \frac{1}{\be} \iint_Q |\nabla u|^{p-\be}\lsb{\chi}{[0,T]} \ dz \\
&\apprle \frac{1}{\be} \be^{p-\be} \iint_{8Q} |\nabla u|^{p-\be} \lsb{\chi}{[0,T]}\ dz +\frac{1}{\be}\iint_{8Q} |h_0|^{p-\be}\lsb{\chi}{[0,T]} \ dz\\
    &\hspace*{5mm}+ \frac{1}{\be} \varepsilon |Q| \al_0^{p-\be} + \frac{1}{\be} C(\varepsilon) \iint_{8Q} \left| \frac{u-w}{\rho}\right|^{p-\be}\lsb{\chi}{[0,T]}\ dz+  \al_0^{p-\be} |Q| + \frac{1}{\be} \iint_{8Q} \al_0^{-\be}  \frac{|u-w|^2}{s}\lsb{\chi}{[0,T]} \ dz \\
    &\hspace*{5mm} + \frac{\ve}{\be c_e^{\be}} \iint_{8Q} \left| \frac{u-w}{\rho}\right|^{p-\be}\lsb{\chi}{[0,T]}\ dz + \frac{C(\ve)}{\be c_e^{\be}}  \iint_{8Q}  |\vec{w}|^{\frac{p-\be}{p-1}} \lsb{\chi}{\qomt}\ dz + \frac{C(\ve)}{\be} \iint_{8Q} |\nabla w|^{p-\be} \lsb{\chi}{[0,T]} \ dz.
 \end{array}
\end{equation*}
Multiplying the above expression  by $\be$ followed by choosing $\be \in (0,\be_0)$  and $\varepsilon \in (0,1)$ small and then using the intrinsic scaling $s = \rho^2 \al_0^{2-p}$  along with \cref{alpha_0}, we get 
\begin{equation*}
 \begin{array}{ll}
        \int_{B} \mathcal{M}(y,t)^{-\be} | u-w|^2(y,t)| \ dy + |Q| \al_0^{p-\be}  \apprle \iint_{8Q} |h_0|^{p-\be}\lsb{\chi}{[0,T]} \ dz \\
    \hspace*{2cm}  +  \iint_{8Q} \lbr \frac{|u-w|}{\rho} \rbr^{p-\be} \lsb{\chi}{[0,T]}\ dz  +  \iint_{8Q} \al_0^{p-2-\be}   \lbr \frac{|u-w|}{\rho}\rbr^2 \lsb{\chi}{[0,T]}\ dz\\
    \hspace*{2cm}  +  \iint_{8Q}|\nabla w|^{p-\be} \lsb{\chi}{[0,T]}\ dz  +  \iint_{8Q}   |\vec{w}|^{\frac{p-\be}{p-1}} \lsb{\chi}{\qomt}\ dz.\\
 \end{array}
\end{equation*}

Rearranging the above expression and dividing throughout by $|Q|$, we get
\begin{equation*}
\label{4.31}
 \begin{array}{ll}
        &\sup_{t \in \tm \cap \{t \geq 0\}} \al_0^{p-2} \hint_{B} \mathcal{M}(y,t)^{-\be} \left| \frac{ u-w}{\rho}\right|^2(y,t) \ dy \ + \  \al_0^{p-\be}\vspace{1em}\\
&   \apprle    \fiint_{8Q}\left[ \al_0^{p-2-\be}   \lbr \frac{|u-w|}{\rho}\rbr^2 +  \lbr \frac{|u-w|}{\rho} \rbr^{p-\be} \right]\lsb{\chi}{[0,T]}\ dz\\
        &\hspace*{2cm} + \fiint_{8Q} |h_0|^{p-\be}\lsb{\chi}{[0,T]} \ dz +  \fiint_{8Q}|\nabla w|^{p-\be} \lsb{\chi}{[0,T]}\ dz  +  \fiint_{8Q}   |\vec{w}|^{\frac{p-\be}{p-1}} \lsb{\chi}{\qomt}\ dz.
 \end{array}
\end{equation*} 
This completes the proof of the Lemma. 
\end{proof}

\subsection{Some consequences of Caccioppoli inequality}

\begin{lemma}
\label{lemma7.4-ini}
Let $\ka \geq 1$, then there exists $\be_0(n,p,q,\lamot,\ka)$ such that for any   $\be \in (0,\be_0)$ and  any very weak solution $u \in L^2(0,T;L^2(\Om)) \cap L^{p-\be}(0,T; W_0^{1,p-\be}(\Om))$ of \cref{main-2}, the following holds: Let $Q_{\rho,s}(x_0,t_0) = B_{\rho}(x_0) \times \tm_s(t_0)$ be the parabolic cylinder with $t_0 - s \leq 0 < t_0+s$ and $s = \rho^2 \al_0^{2-p}$ for some $\al_0 >0$ as in \cref{alpha_0}. Let $\al Q$ be a rescaled parabolic cylinder for some $\al \in (1,8]$ and also suppose that	
 \begin{equation}\label{hypothesis_2-ini}
\fiint_{\al Q} \lbr|\nabla u|+|h_0|\rbr^{p-\be} \lsb{\chi}{\al Q \cap \Om_T} \ dz + \fiint_{\al Q} |\nabla w|^{p-\be} \lsb{\chi}{\al Q \cap \Om_T} \ dz + \frac{1}{\lvert\alpha Q \rvert}\left\| \lvert\ddt{w}\rvert \lsb{\chi}{\alpha Q\cap\Om_T}\right\|_{L^{\frac{p-\be}{p-1}}(0,T; W^{-1.\frac{p-\be}{p-1}}(\Om))}^{\frac{p-\be}{p-1}} \le\kappa \al_0^{p-\be}.
 \end{equation}
 Let us define
\begin{equation*}
 \label{def_J-ini}
 J:= \sup_{t\in \tm\cap [0,T]} \hint_{ B} \lbr \frac{|u-w|}{\rho}\rbr^2 \mathcal{M}(x,t)^{-\be} \ dx,
\end{equation*}
where $\mathcal{M}(z) := \max \{ g(z), \al_0\}$ is the same as in the proof of \cref{caccioppoli} but with $\qomt$ replaced by $\al Q \cap \Om_T$ in \cref{eq6.7}.

 For any $1 \leq \sigma \leq \max \{2,p-\be\}$, with $r = \frac{2(p-\be)}{p}$ and $\vt = \max \left\{ 1, \frac{n\sigma}{n+r}\right\}$,  there exists a universal positive constant $C = C(n,p,\sigma,\kappa)$ such that the following holds:
\begin{equation*}
\label{lemma4.3.1-ini}
 \fiint_{ Q} \lbr \frac{|u-w|}{\rho} \rbr^{\sigma}\lsb{\chi}{[0,T]} \ dz \leq C (\al_0^{\be} J)^{\frac{\sigma -\vt}{2}}  \lbr[[] \fiint_{ Q} \lbr\abs{\frac{u-w}{\rho}}^{\vt \frac{rp}{rp-\be(\sigma-\vt)}} +|\nabla u-\nabla w|^{\vt \frac{rp}{rp-\be(\sigma-\vt)}}\rbr \lsb{\chi}{[0,T]} \ dz\rbr[]]^{\frac{rp-\be(\sigma-\vt)}{rp}}  .
\end{equation*}
\end{lemma}

\begin{proof}
 In order to prove the lemma, we want to make use of \cref{gagliardo_nirenberg_lemma}. First, we note that the choice of $\sigma, \vt, r$ with $\delta = \frac{\vt}{\sigma}$ satisfies \cref{9.20-ini}. Applying \cref{gagliardo_nirenberg_lemma}, we get
 \begin{equation}
 \label{9.25-ini}
\fiint_{Q} \abs{\frac{u-w}{\rho}}^{\sigma} \lsb{\chi}{[0,T]}\ dz \leq C \fint_I \lbr \fint_{B_{\rho}} \abs{\frac{u-w}{\rho}}^{\vt} + |\nabla u-\nabla w|^{\vt} \ dx \rbr \lbr \fint_{B_{\rho}} \abs{\frac{u-w}{\rho}}^r \ dx \rbr^{\frac{\sigma-\vt}{r}}\lsb{\chi}{[0,T]} \ dt.
\end{equation}

With $g(z)$ as in the hypothesis, we can apply H\"older's inequality following by taking the supremum over $t \in I\cap[0,T]$, to get
\begin{equation}
\label{9.26-ini}
\begin{array}{rcl}
 \hint_B  \lbr \frac{|u-w|}{\rho} \rbr^{r} (x,t) \ dx  & = & \hint_B  \lbr \abs{\frac{u-w}{\rho}}^2(x,t)\ \mathcal{M}(x,t)^{-\be} \rbr^{\frac{r}{2}} \mathcal{M}(x,t)^{\frac{\be r}{2}}\ dx  \\
 & \leq & J^{\frac{r}{2}} \lbr \hint_B \mathcal{M}(x,t)^{p-\be} \ dx \rbr^{\frac{2-r}{2}}.
\end{array}
\end{equation}

If $\be_0$ is chosen sufficiently small, then for any $\be \in (0,\be_0]$, we can get $\frac{rp}{\be(\sigma-\vt)} >0$, which along with \cref{9.25-ini} and \cref{9.26-ini} gives
\begin{equation*}
\begin{array}{rcl}
\fiint_{Q} \abs{\frac{u-w}{\rho}}^{\sigma} \lsb{\chi}{[0,T]}\ dz & \apprle &  J^{\frac{\sigma -\vt}{2}} \fint_I \lbr \fint_{B_{\rho}} \abs{\frac{u-w}{\rho}}^{\vt} + |\nabla u-\nabla w|^{\vt} \ dx \rbr \lbr \fint_B \mathcal{M}(z)^{p-\be} \ dx \rbr^{\frac{\be(\sigma-\vt)}{rp}} \lsb{\chi}{[0,T]}\ dt\\
& \apprle & J^{\frac{\sigma -\vt}{2}} \lbr \fiint_{Q} \abs{\frac{u-w}{\rho}}^{\vt \frac{rp}{rp-\beta (\sigma-\vt)}} + |\nabla u-\nabla w|^{\vt\frac{rp}{rp-\beta (\sigma-\vt)}} \lsb{\chi}{[0,T]}\ dz \rbr^{\frac{rp-\beta (\sigma-\vt)}{rp}} \\
&& \times \lbr \fiint_Q \mathcal{M}(z)^{p-\be}  \ dz \rbr^{\frac{\be(\sigma-\vt)}{rp}}.
\end{array}
\end{equation*}
Making use of  \cref{hypothesis_2-ini}, we can follow the proof of \cref{bound_g_x_t-ini} to obtain the bound
\begin{equation*}
\fiint_Q \mathcal{M}(z)^{p-\be} \ dz \apprle \al_0^{p-\be}.
\end{equation*}
Using the identity $\frac{\be (\sigma -\vt)(p-\be)}{rp} = \frac{\be(\sigma-\vt)}{2}$, we get the desired estimate.
 \end{proof}

 \begin{lemma}
\label{lemma9.5-ini}
Let $\frac{2n}{n+2}< p < 2+\be$, then under the assumptions of \cref{caccioppoli}, there holds
\begin{equation*}
\fiint_Q |u-w|^2 \lsb{\chi}{[0,T]}\ dz \apprle_{(n,p,\Lambda_1,\Lambda_0,b_0,r_0)} \rho^2 \al_0^2.
\end{equation*}
\end{lemma}
\begin{proof}
Let us choose $1 \leq \al_1 < \al_2 \leq 16$. Making use of  \cref{lemma7.4-ini} with $\sigma=2$ gives
\begin{equation}
\label{9.30-ini}
\fiint_{\al_1 Q} \abs{\frac{u-w}{\rho}}^2 \lsb{\chi}{[0,T]}\ dz \apprle (\al_0^{\be} J)^{\frac{2-\vt}{2}} \lbr[[] \fiint_{\al_1 Q} \lbr \abs{\frac{u-w}{\rho}}^{\vt \frac{rp}{rp-\be(2-\vt)}} +|\nabla u-\nabla w|^{\vt \frac{rp}{rp-\be(2-\vt)}} \rbr \lsb{\chi}{[0,T]}dz\rbr[]]^{\frac{rp-\be(2-\vt)}{rp}},
\end{equation}
where $r = \frac{2(p-\be)}{p}$, $\vt = \max\left\{1,\frac{2n}{n+r}\right\}$ and 
\begin{equation*}
J = \sup_{t \in \al_1 I \cap [0,T]} \fint_{\al_1 B} \abs{\frac{u-w}{\rho}}^2 \mm(\cdot, t)^{-\be} \ dx.
\end{equation*}
From a calculation similar to  \cref{caccioppoli} applied over $\al_1 Q$ and $\al_2 Q$ for $1\leq \al_1 < \al_2 \leq 16$ and corresponding cut-off functions 
\[
\eta \in C_c^{\infty}(\al_2 B)\ \ \ \text{with} \ \eta \equiv 1 \ \text{on}\  \al_1B \txt{and} \zeta\in C_c^{\infty}(\al_2I)\ \ \   \text{with} \ \zeta \equiv 1 \ \text{on}\  \al_1I,
\]
along with an application of Young's inequality and \cref{alpha_0} (note we have $p-\be <2$), we get
\begin{equation}
\label{9.32-ini}
\begin{array}{rcl}
\al_0^{\be} J & \apprle & \fiint_{\al_2 Q} \abs{\frac{u-w}{\al_2\rho-\al_1\rho}}^2 \lsb{\chi}{[0,T]}\ dz + \al_0^{2-p+\be} \fiint_{\al_2 Q} \abs{\frac{u-w}{\al_2\rho-\al_1\rho}}^{p-\be}  \lsb{\chi}{[0,T]} \ dz \\
&& + \al_0^{2-p+\be}\lbr[[]\fiint_{\al_2 Q} |h_0|^{p-\be}\lsb{\chi}{[0,T]} \ dz +  \fiint_{\al_2 Q}|\nabla w|^{p-\be} \lsb{\chi}{[0,T]}\ dz  +  \fiint_{\al_2 Q}   |\vec{w}|^{\frac{p-\be}{p-1}} \lsb{\chi}{\al_2 Q \cap \Om_T}\ dz\rbr[]]\\
& \apprle & \fiint_{\al_2 Q} \abs{\frac{u-w}{\al_2\rho-\al_1\rho}}^2  \lsb{\chi}{[0,T]} \ dz + \al_0^2.
\end{array}
\end{equation}

Combining \cref{9.30-ini} and \cref{9.32-ini}, along with applying Young's inequality, we get
\begin{equation}
\label{9.33-ini}
\begin{array}{rcl}
\fiint_{\al_1 Q} \abs{{u-w}}^2 \lsb{\chi}{[0,T]}\ dz & \leq &  C \rho^{\vt} \al_0^{\vt} \lbr \fiint_{\al_2Q} \abs{\frac{u-w}{(\al_2-\al_1)\rho}}^2  \lsb{\chi}{[0,T]} \ dz + \al_0^2 \rbr^{\frac{2-\vt}{2}}\\
& \leq & \frac12  \fiint_{\al_2 Q} |u-w|^2  \lsb{\chi}{[0,T]}\ dz  + C (\al_2-\al_1)^{-2 ( \frac{2}{\vt} - 1)} \rho^2 \al_0^2.
\end{array}
\end{equation}
We can now use \cref{iteration_lemma} to absorb the first term on the right of \cref{9.33-ini} which proves the lemma.
\end{proof}

\subsection{Reverse H\"older inequality}
\begin{lemma}
\label{lemma10.1-ini}
Suppose that \cref{alpha_0} holds over $\{Q, 16^2Q\}$ instead of $\{Q,16Q\}$
where $Q = Q_{\rho, \al_0^{2-p} \rho^2}(x_0,t_0)$, then there holds 
\begin{equation*}
\al_0^{p-\be} \apprle \lbr \fiint_{16Q} |\nabla u|^{q_0} \lsb{\chi}{[0,T]} \ dz\rbr^{\frac{p-\be}{q_0}} + \fiint_{16Q} |\Xi|^{p-\be} \lsb{\chi}{[0,T]} \ dz,
\end{equation*}
where 
\begin{equation}
\label{q-ini}
q_0 := \left\{
\begin{array}{lll}
\max\{ q, \overline{q} \}, & \overline{q} = \frac{np(p-\be)}{p(n+2)- \be (2+p-\be)} & \text{if} \ p-\be \geq 2\\
\max\{ q, \overline{q} \}, & \overline{q} = \frac{2np}{p(n+2)- 4\be} & \text{if} \ \frac{2n}{n+2}<p-\be <2,
\end{array}\right.
\end{equation}
and 
\begin{equation}
\label{def_xi}
\Xi := |h_0| + |\nabla w| + |\vec{w}|^{\frac{1}{p-1}}.
\end{equation}

\end{lemma}
\begin{proof}
Since \cref{alpha_0} is  satisfied over $\{Q,16^2Q\}$, but with a different universal constant, we can apply the Caccioppoli inequality from \cref{caccioppoli} to get
\begin{equation}
\label{10.4-ini}
\begin{array}{rcl}
\al_0^{p-\be} & \apprle &  \al_0^{p-2-\be} \fiint_{16Q} \abs{\frac{u-w}{\rho}}^{2}\lsb{\chi}{[0,T]} \ dz + \fiint_{16Q} \abs{\frac{u-w}{\rho}}^{p-\be}\lsb{\chi}{[0,T]} \ dz  + \fiint_{16Q} |\Xi|^{p-\be} \lsb{\chi}{\qomt} \ dz\\
& = & C_{cac} \lbr I_2 + I_{p-\be} + \fiint_{16Q} |\Xi|^{p-\be} \lsb{\chi}{\qomt} \ dz\rbr,
\end{array}
\end{equation}
where we have set $I_{\sigma} := \al_0^{p-\be-\sigma} \fiint_{16Q} \abs{\frac{u-w}{\rho}}^{\sigma} \lsb{\chi}{[0,T]} \ dz$ for $\sigma = 2$ or $\sigma = p-\be$ and $\Xi$ is from \cref{def_xi}.  Thus, we can apply \cref{lemma7.4-ini} to get
\begin{equation}
\label{691-ini}
\begin{array}{rcl}
I_{\sigma} & = & \al_0^{p-\be-\sigma} \fiint_{16Q} \abs{\frac{u-w}{\rho}}^{\sigma} \lsb{\chi}{[0,T]} \ dz \\
& \apprle & \al_0^{p-\be-\sigma} \lbr\al_0^{\be} J\rbr^{\frac{\sigma-\vt}{2}}  \lbr[[] \fiint_{16Q} \lbr \abs{\frac{u-w}{\rho}}^{\vt \frac{rp}{rp-\be(\sigma-\vt)}} +|\nabla u-\nabla w|^{\vt \frac{rp}{rp-\be(\sigma-\vt)}} \rbr \lsb{\chi}{[0,T]} \ dz\rbr[]]^{\frac{rp-\be(\sigma-\vt)}{rp}},
\end{array}
\end{equation}
where $r = \frac{2(p-\be)}{p}$, $\vt = \max\{1, \frac{n\sigma}{n+r} \}$ and 
\begin{equation*}
J := \sup_{t \in 16I \cap [0,T]} \fint_{16B} \abs{\frac{u-w}{\rho}}^2 \mm(\cdot,t)^{-\be} \ dz.
\end{equation*}
Again, we apply \cref{caccioppoli} to estimate $J$ to get
\begin{equation}
\label{10.7-ini}
\al_0^{\be} J \leq C_{cac} \lbr \fiint_{16^2Q} \abs{\frac{u-w}{\rho}}^2 \lsb{\chi}{[0,T]} \ dz + \al_0^{2-p+\be}\fiint_{16^2Q} \abs{\frac{u-w}{\rho}}^{p-\be} \lsb{\chi}{[0,T]} \ dz  + \al_0^{2-p+\be}\fiint_{16^2Q} |\Xi|^{p-\be} \lsb{\chi}{[0,T]} \ dz \rbr.
\end{equation}

To estimate the first term on the right of \cref{10.7-ini}, we split into two cases:
\begin{description}
\item[In the case $p-\be \leq 2$,]  we directly apply \cref{lemma9.5-ini} to get
\begin{equation}
\label{10.8-ini}
\fiint_{16^2Q} \abs{\frac{u-w}{\rho}}^2 \lsb{\chi}{[0,T]} \ dz \apprle \al_0^2.
\end{equation}
\item[In the case $p-\be > 2$,] we get the following sequence of estimates. Firstly, since $u-w=0$ for $\{t \leq 0\}$, which gives $\avgs{u-w}{\xi}=0$ where $\xi$ is defined analogous to \cref{lemma_crucial_1} but on $16^2Q \cap \{t\leq 0\}$. Thus applying H\"older's inequality, we get
\begin{equation}
\label{10.9-ini}
\begin{array}{rcl}
\fiint_{16^2Q} \abs{\frac{u-w}{\rho}}^2 \lsb{\chi}{[0,T]} \ dz 
& {\leq} & \lbr \fiint_{16^2Q} \abs{\frac{(u-w) - \avgs{u-w}{\xi}}{\rho}}^{p-\be} \lsb{\chi}{[0,T]} \ dz\rbr^{\frac{2}{p-\be}} =: H^{\frac{2}{p-\be}}.
\end{array}
\end{equation}

To estimate the term $H$ occurring on the right hand side of \cref{10.9-ini}, we proceed as follows:
\begin{equation}
\label{10.10-ini}
\begin{array}{rcl}
H  & \overset{\redlabel{10.9.b}{a}}{\apprle} &  \fiint_{16^2 Q} |\nabla u-\nabla w|^{p-\be} \lsb{\chi}{[0,T]} \ dz + \sup_{t_1,t_2 \in 16^2I \cap \{t\geq0\}} \abs{\frac{\avgs{u-w}{\mu}(t_2) -\avgs{u-w}{\mu}(t_1)}{\rho}}^{p-\be} \\
& \overset{\redlabel{10.9.c}{b}}{\apprle} &   \fiint_{16^2 Q} |\nabla u-\nabla w|^{p-\be} \lsb{\chi}{[0,T]} \ dz + \abs{\frac{ \|\nabla \mu\|_{L^{\infty}(16^2B)}}{\rho} \iint_{16^2Q} (|\nabla u|+|h_0|)^{p-1} \lsb{\chi}{[0,T]} \ dz}^{p-\be} \\
&& + \abs{\frac{ \|\nabla \mu\|_{L^{\infty}(16^2B)}}{\rho} \iint_{16^2Q} |\vec{w}| \lsb{\chi}{16^2Q \cap \Om_T} \ dz}^{p-\be} \\
& \overset{\redlabel{10.9.d}{c}}{\apprle} &   \fiint_{16^2 Q} |\nabla u-\nabla w|^{p-\be} \lsb{\chi}{[0,T]} \ dz + \lbr \al_0^{2-p} \fiint_{16^2Q} (|\nabla u|+|h_0|)^{p-1} \lsb{\chi}{[0,T]} \ dz\rbr^{p-\be} \\
&& + \lbr \al_0^{2-p} \fiint_{16^2Q} |\vec{w}| \lsb{\chi}{16^2Q \cap \Om_T} \ dz\rbr^{p-\be} \\
& \overset{\redlabel{10.9.e}{d}}{\apprle} &   \al_0^{p-\be} + \lbr \al_0^{2-p} \al_0^{p-1}\rbr^{p-\be}  = \al_0^{p-\be}.
\end{array}
\end{equation}
 To obtain \redlabel{10.9.b}{a}, we apply \cref{lemma_crucial_1} with $\mu\in C_c^{\infty}(16^2B)$ (which is just a rescaled version of \cref{def_eta}). To obtain \redlabel{10.9.c}{b}, we apply \cref{lemma_crucial_2_app} with $\phi = \mu$ and $\varphi = 1$.  To obtain \redlabel{10.9.d}{c}, we note that $|16^2 Q| \equiv \rho^{n+2} \al_0^{2-p}$ and finally to obtain \redlabel{10.9.e}{d}, we make use of the hypothesis of the lemma (more specifically \cref{alpha_0}).
\end{description}
Thus combining \cref{10.8-ini}, \cref{10.9-ini}, \cref{10.10-ini} with \cref{10.7-ini}, we get
\begin{equation}
\label{697-ini}
\al_0^{\be} J \apprle \al_0^2.
\end{equation}

From the choice $\vt = \max\left\{1, \frac{n\sigma}{n+r} \right\}$ along with $r = \frac{2(p-\be)}{p}$, \cref{q-ini} and $\sigma = \max\{2,p-\be\}$, we see that 
\begin{equation}
\label{698-ini}
\vt \frac{rp}{rp-\be(\sigma -\vt)} \leq q_0,
\end{equation}
provided $\be\ll 1$ is sufficiently small(see \cite{bogelein2009very} 192p for the details.)

Making use of \cref{697-ini} in \cref{691-ini} and the observation \cref{698-ini} along with Young's inequality, for any $\ve >0$, we get 
\begin{equation}
\label{10.12-ini}
I_{\sigma} \apprle \ve \al_0^{p-\be} + C_{\ve} \lbr[[] \fiint_{16Q} \lbr  \abs{\frac{u-w}{\rho}}^{q_0} + |\nabla u-\nabla w|^{q_0} \rbr\lsb{\chi}{[0,T]} \ dz \rbr[]]^{\frac{p-\be}{q_0}}.
\end{equation}

To estimate the second term on the right of \cref{10.12-ini}, we can proceed similarly to \cref{10.9-ini} and \cref{10.10-ini} to get
\begin{equation}
\label{10.13-ini}
\begin{array}{l}
\fiint_{16Q} \lbr \abs{\frac{u-w}{\rho}}^{q_0} + |\nabla u-\nabla w|^{q_0}\rbr \lsb{\chi}{[0,T]} \ dz \apprle \fiint_{16Q} |\nabla u-\nabla w|^{q_0} \lsb{\chi}{[0,T]}\ dz  \\
\hfill + \lbr \al_0^{2-p} \fiint_{16Q} (|\nabla u|+\lvert h_0|)^{p-1} \lsb{\chi}{[0,T]} \ dz \rbr^{q_0} + \lbr \al_0^{2-p} \fiint_{16Q} |\vec{w}| \lsb{\chi}{\qomt} \ dz \rbr^{q_0}. 
\end{array}
\end{equation}

To estimate second term of right hand side, we use H\"older's inequality and \cref{alpha_0} to discover
\begin{equation}
\label{10.14-ini}
\begin{array}{rcl}
\fiint_{16Q} |\nabla u|^{p-1} \lsb{\chi}{[0,T]} \ dz 
& \leq & \lbr \fiint_{16Q} |\nabla u|^{p-\be} \lsb{\chi}{[0,T]} \ dz \rbr^{\frac{p-2}{p-\be}} \lbr \fiint_{16Q} |\nabla u|^{q_0} \lsb{\chi}{[0,T]} \ dz\rbr^{\frac{1}{q_0}} \\
& \apprle & \al_0^{p-2} \lbr \fiint_{16Q} |\nabla u|^{q_0} \lsb{\chi}{[0,T]} \ dz\rbr^{\frac{1}{q_0}}.
\end{array}
\end{equation}

Combining \cref{10.14-ini} into \cref{10.13-ini} and making use of \cref{10.12-ini}, we get
\begin{equation}
\label{10.15-ini}
I_{\sigma} \leq CC_{cac}\ve \al_0^{p-\be}  + C_{\ve} \lbr \fiint_{16Q} |\nabla u|^{q_0} \lsb{\chi}{[0,T]} \ dz\rbr^{\frac{p-\be}{q_0}} + C_{\ve} \fiint_{16Q} |\Xi|^{p-\beta} \lsb{\chi}{\qomt} \ dz.
\end{equation}

Combining \cref{10.15-ini} and \cref{10.4-ini}, we get
\begin{equation*}
\al_0^{p-\be} \leq C C_{cac} \ve \al_0^{p-\be} + C_{\ve} \lbr \fiint_{16Q} |\nabla u|^{q_0} \lsb{\chi}{[0,T]} \ dz\rbr^{\frac{p-\be}{q_0}} + \fiint_{16Q} |\Xi|^{p-\be} \lsb{\chi}{\qomt} \ dz.
\end{equation*}
Choosing $\ve$ small, the lemma follows.
\end{proof}

\subsection{Higher integrability at initial boundary - Proof of \texorpdfstring{\cref{main_theorem_2}}.}

We are now ready to prove \cref{main_theorem_2}. The calculations follows very similar to \cite[Theorem 2.1]{bogelein2009very} with a few modifications. For the sake of completeness, we provide the rough sketch below.

Without loss of generality, we can assume $\rho = 1$ and $z_0 = (0,0) \in \pa \Om \times [0,T]$. We take $Q_1 := Q_{(\rho,s)}(0,0) $ and  $Q_2 := Q_{(2\rho,2s^2)}(0,0)$ and for any  $z \in Q_2$, define the parabolic distance of $z$ to $\pa Q_2$ by 
\begin{gather*}
d_p(z):= \inf_{ \tilde{z} \in \RR^{n+1} \setminus Q_2} \min \{ |x-\tilde{x}| , \sqrt{|t-\tilde{t}|} \} \label{par_dist_boundary-ini}. 
\end{gather*}

Furthermore, let $\be_0$ be the constant such that \cref{lemma10.1-ini} holds for any $\be \in (0,\be_0)$ and $p-\be > \frac{2n}{n+2}$. For $z \in Q_2$, let us define the following function:
\begin{gather}
  \psi(z) := \lbr |\nabla u(z)| + |\Xi(z)|\rbr\lsb{\chi}{[0,T]} \quad \text{and} \quad f(z) := d_p^{\al} (z) \psi(z) \quad \text{with} \ \al:= \frac{n+2}{d}, \label{def_g_f} 
\end{gather}
where $d$ is as defined in \cref{def_d-ini} and $\Xi$ is defined in \cref{def_xi}. 
Finally  we define $\al_0$ to be
\begin{gather}
 \al_0^d := \fiint_{Q_2} \psi(z)^{p-\be} \ dz + 1. \label{def_al_0-ini}
\end{gather}
Let $\la_0$ be any number such that 
\begin{gather}
 \la_0 \geq b^{\frac{1}{d}} \al_0 \qquad \text{where} \ b:= 2^{10(n+2)}.\label{5.5-ini}
\end{gather}
Now suppose that $\mfz \in Q_2$ with $f(\mfz)>\lambda_0$, then let us denote the parabolic distance of $\mfz$ to $\pa Q_2$ by $r_{\mfz} := d_p(\mfz) $
and define the intrinsic scaling factor as
\begin{gather}
 \ga = \ga(\mfz) := (r_{\mfz}^{-\al} \la_0 )^{2-p} = (d_p(\mfz)^{-\al} \la_0)^{2-p} \label{intrinsic_scaling-inii}.
\end{gather}

In order to prove higher integrability, we want to apply \cref{gehring_lemma}. So the rest of the proof is devoted to ensuring that all the hypotheses of \cref{gehring_lemma} are satisfied.
\begin{description}[leftmargin=*]
 \item[Case $p \geq 2$:]  Let us note that $r_{\mfz}^{\al} \leq 2^{\al} \leq b^{\frac1{d}} \al_0 \leq \la_0$, which implies $\ga = (r_{\mfz}^{-\al} \la_0)^{2-p} \leq 1$. Hence we shall consider intrinsic cylinders of the type $Q_{\mfz}(R, \ga R^2)$ with $0 < R \leq r_{\mfz}$.

 In order to apply \cref{gehring_lemma}, we need to find an appropriate intrinsic parabolic cylinder around $\mfz$ on which all the hypotheses of \cref{lemma10.1-ini} are satisfied. 
 In order to do this, let us first take $R$ such that $r_{\mfz} \leq 2^9 R < 2^9 r_{\mfz}$. In this case, there holds:
 \begin{equation}
  \label{5.8-ini}
  \begin{array}{ll}
   \fiint_{Q_{\mfz}(R,\ga R^2)} \psi(z)^{p-\be} \ dz & \leq \frac{|Q_2|}{|{Q_{\mfz}(R,\ga R^2)}|} \fiint_{Q_{2}} \psi(z)^{p-\be} \ dz \\
   &\overset{\text{\cref{def_al_0-ini}}}{\le} \frac{2^{n+2}}{R^{n+2} \ga}\al_0^d
   \overset{\text{\cref{5.5-ini}}}{\leq} \frac{2^{10(n+2)}}{r_{\mfz}^{n+2}} \frac{\la_0^d}{b} 
    \overset{\text{\cref{intrinsic_scaling-inii}}}{=} (r_{\mfz}^{-\al} \la_0)^{p-\be}.
  \end{array}
 \end{equation}

 Furthermore, by Lebesgue differentiation theorem, for every $\mfz \in Q_2$ with $f(\mfz) > \la_0$, there holds
 \begin{equation}
  \label{5.9-ini}
  \lim_{r\searrow 0} \fiint_{Q_{\mfz}(r,\ga r^2)} \psi(z)^{p-\be} \ dz = \psi(\mfz)^{p-\be} \overset{\text{\cref{def_g_f}\cref{intrinsic_scaling-inii}}}{=} (r_{\mfz}^{-\al} f(\mfz))^{p-\be} >  (r_{\mfz}^{-\al} \la_0)^{p-\be} .
 \end{equation}

 Thus from \cref{5.8-ini} and \cref{5.9-ini}, we observe that there should exist  $\rho \in \lbr0,\frac{r_{\mfz}}{2^9}\rbr$, such that 
 \begin{equation*}
  \label{def_rho-ini}
  \begin{array}{c}
  \fiint_{Q_{\mfz}(\rho, \ga \rho^2)} \psi(z)^{p-\be} \ dz = (r_{\mfz}^{-\al} \la_0)^{p-\be}, \\
  \fiint_{Q_{\mfz}(R, \ga R^2)} \psi(z)^{p-\be} \ dz \leq (r_{\mfz}^{-\al} \la_0)^{p-\be},  \quad \ \forall\   R \in [\rho, r_{\mfz}].
  \end{array}
 \end{equation*}

  We now set $Q:= Q_{\mfz}(\rho, \ga \rho^2)$, then $2^9Q \subset Q_2$, thus all the hypotheses of \cref{lemma10.1-ini} are satisfied with $(r_{\mfz}^{-\al} \la_0, 1)$ instead of $(\al_0,\kappa)$, i.e., the following holds:
  \begin{equation}
  \label{5.11-ini}
   (r_{\mfz}^{-\al} \la_0)^{p-\be}  = \fiint_Q \psi(z)^{p-\be} \ dz \qquad \text{and} \qquad \fiint_{2^8Q} \psi(z)^{p-\be} \ dz \apprle (r_{\mfz}^{-\al} \la_0)^{p-\be}.
  \end{equation}

In the case $Q \cap \{t\leq 0\} \neq \emptyset$,  we can apply  \cref{lemma10.1-ini} and in the case $2^8Q \subset \Om \times [0,T]$, we are in the interior case and can apply \cite[Lemma 6.3]{bogelein2009very} with $\th$ replaced by $\Xi$ since $|\th| \leq |\Xi|$ (which was first proved in \cite{KL}) to get
\begin{equation}
 \label{5.12-ini}
 (r_{\mfz}^{-\al} \la_0)^{p-\be} \apprle \lbr \fiint_{2^8Q} |\nabla u|^q\lsb{\chi}{[0,T]} \ dz \rbr^{\frac{p-\be}{q}} + \fiint_{2^8Q} |\Xi|^{p-\be}\lsb{\chi}{[0,T]} \ dz.
\end{equation}

 Since $2^9 \rho \leq r_{\mfz}$ and $\ga \leq 1$, we also have for all $z \in 2^8Q$ that 
 \begin{gather}
  d_p(z) \leq \min\{ r_{\mfz} + 2^8 \rho, \sqrt{r_{\mfz}^2 + \ga (2^8 \rho)^2} \} \leq \frac32 r_{\mfz}, \label{up_bound_dpz-ini} \\
  d_p(z) \geq \min\{ r_{\mfz} - 2^8 \rho, \sqrt{r_{\mfz}^2 - \ga (2^8 \rho)^2} \} \geq \frac12 r_{\mfz}. \label{low_bound_dpz-ini}
 \end{gather}

 Now substituting \cref{up_bound_dpz-ini} and \cref{low_bound_dpz-ini} into \cref{def_g_f}, we find
 \begin{equation}
  \label{5.15-ini}
  c^{-1} f(z) \leq r_{\mfz}^{\al} \psi(z) \leq c f(z), \quad \forall \ z \in 2^8Q \ \text{ with  } \ c = c(n)>1.
 \end{equation}

 We now claim the following estimate holds:
 \begin{equation}
  \label{5.16-ini}
  \la_0^{p-\be} \overset{\text{\redlabel{5.16.a}{a}}}{\apprle} \fiint_{2^8Q} f(z)^{p-\be} \ dz \overset{\text{\redlabel{5.16.b}{b}}}{\apprle} \lbr \fiint_{2^8Q} f(z)^q \ dz \rbr^{\frac{p-\be}{q}} + \fiint_{2^8Q} (r_{\mfz}^{\al} |\Xi(z)|)^{p-\be} \ dz \overset{\text{\redlabel{5.16.c}{c}}}{\apprle} \la_0^{p-\be}.
 \end{equation}
 \begin{description}[leftmargin=*]
  \item[Estimate \redref{5.16.a}{a}:] This follows easily by making use of \cref{5.11-ini} and \cref{5.15-ini} and subsequently enlarging the parabolic cylinder $Q$. 
  \item[Estimate \redref{5.16.b}{b}:] This is obtained by the following chain of estimates:
  \begin{equation*}
   \begin{array}{lll}
    \fiint_{2^8Q} f(z)^{p-\be} \ dz & \overset{\text{\cref{5.15-ini}}}{\apprle} & r_{\mfz}^{\al(p-\be)} \fiint_{2^8Q} \psi(z)^{p-\be} \ dz \overset{\text{\cref{5.11-ini}}}{\apprle} \la_0^{p-\be} \\
    & \overset{\text{\cref{5.12-ini}}}{\apprle} &  r_{\mfz}^{\al(p-\be)} \lbr \fiint_{2^8Q} |\nabla u|^q \lsb{\chi}{[0,T]} \ dz \rbr^{\frac{p-\be}{q}}  + r_{\mfz}^{\al(p-\be)} \fiint_{2^8Q} |\Xi(z)|^{p-\be}\lsb{\chi}{[0,T]} \ dz \\
    & \overset{\text{\cref{5.15-ini}}}{\apprle} & \lbr \fiint_{2^8Q} f(z)^q \ dz \rbr^{\frac{p-\be}{q}} + \fiint_{2^8Q} (r_{\mfz}^{\al} |\Xi(z)|)^{p-\be}\lsb{\chi}{[0,T]} \ dz.
   \end{array}
  \end{equation*}

  \item[Estimate \redref{5.16.c}{c}:] This follows by applying Jensen's inequality (since $q < p-\be$) along with the bound \cref{5.11-ini} to get:
  \begin{equation*}
   \lbr \fiint_{2^8Q} f(z)^q \ dz \rbr^{\frac{p-\be}{q}} + \fiint_{2^8Q} (r_{\mfz}^{\al} |\Xi(z)|)^{p-\be} \lsb{\chi}{[0,T]}\ dz \overset{\text{\cref{5.15-ini}}}{\apprle}  \fiint_{2^8Q} f(z)^{p-\be} \ dz \overset{\text{\cref{5.11-ini}}}{\apprle} \la_0^{p-\be}.
  \end{equation*}
 \end{description}

 Thus  \cref{5.16-ini} holds and  as a consequence, we can apply \cref{gehring_lemma} over $2^8Q$ to see that  for any $\be \in (0,\be_0]$, there exists  $\de_0 = \de_0(n,p,b_0,r_0,\varepsilon_0,\lamot,\beta)>0$ such that $f \in L_{loc}^{p-\be+\de_1}(Q_2)$ with $\de_1 = \min\{ \de_0, \tilde{p}-p+\be\}$. This is quantified by the estimate:
\begin{equation*}
 \label{5.19-ini}
 \iint_{Q_2} f(z)^{p-\be+\de} \ dz \apprle \al_0^{\de} \iint_{Q_2} f(z)^{p-\be} \ dz + \iint_{Q_2} (r_{\mfz}^{\al} |\Xi(z)|)^{p-\be+\de} \lsb{\chi}{[0,T]}\ dz\qquad \forall \de \in (0,\de_1].
\end{equation*}

By iterating the previous arguments, for any $\be \in (0,\varepsilon_{\text{geh}}]$ where $\varepsilon_{\text{geh}}>0$ is the gain in higher integrability coming from \cref{gehring_lemma}, we obtain the bound
\begin{equation}
 \label{5.20-ini}
 \iint_{Q_2} f(z)^{p} \ dz \apprle \al_0^{\be} \iint_{Q_2} f(z)^{p-\be} \ dz + \fiint_{Q_2} |\Xi(z)|^p\lsb{\chi}{[0,T]} \ dz.
\end{equation}

For any $z \in Q_1$, we have $d_p(z) \geq \min \{ 1 , \sqrt{3}\} \geq 1$, $\frac{|Q_2|}{|Q_1|} = C(n)$, which implies the following bounds  hold:
\begin{gather}
 |\nabla u (z)| \leq \psi(z) \leq f(z) \qquad \forall\  z \in Q_1 \cap \RR^n \times [0,T], \label{5.21-ini}\\
 \fiint_{Q_1} |\nabla u|^{p+\be} \lsb{\chi}{[0,T]}\ dz \apprle_n \fiint_{Q_2} f(z)^p \ dz ,\label{5.22-ini}\\
 f(z) \leq 2^{\al} \psi(z) \qquad \forall z \in Q_2,  \quad \text{since} \ d_p(z) \leq 2.\label{5.23-ini}
\end{gather}

Using \cref{5.21-ini}, \cref{5.22-ini}, \cref{5.23-ini} along with \cref{5.20-ini} and making use of \cref{def_g_f} and \cref{def_al_0-ini}, we get
\begin{equation*}
\begin{array}{rcl}
 \fiint_{Q_1} |\nabla u|^p \lsb{\chi}{[0,T]}\ dz & \apprle & \al_0^{\be} \fiint_{Q_2} \psi(z)^{p-\be} \ dz  + \fiint_{Q_2} \Xi^p \lsb{\chi}{[0,T]} \ dz \\
 & \apprle &  \lbr \fiint_{Q_2} \lbr |\nabla u| + |\Xi| \rbr^{p-\be}\lsb{\chi}{[0,T]} \ dz \rbr^{1+\frac{\be}{d}} + \fiint_{Q_2} (1+\Xi^p)\lsb{\chi}{[0,T]} \ dz.
 \end{array}
\end{equation*}
Substituting the expression for $\Xi$ from \cref{def_xi}, we get
\begin{equation*}
\begin{array}{rcl}
 \fiint_{Q_1} |\nabla u|^p \lsb{\chi}{[0,T]}\ dz & \apprle &  \lbr \fiint_{Q_2} \lbr |\nabla u| + |h_0| \rbr^{p-\be}\lsb{\chi}{[0,T]} \ dz \rbr^{1+\frac{\be}{d}} + \fiint_{Q_2} (1+h_0^p)\lsb{\chi}{[0,T]} \ dz \\
 && + \lbr \fiint_{Q_2} |\nabla w|^{p-\be}  \lsb{\chi}{[0,T]} \ dz \rbr^{1 + \frac{\be}{d}} + \lbr \fiint_{Q_2} |\vec{w}|^{\frac{p-\be}{p-1}} \lsb{\chi}{[0,T]}  \ dz \rbr^{1 + \frac{\be}{d}}\\
 && + \fiint_{Q_2} |h_0|^{p-\be} \lsb{\chi}{[0,T]}  \ dz  + \fiint_{Q_2} |\nabla w|^p \lsb{\chi}{[0,T]} \ dz+ \fiint_{Q_2} |\vec{w}|^{\frac{p}{p-1}} \lsb{\chi}{[0,T]}  \ dz. 
 \end{array}
\end{equation*}
This proves the asserted estimate.

 \item[Case $\frac{2n}{n+2} <p < 2$:] The basic change with respect to the case $p\geq 2$ is that, we now switch to the sub-quadratic scaling, i.e., we consider intrinsic cylinders of the type $Q_{\mfz}(\ga^{-\frac12}R,R^2)$. 
 
 The parameter $\al_0$ is still given by \cref{def_al_0-ini} and $\la_0$ is chosen as in \cref{5.5-ini} and $\ga$ is again given as in \cref{intrinsic_scaling-inii}, where $\mfz \in Q_2$ with $f(z) > \la_0$. But in contrast to $p\geq 2$ case, we have $\ga = (r_{\mfz}^{-\al} \la_0 )^{2-p} \geq 1$. Hence for $R \in (0,r_{\mfz})$, we have $Q_{\mfz}(\ga^{-\frac12}R,R^2)\subset Q_2$. Now once again, in order to apply \cref{gehring_lemma}, we need to find a suitable intrinsic parabolic cylinder around $\mfz$, which enables us to apply \cref{lemma10.1-ini} or \cite[Lemma 6.3]{bogelein2009very}.  We observe, from the definition \cref{def_g_f} that $n+2 = \al d$ (recall $d$ is defined as in \cref{def_d-ini}) and $(2-p)\frac{n}{2} + d = p-\be$, which gives
 \begin{equation*}
  \begin{array}{ll}
   \fiint_{Q_{\mfz}(\ga^{-\frac12}R, R^2)} \psi(z)^{p-\be} \ dz & \apprle \frac{|Q_2|}{|Q_{\mfz}(\ga^{-\frac12}R, R^2)|}\fiint_{Q_2} \psi(z)^{p-\be} \ dz  = \frac{2^{n+2}}{R^{n+2} \ga^{-\frac{n}{2}}} \al_0^d  
   \apprle (r_{\mfz}^{-\al} \la_0 )^{p-\be}. 
  \end{array}
 \end{equation*}
 Now we can continue as in the $p \geq 2$ case to obtain the desired conclusion. 
\end{description}
This completes the  proof of the theorem.\hfill \qed

\section{Proof of \texorpdfstring{\cref{main_existence}}.}
\label{section7}

Let us first use the approximation $u_0^k \in C^2(0,T; L^2(\Om)) \cap L^p(0,T;W^{1,p}(\Om))$ given in the hypothesises of \cref{main_existence} satisfying \cref{assumption1} and \cref{assumption2}. Subordinate to this sequence, there exists a unique weak solution $u^k \in C^0(0,T;L^2(\Om)) \cap L^p(0,T; u_0^k + W_0^{1,p}(\Om))$ solving 
\begin{equation}
 \label{main-existence-1}
\left\{ \begin{array}{ll}
        u^k_t - \dv \aa(x,t,\nabla u^k) = 0 & \txt{on} \Om \times (0,T),\\
        u^k = u^k_0 & \txt{on} \pa \Om \times (0,T), \\
        u^k = 0 & \txt{on} \Om \times \{t=0\}.
        \end{array}\right.
\end{equation}
in the sense of \cref{def_weak_solution}, i.e.,  the following holds for any $\varphi\in W^{1,2}(0,T;L^2(\Om))\cap L^{p}(0,T;W^{1,p}(\Om))$,
\begin{equation}
\label{12.5}
\int_{\Om} u^k \varphi (x,t) \ dx + \iint_{\Om \times (0,t)} \{ -u^k \varphi_t + \iprod{\aa(x,t,\nabla u^k)}{\nabla \varphi} \ dz = \int_{\Om} u^k_0(x,0) \varphi (x,0) \ dx = 0.
\end{equation}

In order to pass through the limit in \cref{12.5}, it would suffice to show the following convergence results:
\begin{itemize}
\item $u^k \rightarrow u$ in $C^0(0,T;L^2_{\loc}(\Om)).$
\item $\{u^k\}$ is precompact in $L^p(0,T;W^{1,p}_{\loc}(\Om))$.
\end{itemize}
Note that we only require $\{u^k\}$ is precompact in $L^p(0,T;W^{1,p}_{\loc}(\Om))$, which is away from the lateral boundary.

Let $i, j \in\NN$ be any two exponents, since $u^i$ and $u^j$ are regular weak solutions,  then from standard energy estimates obtained locally in space,   for any $4B \subset \Om$ and any non negative function $\varphi \in C_c^{\infty}(2B)$ with $\varphi \equiv 1$ on $B$, using $(u^i - u^j) \varphi^p$ as a test function in \cref{12.5}, we get the estimate
\begin{equation}
\label{12.7}
\begin{array}{l}
\hspace*{-1cm}\sup_{t \in [0,s]} \int_B (u^i - u^j)^2  \ dx + \iint_{B \times [0,s]} \iprod{\aa(x,t,\nabla u^i) - \aa(x,t,\nabla u^j)}{\nabla u^i - \nabla u^j}  \ dz  \\
\hspace*{2cm}{\apprle}  \iint_{2B \times [0,s]} \lbr[[] (|\nabla u^i|+|h_0|)^{p-1}+(|\nabla u^j|+|h_0|)^{p-1}\rbr[]] |u^i - u^j| |\nabla \varphi^p| \ dz.\\
\end{array}
\end{equation}
In the case $p \geq 2$, making use of  \cref{3.7-1}, we get
\[\begin{array}{rcl}
\iint_{B \times [0,s]} \lvert \nabla u^i-\nabla u^j\rvert^p\ dz & \apprle & \iint_{B \times [0,s]} (\lvert \nabla u^i\rvert^2+\lvert \nabla u^j\rvert^2)^\frac{p-2}{2}\lvert \nabla u^i-\nabla u^j\rvert\ dz\\
&\apprle& \iint_{B \times [0,s]} \iprod{\aa(x,t,\nabla u^i) - \aa(x,t,\nabla u^j)}{\nabla u^i - \nabla u^j}  \ dz,\\
\end{array}
\]
and for $p\leq 2$, we have
\[
\begin{array}{ll}
&\hspace*{-1cm}\iint_{B\times[0,s]}\lvert \nabla u^i-\nabla u^j\rvert^p\ dz\vspace{1em}\\
&=\iint_{B\times[0,s]}(\lvert \nabla u^i\rvert^2+\lvert \nabla u^j\rvert^2)^{\frac{p(2-p)}{4}}(\lvert \nabla u^i\rvert^2+\lvert \nabla u^j\rvert^2)^\frac{p(p-2)}{4}\lvert \nabla u^i-\nabla u^j\rvert^p\ dz\vspace{1em}\\
&\apprle \ep\iint_{B\times[0,s]}\lvert \nabla u^i\rvert^p+\lvert \nabla u^j\rvert^p\ dz+C(\ep)\iint_{B\times[0,s]}(\lvert \nabla u^i\rvert+\lvert \nabla u^j\rvert)^\frac{p-2}{2}\lvert \nabla u^i-\nabla u^j\rvert^2\ dz\vspace{1em}\\
&\apprle\ep\iint_{B\times[0,s]} \lvert\nabla u^i\rvert^p+\lvert \nabla u^j\rvert^p\ dz+C(\ep)\iint_{B \times [0,s]} \iprod{\aa(x,t,\nabla u^i) - \aa(x,t,\nabla u^j)}{\nabla u^i - \nabla u^j}  \ dz.

\end{array}
\]
 Using \cref{abounded} and H\"older's inequality in \cref{12.7}, we get
\begin{equation}
\label{eq7.4}
\begin{array}{l}
\sup_{t \in [0,s]} \int_B (u^i - u^j)^2  \ dx+\iint_{B \times [0,s]} |\nabla u^i - \nabla u^j|^p  \ dz  \apprle \ep\iint_{B\times[0,s]}\lvert \nabla u^i\rvert^p+\lvert \nabla u^j\rvert^p\ dz\\
\hspace*{0.8cm}+C(\ep)\|\nabla \varphi^p\|_{L^{\infty}(2B)}\lbr \iint_{2B \times [0,s]} |\nabla u^i|^{p}+|\nabla u^j|^{p} + |h_0|^p\ dz \rbr^{\frac{p-1}{p}} \lbr \iint_{2B \times [0,s]} |u^i - u^j|^p \ dz \rbr^{\frac{1}{p}}.
\end{array}
\end{equation}
In particular, the standard energy estimate takes the form
\begin{equation*}
\begin{array}{l}
\sup_{t \in [0,s]} \int_B (u^i)^2  \ dx + \iint_{B \times [0,s]} |\nabla u^i|^p  \ dz  
\apprle \|\nabla \varphi^p\|_{L^{\infty}(2B)}\lbr \iint_{2B \times [0,s]} |\nabla u^i|^{p} + |h_0|^p\ dz \rbr^{\frac{p-1}{p}} \lbr \iint_{2B \times [0,s]} |u^i|^p \ dz \rbr^{\frac{1}{p}}.
\end{array}
\end{equation*}

Using the interior higher integrability from \cite{KL} and  the initial boundary higher integrability from \cref{main_corollary_2}, we get
\begin{equation}
\label{eq7.5}
\fiint_{2B \times [0,s]} |\nabla u^i|^p \ dz \apprle \lbr \fiint_{4B \times [0,4s]} ( |\nabla u^i| + |h_0|)^{p-\be} \ dz \rbr^{1+\frac{\be}{d}} + \fiint_{4B \times [0,4s]} (1 + |h_0|^p) \ dz.
\end{equation}
Now we make use of  \cref{corollary_main_theorem_1} along with  the hypothesises \cref{assumption1} and \cref{assumption2}, we can bound the right hand side of \cref{eq7.5} uniformly by a term depending only on $u_0$ and $h_0$, in particular, there holds
\begin{equation}
\label{eq7.6}
\begin{array}{rcl}
\hspace*{-0.5cm}\fiint_{2B \times [0,s]} |\nabla u^i|^{p} \ dz & \apprle &  \lbr \frac{1}{s|B|}\rbr^{\frac{\be}{d}}  \lbr \iint_{\Om_T} \lbr |\nabla u_0|+|h_0|\rbr^{p-\be}\ dz + \left\|\ddt{u_0} \right\|_{L^{\pbp}(0,T; W^{-1,\pbp}(\Om))}^{\pbp}\rbr^{1+\frac{\be}{d}} \\
&& + \fiint_{4B \times [0,4s]} (1 + |h_0|^p) \ dz \\
& =: &  \mathbf{R}.
\end{array}
\end{equation}
Denoting the term appearing on the right hand side of \cref{eq7.6} by $\mathbf{R}$, we use the above estimate in  \cref{eq7.4} to get
\begin{equation}
\label{eq7.7}
\begin{array}{ll}
&\sup_{t \in [0,s]} \int_B (u^i - u^j)^2  \ dx + \iint_{B \times [0,s]} |\nabla u^i - \nabla u^j|^p  \ dz   \\
&\hspace*{3cm}\apprle \ep\mathbf{R}+ C(\epsilon)\lbr[[] s \abs{B}  \mathbf{R}\rbr[]]^{\frac{p-1}{p}}\lbr \iint_{2B \times [0,s]} |u^i - u^j|^p \ dz \rbr^{\frac{1}{p}}.

\end{array}
\end{equation}

We shall now show how to obtain the necessary convergence results that allows us to pass through the limit in \cref{12.5}. To do this, we follow the structure from \cite[Proof of Theorem 2]{zhou} (see also \cite{bocmur} for the details).

\begin{description}
\steplabel{Step 1:}{tep1} First, we want to obtain a limit for $u^j$. In order to do this, we use the higher integrability for very weak solutions from \cref{eq7.6} to see that $\nabla u^i$ is uniformly bounded in $L^p(0,T; L^{p}_{\loc}(\Om))$.  Also from \cref{main_theorem_1}, we see that $u^i-u^i_0$ is uniformly bounded in $L^{p-\beta}(0,T;W_0^{1,p-\beta})$ so that $u^i$ is bounded in $L^p(0,T;W_{\loc}^{1,p}(\Om))$. Thus there exists a subsequence (still denoted by $\{u^i\}$) such that 
\begin{gather*}
u^i \rightharpoonup u\ \  \text{  weakly in } L^p(0,T; W^{1,p}_{\loc}(\Om)).
\end{gather*}

\steplabel{Step 2:}{tep2}  From \cref{main-existence-1} and \stepref{tep1}{Step 1},  we see that $u^i_t $ is uniformly bounded in $L^\frac{p}{p-1}(0,T;W^{-1,\frac{p}{p-1}}_{\loc}(\Om))$. Applying Lions-Aubin Lemma (see \cite[Proposition 1.3 on page 106]{showalter2013monotone}), there exists a subsequence such that
\begin{equation*}
\begin{cases}
u^i_t\rightharpoonup u_t&\text{ weakly in }L^p(0,T;W^{-1,\frac{p}{p-1}}_{\loc}(\Om)),\\
u^i \rightarrow u\ \  &\text{  strongly in } L^p(0,T; L^{p}_{\loc}(\Om)).
\end{cases}
\end{equation*}

\steplabel{Step 3:}{tep3} As a direct consequence of \stepref{tep2}{Step 2} and \cref{eq7.7}, we take  $\ep$ sufficiently small followed by  taking $i$ and $j$ large enough to obtain
\begin{equation*}
\nabla u^i \rightarrow \nabla u \ \  \text{  a.e in } L^p(0,T;W_{\loc}^{1,p}(B)) \txt{for any} B \subset \Om.
\end{equation*}

\steplabel{Step 4}{tep4} From \cref{abounded}, we see that $\aa(x,t,\nabla u^i)$ is bounded in $L^{\frac{p}{p-1}}(0,T;W^{-1,\frac{p}{p-1}}_{\loc}(\Om))$. Thus, as a consequence of \stepref{tep3}{Step 3}, we find that
\[
\aa(x,t,\nabla u^i) \rightharpoonup \aa(x,t,\nabla u) \ \ \text{  weakly in } L^{\frac{p}{p-1}}(0,T; W_{\loc}^{-1,\frac{p}{p-1}}(\Om)).
\]

\end{description}

From the above convergence estimates, we can now take $\lim_{k \rightarrow \infty}$ in \cref{12.5} to obtain the existence of a  very weak solution 
$
u \in C^0(0,T; L_{\loc}^2(\Om)) \cap L^{p-\be}(0,T; u_0 + W_0^{1,p-\be}(\Om))
$
of \cref{main-2}. This completes the proof of the theorem.\hfill \qed

\section*{References}

\end{document}